\documentclass[a4paper,reqno]{amsart}
\usepackage{fullpage}
\usepackage{amsmath,amsthm,amssymb,amsfonts}
\usepackage{stmaryrd}
\usepackage{hyperref}
\usepackage[capitalize]{cleveref}
\usepackage[all]{xy}

\theoremstyle{plain}
\newtheorem{thm}{Theorem}[section]
\newtheorem{lem}[thm]{Lemma}
\newtheorem{cor}[thm]{Corollary}
\newtheorem{prop}[thm]{Proposition}

\theoremstyle{definition}

\newtheorem{defn}[thm]{Definition}
\newtheorem{condition}[thm]{Condition}

\theoremstyle{remark}
\newtheorem{rem}[thm]{Remark}
\newtheorem{eg}[thm]{Example}
\newtheorem*{prf}{Proof}
\newtheorem*{prf:thm:basechange}{Proof of \cref{thm:basechange}}
\numberwithin{equation}{thm}

\newcommand{\sq}{\hfill{\square}}
\crefname{thm}{Theorem}{Theorems}
\crefname{lem}{Lemma}{Lemmas}
\crefname{cor}{Corollary}{Corollaries}
\crefname{prop}{Proposition}{Propositions}
\crefname{notation}{Notation}{Notations}
\crefname{assumption}{Assumption}{Assumptions}
\crefname{defn}{Definition}{Definitions}
\crefname{condition}{Condition}{Conditions}
\crefname{rem}{Remark}{Remarks}
\crefname{eg}{Example}{Examples}
\crefname{equation}{}{}

\newcommand{\Z}{\mathbb{Z}}
\newcommand{\ZC}{\mathcal{Z}}
\newcommand{\R}{\mathbb{R}}
\newcommand{\A}{\mathcal{A}}
\newcommand{\B}{\mathcal{B}}
\newcommand{\C}{\mathbb{C}}
\newcommand{\CC}{\mathcal{C}}
\newcommand{\DD}{\mathsf{D}}
\newcommand{\E}{\mathcal{E}}
\newcommand{\F}{\mathcal{F}}
\newcommand{\G}{\mathcal{G}}
\newcommand{\HC}{\mathcal{H}}
\newcommand{\I}{\mathcal{I}}
\newcommand{\J}{\mathcal{J}}
\newcommand{\K}{\mathcal{K}}
\newcommand{\LC}{\mathcal{L}}
\newcommand{\M}{\mathcal{M}}
\newcommand{\OC}{\mathcal{O}}
\newcommand{\PC}{\mathcal{P}}

\newcommand{\U}{\mathcal{U}}
\newcommand{\V}{\mathcal{V}}
\newcommand{\UM}{\mathcal{UM}}
\newcommand{\m}{\mathsf{m}}
\newcommand{\n}{\mathsf{n}}
\newcommand{\HG}{\phi}
\newcommand{\GF}{\psi}
\newcommand{\FE}{\theta}
\newcommand{\EG}{\imath}
\newcommand{\FH}{\jmath}
\newcommand{\lhaar}{h}

\newcommand{\lrhaar}{h}

\newcommand{\bs}{\backslash}

\newcommand{\isom}{\cong}
\newcommand{\otop}{\mbox{\raisebox{0.55mm}{\small{$\bigcirc$}}\hspace{-1.9ex}\small{$\top$}}}

\renewcommand{\1}{\mathrm{1}\hspace{-0.25em}\mathrm{l}} 
\renewcommand{\hat}{\widehat}
\renewcommand{\tilde}{\widetilde}
\newcommand{\rpb}{{}^*\!}

\DeclareMathOperator{\id}{id}
\DeclareMathOperator{\Ker}{Ker}
\DeclareMathOperator{\Ran}{Im}
\DeclareMathOperator{\ran}{im}

\DeclareMathOperator{\Aut}{Aut}
\DeclareMathOperator{\Dom}{Dom}
\DeclareMathOperator{\Rep}{Rep}
\DeclareMathOperator{\Irr}{Irr}
\DeclareMathOperator{\KK}{KK}
\DeclareMathOperator{\Cor}{Corr}
\DeclareMathOperator{\Calg}{C*alg}
\DeclareMathOperator{\Ad}{Ad}
\DeclareMathOperator{\Ind}{Ind}
\newcommand{\WInd}{\tilde{\Ind}}
\DeclareMathOperator{\op}{op}
\DeclareMathOperator{\cop}{cop}

\providecommand{\tensor}[1]{\underset{#1}{\otimes}}
\providecommand{\outensor}[1]{\underset{#1}{\boxtimes}}
\providecommand{\cotensor}[1]{\underset{#1}{\square}}
\providecommand{\coequalizer}[2]{\Ker\left({#1}\,,\,{#2}\right)}
\providecommand{\redltimes}{\underset{r}{\ltimes}}
\newcommand{\barotimes}{\bar{\otimes}}

\providecommand{\norm}[1]{\left|\!\left|#1\right|\!\right|}
\newcommand{\bra}{\langle}
\newcommand{\ket}{\rangle}
\providecommand{\clin}[1]{\left[#1\right]}
\providecommand{\mult}[1]{\M(#1)} 

\newcommand{\rsurj}{\twoheadrightarrow}
\providecommand{\xrarr}[1]{\xrightarrow{#1}}

\begin{document}
	\title{Induced coactions along a homomorphism of locally compact quantum groups}
	\author{Kan Kitamura}
	\address{Graduate School of Mathematical Sciences, The University of Tokyo, 3-8-1 Komaba Meguro-ku Tokyo 153-8914, Japan}
	\email{kankita@ms.u-tokyo.ac.jp}
	\date{}
	
	\begin{abstract}
		We consider induced coactions on C*-algebras 
		along a homomorphism of locally compact quantum groups which need not give a closed quantum subgroup. 
		Our approach generalizes the induced coactions constructed by Vaes, 
		and also includes certain fixed point algebras. 
		We focus on the case when the homomorphism satisfies a quantum analogue of properness. 
		Induced coactions along such a homomorphism still admit the natural formulations of various properties 
		known in the case of a closed quantum subgroup, 
		such as imprimitivity and adjointness with restriction. 
		Also, we show a relationship of induced coactions and restriction which is analogous to base change formula for modules over algebras. 
		As an application, we give an example that shows several kinds of 
		1-categories of coactions with forgetful functors 
		cannot recover the original quantum group. 
	\end{abstract}
	
	\maketitle
	\setcounter{tocdepth}{1}
	\tableofcontents
	\section{Introduction}\label{sec:intro}
	
	Induced coactions of a locally compact group from its closed subgroup on C*-algebras 
	were studied by Rieffel~\cite{Rieffel:ind}, 
	with a motivation coming from Mackey's works on imprimitivity. 
	Induced coactions from compact subgroups play an essential role in the reformulation of the Baum--Connes conjecture by Meyer--Nest~\cite{Meyer-Nest:bctri}. 
	Vaes~\cite{Vaes:impr} has developed the theory of induction and imprimitivity for a locally compact quantum group and its closed quantum subgroup. 
	This induction procedure for quantum groups gives a fundamental tool to study the quantum analogue of Baum--Connes conjecture in 
	\cite{Voigt:bcfo}, 
	\cite{Voigt:cpxss} for example. 
	
	For a quantum group $G$ and its closed quantum subgroup $H$ with suitable conditions, 
	the quantum homogeneous space $G/H$ can be given by the induced $G$-coaction $\Ind_H^G \C$ from $\C$ with a trivial $H$-coaction. 
	In the classical case, a homogeneous space $G/H$ can be also constructed by taking the quotient space of $G$ with the right $H$-action, 
	or equivalently by taking the fixed point algebra of $C_0(G)$ with the right $H$-coaction. 
	When the quantum subgroup $H$ is compact, this method still works as shown in~\cite{Soltan:noncptaction}. 
	However, in the general non-compact case, we have technical subtleties to construct $C^r_0(G/H)$ as a fixed point algebra. 
	There is difficulty in finding a good notion of ``properness" for coactions on C*-algebras 
	to integrate sufficiently many elements of a C*-algebra by the Haar measure, which can be an infinite measure. 
	See~\cite{Rieffel:fixedptalg}, 
	\cite{Buss:fixedptalg} 
	for some approaches to construct fixed point algebras. 
	
	Another interpretation for induced coactions and fixed point algebras is to regard them as (the coordinate rings of) balanced products of non-commutative spaces with quantum group actions, 
	which should be realized as C*-algebraic cotensor products if exist. 
	For C*-algebras with coactions of a compact quantum group, such a construction has appeared in~\cite{deRijdt-vanderVennet:moneq},~\cite{Voigt:bcfo}. 
	Again when the quantum group is non-compact, it becomes much harder to find appropriate constructions for cotensor products of C*-algebras. 
	
	In this paper, we give a unified approach to induced coactions for a closed quantum subgroup and fixed point algebras. 
	We introduce a generalized concept of induced coactions along a general homomorphism of locally compact quantum groups, in a similar spirit as in~\cite{Kasprzak:qghmg}. 
	Next, we introduce a class of homomorphisms with a measure theoretic condition which behaves particularly nicely with our framework. 
	This class of homomorphisms has appeared in~\cite{Kasprzak-Khosravi-Soltan:integrable} in a different formulation, 
	and classically it corresponds to the class of continuous homomorphisms of locally compact groups 
	that are proper as continuous maps of topological spaces. 
	We would like to call such a homomorphism a proper homomorphism in this paper. 
	We construct induced coactions along a proper homomorphism. 
	We do this by decomposing the homomorphism into a closed quantum subgroup and the Pontryagin dual of an open quantum subgroup. 
	Moreover, we use this decomposition to investigate the properties of induced coactions. 
	Thus our work heavily depends on the techniques developed in~\cite{Vaes:impr}, where induced coactions along closed quantum subgroup were constructed, and those in~\cite{Kalantar-Kasprzak-Skalski:open}, where the concept of open quantum subgroups was investigated. 
	Also, we should mention that many motivating examples and ideas in this paper come from~\cite{Nest-Voigt:eqpd}. 
	
	For a locally compact quantum group $G$, we have a new quantum group called the quantum double of $G$, which contains $G$ as a closed quantum subgroup. 
	It is observed in~\cite{Nest-Voigt:eqpd} that the procedure of taking induced coactions from a closed quantum subgroup $H$ of $G$ 
	is compatible with the coactions of their quantum doubles $\DD(G)$ and $\DD(H)$. 
	More precisely, if the C*-algebra $A$ has an $H$-coaction which extends to a $\DD(H)$-coaction, then 
	this extension induces a $\DD(G)$-coaction on the induced $G$-coaction $\Ind_H^G A$. 
	In this paper, we show slightly stronger statement holds for more general settings. 
	This can be seen as a comodule analogue of base change formula for modules over algebras, 
	and it can be stated as follows if we could use cotensor products 
	(see \cref{thm:basechange} for the precise statement). 
	Consider a commutative diagram of locally compact quantum groups and homomorphisms, 
	\begin{align*}
		\xymatrix{
			F\ar[r]\ar[d] & E\,\ar[d] \\
			H\ar[r] & G, }
	\end{align*}
	such that $C^r_0(G)\isom C^r_0(E)\cotensor{F}C^r_0(H)$ canonically. 
	Then for a C*-algebra $A$ with an $H$-coaction, 
	it holds $C^r_0(G)\cotensor{H}A\isom C^r_0(E)\cotensor{F}A$, 
	if this holds when $A=\C$. 
	Although its conclusion looks quite intuitive, the verification involves several technically naive points. 
	Practically, we consider the situation when $E$ and $F$ are closed quantum subgroups of $G$ and $H$, respectively. 
	Then it will give a good tool for calculation which reduces the induction procedure to smaller quantum groups. 
	
	As already mentioned, 
	a quantum homogeneous space $G/H$ 
	has two (candidates of) interpretations. 
	One is an induced coaction of $\C$ along $H\to G$, and 
	the other is a fixed point algebra of $C^r_0(G)$ with the right $H$-coaction. 
	Our approach of induced coactions gives a way to compare them 
	when we formulate the latter interpretation 
	as an induced coaction of $C^r_0(G)$ along a trivial homomorphism $H\to 1$, where $1$ is the trivial group. 
	We give a criterion for their coincidence, and 
	we will also see this coincidence gives a sufficient condition for the case $A=\C$ of the analogue of base change formula above. 
	
	For monoidally equivalent two compact quantum groups, it is proved that various properties are preserved in 
	\cite{deRijdt-vanderVennet:moneq},~\cite{Voigt:bcfo} by using cotensor products with the biGalois object. 
	It can be hoped that the techniques of the generalized induced coactions 
	will give some tool to compare two quantum groups 
	which need not be monoidally equivalent but still have certain similarities. 
	As an application, we consider a certain problem of recovering a locally compact quantum group 
	from its category of a certain kind of $G$-comodules with the forgetful functor, in a similar spirit with Tannaka--Krein reconstruction theorem. 
	For compact quantum groups, a celebrated result was obtained by Woronowicz~\cite{Woronowicz:reconstruction}. 
	But if we consider possibly non-compact cases, the situation becomes significantly difficult, 
	although Tatsuuma managed to give a reconstruction theorem in~\cite{Tatuuma:reconstruction} for locally compact groups. 
	
	On the other hand, there is another approach taken by~\cite{Meyer-Roy-Woronowicz:qghom}, as an answer for a question by Debashish Goswami. 
	Instead of the category $\Rep(G)$ of unitary representations, they consider the category $\CC^G$ of $G$-equivariant C*-algebras with $G$-equivariant morphisms of C*-algebras 
	for a locally compact quantum group $G$ 
	(or more generally a C*-bialgebra associated with a manageable multiplicative unitary). 
	Then they show that there is a bijective correspondence of 
	a homomorphism $H\to G$ of locally compact quantum groups 
	and a functor $\CC^G\to \CC^H$ which commutes with the forgetful functors $\CC^G\to \CC^1 \leftarrow \CC^H$. 
	See~\cite[Theorem~6.1]{Meyer-Roy-Woronowicz:qghom} for detail. 
	We consider similar situations for the $1$-categories of coactions with two different kinds of morphisms. 
	One is the $1$-category of isomorphism classes of equivariant correspondences $\Cor^G$, 
	and the other is the equivariant Kasparov category $\KK^G$. 
	For $\CC=\Cor$ or $\KK$ and a certain pair $(G,H)$ of locally compact quantum groups, 
	we will construct a categorical equivalence $\CC^G\xrarr{\sim} \CC^H$ 
	that commutes with the forgetful functors 
	$\CC^G\to \CC^1 \leftarrow \CC^H$ up to a natural isomorphism in \cref{cor:fiberfunc}. 
	We also see even when $G$, $H$ are non-monoidally equivalent finite quantum groups, they can form such a pair. 
	Here we should note that our situation is slightly weaker than that considered in~\cite{Meyer-Roy-Woronowicz:qghom}, 
	where they impose that the commutativity with forgetful functors must be equality, not up to a natural isomorphism. 
	
	We explain the organization of this paper. 
	In \cref{sec:pre}, we recall the basic concepts about locally compact quantum groups defined by~\cite{Kustermans-Vaes:lcqg}. 
	We also fix the notation. 
	In \cref{sec:formal}, we define the formal generalization of induced coactions from closed quantum subgroups. 
	We prove few properties including induction in stages (\cref{thm:indstages}). 
	We also discuss the relationship between two interpretations of quantum homogeneous spaces in \cref{ssec:hmgquot}. 
	From \cref{sec:proper}, we focus on the case of a proper homomorphism $\HG:H\to G$ of locally compact quantum groups. 
	In \cref{ssec:properind}, we show the existence of induced coactions along $\HG$ when $G,H$ are regular, as a consequence of induction in stages. 
	As a supplementary topic, we give another description of induced coactions when $H=K$ is compact in \cref{ssec:cptind}. 
	This will justify the interpretation of induced coactions as cotensor products for the compact case. 
	In \cref{sec:basechange}, we prove the analogue of base change formula for induced coactions (\cref{thm:basechange}). 
	We also give some concrete situations of this theorem by using the terminologies from 
	\cref{sec:appendix} about double crossed products and Yetter--Drinfeld C*-algebras. 
	For completeness in \cref{sec:MK}, we show several properties of induced coactions along a proper homomorphism $\HG$ 
	that were previously known in the case of closed quantum subgroups by~\cite{Vaes:impr},~\cite{Nest-Voigt:eqpd}. 
	We emphasize the functorial aspects of induced coactions. 
	We consider a relationship of induced coactions and twisted tensor products (\cref{prop:projfml}), 
	imprimitivity for reduced crossed products (\cref{prop:impr}), 
	and adjointness of induction and restriction (\cref{prop:adjoint}). 
	Finally in \cref{sec:eg}, we use these techniques to compare categories of coactions of certain double crossed products with a discrete abelian group. 
	Especially, we show there is a pair of non-isomorphic finite quantum groups whose categories of coactions are equivalent in a compatible way with forgetful functors (\cref{cor:fiberfunc}). 
	\\
	{\bf Acknowledgements:} 
	This work contains the master's thesis of the author. 
	He appreciates Yasuyuki Kawahigashi, who is the supervisor of the author, for invaluable supports. 
	He is also grateful to Yuki Arano for helpful advices and discussions. 
	This work was supported by JSPS KAKENHI Grant Number JP21J21283, and the WINGS-FMSP program at the University of Tokyo. 
	
	\section{Preliminaries}\label{sec:pre}
	
	In this section, we recall and fix notation and terminologies about quantum groups and their coactions on C*-algebras. 
	We roughly follow~\cite{Takesaki:book} and~\cite{Lance:book} for general notation about operator algebras and Hilbert C*-modules. 
	
	Let $A,B$ be C*-algebras. 
	For a non-degenerate c.p.~map $f:A\to \mult{B}$, 
	we abuse notation and simply write $f:\mult{A}\to\mult{B}$ for 
	$\mult{f}:\mult{A}\to\mult{B}$ 
	which is the unique u.c.p.~map extending $f$ such that 
	it is strictly continuous on the unit ball of $\mult{A}$. 
	We also write $1_A$ for the unit of the multiplier C*-algebra $\mult{A}$. 
	Unless stated otherwise, $A\otimes B$ always means the spatial tensor product of $A$ and $B$. 
	We will use $\barotimes$ for von Neumann algebraic tensor products. 
	For C*-algebras $A,B$ and Hilbert spaces $\HC,\HC'$, we write 
	$\sigma$ and $\Sigma$ 
	for the flipping maps 
	$\sigma:A\otimes B\isom B\otimes A $ and $\Sigma:\HC\otimes \HC' \isom \HC' \otimes \HC$, respectively. 
	We also use $\sigma$ for the multipliers of tensor products of C*-algebras and the tensor products of von Neumann algebras. 
	We also use the so-called leg-numbering notation. 
	By using this we can write $\sigma=(-)_{21}$. 
	Often we ignore the leg of $\C$. 
	
	When $X$, $Y$ are sets and $f,g:X\to Y$ are maps, we write 
	$\Ker(f,g):=\{x\in X \, | \, f(x)=g(x) \}$. 
	When $Y$ is a subset of a Banach space $X$, 
	we write $\clin{Y}$ for the norm-closure of the linear span of $Y$ in $X$.

	\subsection{Locally compact quantum groups and multiplicative unitaries}\label{ssec:lcqg}
	
	Next, we recall and fix the notation about locally compact quantum groups and multiplicative unitaries by following 
	\cite{Kustermans-Vaes:lcqg},~\cite{Baaj-Skandalis:mltu},~\cite{Woronowicz:mltu}. 
	
	\begin{defn}
		Let $A$ be a C*-algebra and $\Delta : A\to \mult{A\otimes A}$ be an injective non-degenerate $*$-homomorphism. 
		\begin{enumerate}
			\item
			We say $(A,\Delta)$ is a \emph{C*-bialgebra} 
			if $(\Delta \otimes \id_A)\Delta = (\id_A\otimes \Delta)\Delta$. 
			\item
			A C*-bialgebra $(A,\Delta)$ is \emph{bisimplifiable} if 
			$\clin{\Delta(A)(1_A\otimes A)}=\clin{\Delta(A)(A\otimes 1_A)}=A\otimes A$. 
			\item 
			For C*-bialgebras $(A,\Delta_A), (B,\Delta_B)$ and a non-degenerate $*$-homomorphism $f:A\to\mult{B}$, 
			we say $f$ is a \emph{homomorphism of C*-bialgebras} if 
			$\Delta_B f = (f\otimes f)\Delta_A$. 
			\item
			When $M$ is a von Neumann algebra and 
			$\Delta:M\to M\barotimes M$ is an injective normal unital $*$-homomorphism, 
			$(M,\Delta)$ is a \emph{von Neumann bialgebra} 
			if $(\Delta\barotimes\id_M)\Delta=(\id_M\barotimes\Delta)\Delta$. 
			\end{enumerate}
	\end{defn}
	
	\begin{defn}[Locally compact quantum group]
		Let $G=(M,\Delta)$ be a von Neumann bialgebra, and $\lrhaar:M_+ \to[0,\infty]$ be a normal semi-finite faithful weight on $M$. 
		\begin{enumerate}
			\item
			The weight $\lrhaar$ is called a \emph{left} [resp.~\emph{right}] \emph{invariant} if 
			for all $x\in \lrhaar^{-1}(\R_{\geq 0})$ and for all $0\leq \omega\in M_*$ 
			it holds that 
			$\omega(1)\lrhaar(x)=\lrhaar((\omega\barotimes \id_M)\Delta(x))$
			[resp.~$\omega(1)\lrhaar(x)=\lrhaar((\id_M\barotimes \omega)\Delta(x))$]. 
			\item
			The von Neumann bialgebra $G=(M,\Delta)$ is called a 
			\emph{locally compact quantum group} if 
			there exist a normal semi-finite faithful left invariant weight and a normal semi-finite faithful right invariant weight. 
			Then we write $(L^{\infty}(G),\Delta_G):=(M,\Delta)$, 
			and its left [resp.~right] invariant weight is called a \emph{left} [resp.~\emph{right}] \emph{Haar weight}. 
			\end{enumerate}
	\end{defn}
	
	Let $G=(L^{\infty}(G),\Delta_G)$ be a locally compact quantum group. 
	The left or right Haar weight is unique up to a scalar multiplication. 
	Take the left Haar weight $\lhaar$ of $G$. 
	Via the GNS construction of $\lhaar$, we obtain a Hilbert space $L^2(G)$ with 
	$\Lambda:\{x\in L^{\infty}(G) \,|\, \lhaar(x^*x)<\infty\}\to L^2(G)$, and 
	a normal faithful $*$-representation $M\to \B(L^2(G))$. 
	Moreover we can obtain a multiplicative unitary $W^G \in \U(L^2(G)^{\otimes 2})$ in the sense that 
	$W^G_{23}W^G_{12}=W^G_{12}W^G_{13}W^G_{23}$, 
	which is characterized by $W^G{}^*(\Lambda(x)\otimes \Lambda(y))=(\Lambda\otimes \Lambda)\Delta_G(y)(x\otimes 1)$ 
	for any $x,y\in \Dom(\Lambda)$, 
	where $\Lambda\otimes \Lambda$ means 
	the map given by the GNS construction of $h\otimes h$. 
	It holds $\Delta_G=\Ad W^G{}^*(-)_2$. 
	We shall write $\K(G)=\K(L^2(G))$ 
	for short, and often identify $\K(G)^*$ as $\B(L^2(G))_*$. 
	
	We write $C^r_0(G):=\clin{(\id_{\K(G)}\otimes\K(G)^*) (W^G)}\subset \B(L^2(G))$, which is a well-defined C*-algebra with $C^r_0(G)''=L^{\infty}(G)$. 
	We still write $\Delta_G$ for the restriction of $\Delta_G$ on $C^r_0(G)$, and it holds $(C^r_0(G),\Delta_G)$ is a well-defined bisimplifiable C*-bialgebra. 
	
	There is a locally compact quantum group $\hat{G}$, which we call the dual of $G$, such that 
	the GNS construction of the left Haar weight of $\hat{G}$ can be canonically identified with $L^2(G)$ 
	and $W^{\hat{G}}=\Sigma W^G{}^*\Sigma$ under this identification. 
	We write $\hat{W}^{G}:=W^{\hat{G}}$. 
	When we write $\Delta_G^{\cop}:=\sigma \Delta_G$, then
	$G^{\op}:=(L^{\infty}(G),\Delta_G^{\cop})$ and $G^{\cop}:=(L^{\infty}(G)^{\op},\Delta_G)$ are also locally compact quantum groups, where 
	$L^{\infty}(G)^{\op}$ means the opposite algebra of $L^{\infty}(G)$.
	We write $J^G$ for the modular conjugation of the left Haar weight of $G$. 
	Then we have an isomorphism of von Neumann bialgebras
	$R^G:=\hat{J}^{G}(-)^*\hat{J}^{G}:L^\infty(G)\to L^\infty(G^{\op,\cop})$ 
	called the unitary antipode, 
	where $\hat{J}^{G}:=J^{\hat{G}}$ is the modular conjugation of the left Haar weight of $\hat{G}$. 
	Sometimes we use the fact that $J^G\hat{J}^G$ is a scalar multiplication of $\hat{J}^GJ^G$. 
	
	For a locally compact quantum group $G$ we also have the universal version $(C^u_0(G),\Delta^u_G)$ of the C*-bialgebra $(C^r_0(G),\Delta_G)$, 
	and a canonical homomorphism of C*-bialgebras $\lambda_G :C^u_0(G)\to C^r_0(G)$ with $\lambda_G(C^u_0(G))=C^r_0(G)$. 
	
	\begin{defn}
		Let $G$ be a locally compact quantum group. 
		\begin{enumerate}
			\item
			We say $G$ is a \emph{compact quantum group} if $C^r_0(G)$ is unital. 
			\item
			We say $G$ is a \emph{discrete quantum group} if $\hat{G}$ is a compact quantum group. 
			\item
			We say $G$ is \emph{regular} if the multiplicative unitary $W^G\in \U(L^2(G)\otimes L^2(G))$ is \emph{regular}, 
			\\
			i.e.~
			$\clin{(\id\otimes\K(G)^*)(\Sigma W^G)}=\K(G)$. 
			\end{enumerate}
	\end{defn}
	
	When a locally compact quantum group $G$ is regular, then $\hat{G}$ and $G^{\op}$ are also regular, and we have $\clin{C^r_0(G)_1 W^G \K(G)_2}=\clin{\K(G)_2 W^G C^r_0(G)_1}=C^r_0(G)\otimes \K(G)$. 
	
	We also have a multiplicative unitary 
	$V^G=(\hat{J}^{G}\otimes \hat{J}^{G})
	\hat{W}^{G}(\hat{J}^{G}\otimes \hat{J}^{G}) 
	\in\M(\hat{J}^{G}C^r_0(\hat{G})\hat{J}^{G}\otimes C^r_0(G))$ 
	with respect to the right regular representation of $G$. 
	This satisfies $\Delta_G=\Ad V^G(-)_1$. 
	We warn that in some references the symbol $W$ is used instead of $V$, 
	which roughly corresponds to the different choice of left or right 
	for the GNS construction of Haar weights. 
	
	\subsection{Homomorphisms and bicharacters}\label{ssec:bichar}
	We recall several forms of homomorphisms between quantum groups 
	by following~\cite{Meyer-Roy-Woronowicz:qghom}. 
	
	\begin{defn}
		Let $G,H$ be locally compact quantum groups. 
		\begin{enumerate}
			\item
			A unitary $X\in \U\mult{C^r_0(H)\otimes C^r_0(\hat{G})}$ is called a \emph{bicharacter} if 
			$(\Delta_{H}\otimes \id_{C^r_0(\hat{G})}) X_{12}=X_{13}X_{23}$ and $(\id_{C^r_0(H)}\otimes \Delta_{\hat{G}}) X_{12}=X_{13}X_{12}$. 
			\item
			A \emph{homomorphism} of locally compact quantum groups from $H$ to $G$ is a homomorphism of C*-bialgebras $C^u_0(G)\to \mult{C^u_0(H)}$. 
			We use a formal symbol $\HG:H\to G$ to indicate a homomorphism from $H$ to $G$, 
			and we write $\HG^{u}:C^u_0(G)\to \mult{C^u_0(H)}$ for the homomorphism of C*-bialgebras associated with $\HG$. 
			\item
			For a homomorphism $\HG:H\to G$, 
			we define $\HG^r:C^r_0(G)\to \mult{C^r_0(H)}$ as a $*$-homomorphism with 
			$\HG^r\lambda_G=\lambda_H\HG^{u}$, if it exists. 
			Such a $*$-homomorphism is unique if it exists, 
			and $\HG^r$ is necessarily a homomorphism of C*-bialgebras. 
		\end{enumerate}
	\end{defn}
	
	Let $G,H$ be locally compact quantum groups. 
	For a homomorphism $\HG:H\to G$, 
	it holds that 
	there is the non-degenerate $*$-homomorphism $\HG^*\Delta_G:C^r_0(G)\to \mult{C^r_0(H)\otimes C^r_0(G)}$ with 
	$(\HG^*\Delta_{G})\lambda_G = 
	(\lambda_H\otimes \lambda_G)(\HG^{u}\otimes \id)\Delta^u_G$, and that 
	there is the bicharacter $W^\HG\in \U\mult{C^r_0(H)\otimes C^r_0(\hat{G})}$ 
	which satisfies $(\HG^*\Delta_G\otimes \id)W^G=W^{\HG}_{13}W^G_{23}$. 
	It also holds that $\HG^*\Delta_G=\Ad W^{\HG}{}^*(-)_{2}$. 
	The correspondence $\HG\mapsto W^\HG$ is bijective 
	as a map from the set of homomorphisms $H\to G$ to the set of bicharacters in $\U\mult{C^r_0(H)\otimes C^r_0(\hat{G})}$. 
	Any homomorphism $C^r_0(G)\to \mult{C^r_0(H)}$ of C*-bialgebras has the form of $\HG^r$ for some homomorphism $\HG:H\to G$, 
	and we have $W^\HG=(\HG^{r}\otimes \id)(W^G)$. 
	
	For a homomorphism $\HG:H\to G$, there are well-defined homomorphisms 
	$\hat{\HG}:\hat{G}\to \hat{H}$ and $\HG^{\op}:H^{\op}\to G^{\op}$
	such that $W^{\hat{\HG}}=\Sigma W^{\HG}{}^*\Sigma$ and $(\HG^{\op}){}^{u}=\HG^{u}:C^u_0(G)\to C^u_0(H)$. 
	When $\HG:H\to G$ and $\GF:G\to F$ are homomorphisms of locally compact quantum groups, their composition $\GF\HG:H\to F$ can be given by $(\GF\HG)^{u}:=\HG^{u}\GF^{u}$. 
	This is a well-defined homomorphism of locally compact quantum groups and the corresponding bicharacter $W^{\GF\HG}$ satisfies 
	$W^{\GF}_{23}W^{\HG}_{12}=W^{\HG}_{12}W^{\GF\HG}_{13}W^{\GF}_{23}$. 
	
	We have the trivial homomorphism $1_{H\to G}:H\to G$ 
	determined by the bicharacter $1\in\U\mult{C^r_0(H)\otimes C^r_0(\hat{G})}$. 
	The identity homomorphism $\id_G:G\to G$ which is given by $\id_{C^r_0(G)}$ satisfies $W^{\id_G}=W^G$. 
	
	Similarly to $\HG^*\Delta_G$, for a homomorphism $\HG:H\to G$ there is the non-degenerate $*$-homomorphism 
	$\Delta_G\rpb \HG:C^r_0(G)\to\M(C^r_0(G)\otimes C^r_0(H))$ with 
	$(\Delta_G\rpb \HG)\lambda_G = (\lambda_G\otimes \lambda_H)(\id\otimes \HG^{u})\Delta^u_G$. 
	Then we can take the unitary 
	$V^\HG\in\U\mult{\hat{J}^{G}C^r_0(\hat{G})\hat{J}^{G}\otimes C^r_0(H)}$ such that 
	$(\id\otimes\Delta_G\rpb \HG)V^G = V^G_{12}V^{\HG}_{13}$. 
	It holds $\sigma(R^H\otimes R^G)(\HG^*\Delta_G)=(\Delta_G\rpb \HG)R^G$, and 
	$V^\HG = (\hat{J}^{G}\otimes \hat{J}^{H})\hat{W}^{\HG}(\hat{J}^{G}\otimes \hat{J}^{H})$. 
	It also follows that 
	$V^\HG=(\id\otimes \HG^{r})V^G$ if $\HG^r$ exists and that 
	$V^{\HG}=(\hat{\HG}^{r}{}'\otimes \id)V^H$ if $\hat{\HG}^r$ exists, 
	where we write 
	$\hat{\HG}^{r}{}':=\hat{J}^{G}\hat{\HG}^{r}(\hat{J}^{H}(-)\hat{J}^{H})\hat{J}^{G}$. 
	We write $\hat{V}^{G}:=V^{\hat{G}}$, $\hat{V}^{\HG}:=V^{\hat{\HG}}$ 
	and $\hat{W}^{\HG}:=W^{\hat{\HG}}$ just as $\hat{W}^G$ and $\hat{J}^G$. 
	
	\subsection{Coactions}\label{ssec:coaction}
	
	Next, we collect definitions about coactions on C*-algebras of quantum groups. 
	
	\begin{defn}\label{def:coaction}
		Let $G$ be a locally compact quantum group, and $A$ be a C*-algebra.
		Let $\alpha:A\to \mult{C^r_0(G)\otimes A}$ be an injective non-degenerate $*$-homomorphism. 
		\begin{enumerate}
			\item
			We say $\alpha$ is a \emph{left $G$-coaction} on $A$ if
			$(\id\otimes \alpha) \alpha = (\Delta_G\otimes \id) \alpha$. 
			\item
			When $\alpha$ is a left $G$-coaction on $A$, 
			we say $\alpha$ is \emph{continuous (in the strong sense)} if 
			\\
			$\clin{(C^r_0(G)\otimes 1_A) \alpha(A)}=C^r_0(G)\otimes A$. 
			In this situation, we say $A=(A,\alpha)$ is a \emph{left $G$-C*-algebra}. 
			\item
			When $\alpha$ is a left $G$-coaction on $A$, 
			we say $\alpha$ is \emph{continuous in the weak sense} 
			if 
			\\
			$\clin{(\K(G)^*\otimes \id_A)(\alpha(A))}= A$. 
			\item 
			Let $(A,\alpha), (B,\beta)$ be pairs of C*-algebras and left $G$-coactions, and $f:A\to \mult{B}$ be a $*$-homomorphism. 
			When $f$ is non-degenerate, 
			$f$ is a \emph{left $G$-$*$-homomorphism} if 
			it holds that $\beta f=(\id\otimes f)\alpha$. 
			When $\alpha$ is continuous, 
			we say a $*$-homomorphism $f:A\to \mult{B}$ 
			(which need not be non-degenerate) is a \emph{left $G$-$*$-homomorphism} 
			if for all $a\in A$ and $x\in C^r_0(G)$, it holds that $(x\otimes 1_A)(\beta f(a))=(\id\otimes f)((x\otimes 1_A)\alpha(a))$. 
			\item 
			Consider a von Neumann algebra $M$ and an injective normal unital $*$-homomorphism $\alpha:M\to L^\infty(G)\barotimes M$. 
			We say $\alpha$ is a \emph{von Neumann algebraic left $G$-coaction} on $M$ if
			$(\id\barotimes \alpha) \alpha = (\Delta_G\barotimes \id) \alpha$. 
			\end{enumerate}
		
		We define a right $G$-coaction $\alpha:A\to \mult{A\otimes C^r_0(G)}$ on a C*-algebra $A$ 
		so that $\sigma\alpha$ is a left $G^{\op}$-coaction on $A$, 
		and a von Neumann algebraic right coaction in a similar way. 
		Likewise, we also define continuity and a $G$-$*$-homomorphism for a right coaction. 
	\end{defn}
	
	For a locally compact quantum group $G$, a left $G$-coaction is continuous in the weak sense if it is continuous in the strong sense, 
	and when $G$ is regular the converse also holds by~\cite[Proposition~5.8]{Baaj-Skandalis-Vaes:nonsemireg}. 
	
	For a locally compact quantum group $G$ and a left $G$-C*-algebra $(A,\alpha)$, 
	we can consider $\alpha(A)$ as a left $G$-C*-algebra with 
	$\Delta_G\otimes \id_A$, 
	and $\alpha:A\to \alpha(A)$ is a left $G$-$*$-isomorphism. 
	When $C$ is a C*-algebra (with a trivial coaction), $A\otimes C$ is a left $G$-C*-algebra with $\alpha\otimes\id$. 
	Also, $C^r_0(G)$ can be considered as a left $G$-C*-algebra with $\Delta_G$ and as a right $G$-C*-algebra again with $\Delta_G$. 
	Unless stated otherwise, we always equip these algebras with the coactions described here. 
	
	For a left $G$-C*-algebra $(A,\alpha)$, 
	we can define the reduced crossed product $G\redltimes A:=[C^r_0(\hat{G})_1 \alpha(A)]\subset \M(\K(G)\otimes A)$, 
	which is a C*-algebra with the continuous right $\hat{G}$-coaction $\Ad \hat{V}^{G}_{13}(-)_{12}$. 
	For a homomorphism $\HG:H\to G$ and a left $G$-C*-algebra $(A,\alpha)$, 
	there is the left continuous $H$-coaction $\HG^*\alpha$ on $A$ such that 
	$(\id\otimes \alpha)\HG^*\alpha=(\HG^*\Delta_G\otimes \id)\alpha$. 
	We call $\HG^*\alpha$ the \emph{restriction (or pullback) of $\alpha$ along $\HG$}. 
	We sometimes write $\HG^*A$ in this situation to make it clear what quantum group coacts on $A$. 
	Similarly for a right $G$-C*-algebra $(A,\alpha)$, we can also define 
	the restriction $\alpha\rpb \HG$ of $\alpha$ along $\HG$ as the right continuous $H$-coaction on $A$ which satisfies 
	$(\alpha\otimes\id)\alpha\rpb\HG=(\id\otimes\Delta_G\rpb\HG)\alpha$. 
	Also, we sometimes write $A\rpb\HG$. 
	%This notation is consistent with $\HG^*\Delta_G$ and $\Delta_G\rpb\HG$. 
	
	Now we recall and fix terminologies about coactions on Hilbert C*-modules and C*-correspondences. 
	We refer~\cite{Baaj-Skandalis:eqkk} for the details of the definitions and the constructions. 
	Consider a locally compact quantum group $G$ and left $G$-C*-algebras $(A,\alpha)$, $(B,\beta)$. 
	For a right $B$-Hilbert module $\E$, we can define a left or right coaction and its continuity in a similar way as in \cref{def:coaction}, but we additionally require $\left[\kappa(\mathcal{E})(C^r_0(G)\otimes 1_B)\right]=C^r_0(G)\otimes \mathcal{E}$ for a left $G$-coaction $\kappa$ of $\mathcal{E}$ to be continuous. 
	A left $G$-coaction $\kappa$ on $\E$ gives rise to 
	the left coaction on the linking algebra $\K_B(\E\oplus B)$ whose restrictions to $B$ and $\E$ equal $\beta$ and $\kappa$, respectively. 
	By abusing notation we continue to write $\kappa$ for this coaction on $\K_B(\E\oplus B)$. 
	It is shown in~\cite[Remark~12.5]{Vaes:impr} that 
	when $G$ is a regular locally compact quantum group and $(B,\beta)$ is a left $G$-C*-algebra, 
	a left $G$-coaction $\kappa$ on a right Hilbert $B$-module $\E$ is automatically continuous, 
	whose corresponding left coaction on $\K_B(\E\oplus B)$ is also continuous. 
	
	For an $(A,B)$-correspondence $(\E,\pi)$ (i.e.~a right Hilbert $B$-module with a $*$-homomorphism $\pi:A\to\LC_B(\E)$), 
	a continuous left $G$-coaction $\kappa$ on $\E$ 
	is a \emph{continuous left $G$-coaction on $(\E,\pi)$} if $\pi:A\to\LC_B(\E)=\mult{\K_B(\E)}$ is a left $G$-$*$-homomorphism. 
	In this case we say $(\E,\pi)$ is a \emph{left $G$-$(A,B)$-correspondence}. 
	The inner tensor product of left $G$-C*-correspondences has a canonical structure of a left $G$-C*-correspondence. 
	We say a left $G$-$(A,B)$-correspondence $\E$ is a left \emph{$G$-$(A,B)$-imprimitivity bimodule} if $\E$ is also an imprimitivity bimodule. 
	In such a situation, $A$ and $B$ are said to be \emph{left $G$-Morita equivalent}. 
	
	\subsection{Closed and open quantum subgroups}\label{ssec:clopen}
	
	We recall the notions of closed quantum subgroups and open quantum subgroups by following~\cite{Daws-Kasprzak-Skalski-Soltan:closed},~\cite{Kalantar-Kasprzak-Skalski:open}. 
	
	\begin{defn}\label{def:clopen}
		Let $G,H$ be locally compact quantum groups and $\HG:H\to G$ be a homomorphism. 
		\begin{enumerate}
			\item 
			$\HG$ gives a \emph{closed quantum subgroup} if 
			$\hat{\HG}^r$ exists and $\hat{\HG}^r$ extends to a normal injective unital
			$*$-homomorphism 
			$L^{\infty}(\hat{H})\to L^{\infty}(\hat{G})$. 
			We still write $\hat{\HG}^r$ for this extension. 
			\item
			$\HG$ gives an \emph{open quantum subgroup} if 
			$\HG^r$ exists and $\HG^r$ extends to a normal surjective unital
			$*$-homomorphism 
			$L^{\infty}(G)\to L^{\infty}(H)$. 
			We still write $\HG^r$ for this extension. 
			\end{enumerate}
	\end{defn}
	
	We have followed the definition of a closed quantum subgroup by Vaes. 
	By~\cite[Theorem~3.5, Theorem~3.6]{Daws-Kasprzak-Skalski-Soltan:closed}, 
	if $\HG:H\to G$ gives a closed quantum subgroup 
	then $\HG^{u}:C^u_0(G)\to C^u_0(H)$ is surjective. 
	An open quantum subgroup is a closed quantum subgroup by~\cite[Theorem~3.6]{Kalantar-Kasprzak-Skalski:open}. 
	
	We collect some properties about open quantum subgroups and their consequences (see Section~2 in~\cite{Kalantar-Kasprzak-Skalski:open}). 
	Consider a homomorphism $\iota:H\to G$ 
	giving an open quantum subgroup. 
	Let $P=P^{\iota}\in\PC\ZC L^{\infty}(G)$, where $\PC$ indicates the set of projections, be the central support of $\iota^{r}:L^{\infty}(G)\rsurj L^{\infty}(H)$. 
	Then $P$ satisfies $[W^G,P\otimes P]=0$, and $\Delta_{G}(P)(P\otimes 1)=P\otimes P=\Delta_{G}(P)(1\otimes P)$ in $\B(L^2(G)^{\otimes 2})$. 
	It follows $H=(L^\infty(H),\Delta_H)$ can be canonically identified as $(PL^\infty(G), (P\otimes P)\Delta_{G}(-))$, 
	with the left and right Haar weights given by restricting those of $G$ via $P$. 
	We write $p=p^{\iota}$ for $P$ considered as a coisometry $L^2(G)\rsurj L^2(H)$, so we have $P=p^*p$. 
	It holds 
	$\iota^r=p(-)p^*:L^{\infty}(G)\to L^{\infty}(H)=pL^{\infty}(G)p^*$, and 
	$(p\otimes p)W^G=W^H(p\otimes p)$. 
	We also remark that the compatibility of the left Haar weights shows $pJ^G=J^Hp$. 
	The calculations 
	\begin{align*}
		&
		p_2W^Gp_2^*=p_2W^G\Delta_G(P)p_2^*=p_2W^G(P\otimes P)p_2^*
		=p_2 (p\otimes p)^*W^H(p\otimes p)p_2^*
		=p_1^*W^Hp_1 , 
		\\&
		W^H=(p\otimes p)W^G (p\otimes p)^*
		=p_2 W^{\iota} p_2^*=p_2(\id\otimes \hat{\iota}^{r})(W^{H})p_2^* , 
	\end{align*}
	show that $\hat{\iota}^r(p(-)p^*):C^r_0(\hat{G})\rsurj \hat{\iota}^r C^r_0(\hat{H})$ 
	is a well-defined conditional expectation by taking $\omega\otimes\id$ for each $\omega$ in $\B(L^2(G))_*$ or $\B(L^2(H))_*$ to them. 
	This map also gives the conditional expectations  
	$L^{\infty}(\hat{G})\rsurj \hat{\iota}^rL^{\infty}(\hat{H})$ and 
	$\mult{C^r_0(\hat{G})}\rsurj \hat{\iota}^r\mult{C^r_0(\hat{H})}$. 
	
	Finally, we remark that if $G$ is regular, then $H$ is also regular. 
	To see this, since we have 
	\begin{align*}
		&
		\clin{(\id\otimes \K(H)^*)(\Sigma W^H)}
		=
		\clin{(\id\otimes \K(H)^*)\bigl((p\otimes p)\Sigma W^G(p\otimes p)^*\bigr)}
		\\&
		\subset
		p\clin{(\id\otimes \K(G)^*)(\Sigma W^G)}p^*
		=p\K(G)p^*=\K(H), 
	\end{align*}
	it suffices to show $\K(H)\subset \clin{(\id\otimes \K(H)^*)(\Sigma W^H)}$, which is called \emph{semi-regularity} of $H$. 
	Note that $H$ is semi-regular if and only if so is $\hat{H}$. 
	But as pointed out in the proof of~\cite[Proposition~9.3]{Baaj-Vaes:doublecrossed}, 
	it follows from~\cite[Proposition~5.7]{Baaj-Skandalis-Vaes:nonsemireg} 
	that semi-regularity passes to any closed quantum subgroups. 
	
	\section{Formal aspects of induced coactions}\label{sec:formal}
	
	First, we recall the description in~\cite{Vaes:impr} of induced coactions along a homomorphism $\HG:H\to G$ of regular locally compact quantum groups which gives a closed quantum subgroup. 
	We write the von Neumann algebra 
	$L^\infty(G/H):=\{x\in L^\infty(G)\,|\,\Ad V^{\HG}x_1=x\otimes 1\}$, and 
	when $G,H$ are regular there is a C*-subalgebra $C^r_0(G/H)\subset L^\infty(G/H)$. 
	The restrictions of $\Delta_G$ give a von Neumann algebraic left $G$-coaction on $L^\infty(G/H)$ and a continuous left $G$-coaction on $C^r_0(G/H)$. 
	We note $\Ad V^G_{12}(-)_1$ gives 
	$L^\infty(\hat{G})\vee L^\infty(G/H)\xrarr{\sim}G\ltimes L^\infty(G/H)$ and 
	$[C^r_0(\hat{G})C^r_0(G/H)]\xrarr{\sim}G\redltimes C^r_0(G/H)$. 
	These $*$-isomorphisms are from algebras with one leg and algebras with two legs, 
	and it will be convenient to distinguish these algebras for later usages of the leg-numbering notation. 
	Let 
	\begin{align*}
		&
		\I^\HG:=\{v\in\B(L^2(H),L^2(G)) \,|\, 
		v\hat{J}^H x\hat{J}^H=\hat{J}^G\hat{\HG}^{r}(x)\hat{J}^G v 
		\mathrm{\ for\ all\ } x\in L^\infty (\hat{H}) \} , 
	\end{align*}
	which is an $(L^\infty(\hat{G})\vee L^\infty(G/H),L^{\infty}(\hat{H}))$-imprimitivity W*-bimodule (see the remark after \cite[Definition~4.2]{Vaes:impr}). 
	Now suppose $G,H$ are regular. 
	There is $I^\HG\subset \I^\HG$ which is a 
	$([C^r_0(\hat{G}) C^r_0(G/H)],C^r_0(\hat{H}))$-imprimitivity (C*-)bimodule 
	by Theorem~6.2 (or more generally Theorem~7.3) and the argument from Theorem~8.2 in~\cite{Vaes:impr}. 
	Actually $I^\HG$ is constructed before $C^r_0(G/H)$, and then $C^r_0(G/H)$ is constructed by using~\cite[Theorem~6.7]{Vaes:impr}, which is the quantum version of Landstad's theorem. 
	It follows that $I^\HG\subset \I^\HG$ is dense with respect to the strong* operator topology of $\B(L^2(G)\oplus L^2(H))$, and that
	\begin{align}\label{eq:recallclosedind1}
		&
		\clin{I^\HG L^2(H)} = \clin{\I^\HG L^2(H)} = L^2(G), 
		\quad\mathrm{and}\quad
		\clin{I^\HG{}^* L^2(G)} = \clin{\I^\HG{}^* L^2(G)} = L^2(H). 
	\end{align}
	More generally for a left $H$-C*-algebra $(A,\alpha)$, 
	there exists the C*-subalgebra $\Ind_{\HG}A\subset\M(C^r_0(G)\otimes A)$ with the induced left $G$-coaction. 
	This is constructed as a certain C*-subalgebra of 
	the C*-algebra that will turn out to be $\M(G\redltimes\Ind_{\HG}A)$, 
	again by using~\cite[Theorem~6.7]{Vaes:impr}. 
	Also, by Theorem~7.3 and the proof of Theorem~8.2 in~\cite{Vaes:impr}, 
	it holds that 
	$[I^\HG_1\alpha(A)]\subset \LC_A(L^2(H)\otimes A, L^2(G)\otimes A)$ is a right 
	$\hat{G}$-$([C^r_0(\hat{G})_1 \Ind_\HG A],H\redltimes A)$-imprimitivity bimodule 
	with the continuous right $\hat{G}$-coaction given by 
	$\hat{V}^{G}_{13}(-)_{12}\hat{V}^{\HG}_{13}{}^*$. 
	When $A=\C$ we have $\Ind_{\HG}\C=C^r_0(G/H)$. 
	
	\begin{rem}\label{rem:strreg}
		We note that the imprimitivity theorem for full crossed products is also shown in~\cite{Vaes:impr} and locally compact quantum groups are assumed to be strongly regular. 
		But as long as we concern only with the reduced crossed products, regularity is enough to show the results quoted above. 
		Indeed, it is already pointed out by the remark after Theorem~6.2 in the original paper~\cite{Vaes:impr} and by Remark~1.5 in~\cite{Kalantar-Kasprzak-Skalski:open}, 
		that the existence of the imprimitivity bimodule $I^\HG$ do not require strong regularity. 
		We can continue to check the proofs of the quoted facts still work only with regularity. 
		We list the results of~\cite{Vaes:impr} which were stated under the assumption of strong regularity, that we will use hereafter: 
		Theorem~7.2, Theorem~7.3 for reduced crossed products, the proof of Theorem~8.2. 
		We only need these results for injective coactions. 
	\end{rem}
	
	\subsection{Definition and first properties}\label{ssec:formal}
	The left $G$-C*-algebra $\Ind_\HG A$ above satisfies the characterization due to~\cite[Theorem~7.2]{Vaes:impr}. 
	In order to consider induced coactions in more general settings, 
	we formally extend that characterization as follows. 
	
	\begin{defn}\label{def:ind}
		Let $\HG:H\to G$ be a homomorphism of locally compact quantum groups, 
		and $(A,\alpha)$ be a C*-algebra with a left $H$-coaction which is continuous in the weak sense. 
		We define 
		\begin{align*}
			\WInd_\HG A:=\{x\in\mult{\K(G)\otimes A} \,|\, 
			x\in (C^r_0(G)'\otimes 1_A)' \mathrm{\ and\ } 
			\Ad V^{\HG}_{12}(x_{13})=(\id_{\K(G)}\otimes \alpha)(x) \} . 
		\end{align*}
		A C*-algebra $C$ with the left $G$-coaction given below is called an \emph{induced coaction of $(A,\alpha)$ along $\HG$} if following three conditions hold. 
		
		\begin{enumerate}
			\item[(i)] 
			$C\subset \mult{\K(G)\otimes A}$ is a non-degenerate C*-subalgebra 
			with $C\subset \WInd_{\HG}A$. 
			\item[(ii)]
			The restriction of $\Ad V^G_{12}(-)_{13}:\mult{\K(G)\otimes A}\to \mult{\K(G)\otimes \K(G)\otimes A}$ gives a well-defined left $G$-coaction on $C$ which is continuous in the weak sense. 
			\item[(iii)]
			$\Ad V^G_{12}(-)_{13}: \WInd_{\HG}A\to \mult{\K(G)\otimes C}$ is a well-defined $*$-homomorphism whose restriction to the unit ball is continuous from the strict topology of $\mult{\K(G)\otimes A}$ to the strict topology of $\mult{\K(G)\otimes C}$. 
		\end{enumerate}
	\end{defn}
	
	We can also define an induced coaction of a right coaction that is continuous in the weak sense along a homomorphism of locally compact quantum groups, 
	by reversing left and right in the definition above. 
	
	\begin{rem}\label{rem:indunique}
		In the situation of \cref{def:ind} above, 
		an induced coaction $C$ of $(A,\alpha)$ along $\HG$ is unique if it exists 
		by the same argument from~\cite[Remark~6.8]{Vaes:impr}. 
		Indeed, for any C*-algebra $D$ satisfying (i) and (ii), 
		it holds that 
		\begin{align*}
			&
			\clin{CD}=\clin{C\bigl((\K(G)^*\otimes\id\otimes\id)(\Ad V^G_{12}D_{13})\bigr)}=C . 
		\end{align*}
		If $D$ additionally satisfies (iii), we also have $C=\clin{DC}=D$ by symmetry. 
		We will write $\Ind_{\HG}A$ for \emph{the} unique induced coaction $C$ of $(A,\alpha)$ along $\HG$ if it exists, 
		and $\Ind_\HG \alpha$ for its left $G$-coaction. 
	\end{rem}
	
	\begin{rem}\label{rem:indmulthom}
		Let $\HG:H\to G$ be a homomorphism of locally compact quantum groups and 
		$(A,\alpha)$, $(B,\beta)$ be C*-algebras with left $H$-coactions which are continuous in the weak sense. 
		Suppose $\Ind_\HG A$ and $\Ind_\HG B$ exist. 
		It holds $\Ind_\HG A\subset \mult{C^r_0(G)\otimes A}$ is a non-degenerate C*-subalgebra since 
			\begin{align*}
				&\clin{(\Ind_\HG A) (C^r_0(G)\otimes A)}
				=
				\clin{
					(\K(G)^*\otimes \id\otimes\id)\bigl(
					V^G_{12} (\Ind_\HG A)_{13} V^G_{12}{}^* (\K(G)\otimes C^r_0(G)\otimes A) \bigr) 
				}
				\\&
				= C^r_0(G)\otimes A. 
			\end{align*}
		When $f:A\to \M(B)$ is a non-degenerate left $H$-$*$-homomorphism, 
		$(\id\otimes f)\Ind_\HG A$ satisfies 
		(i) and (ii) for the definition of $\Ind_\HG B$. 
		It follows from the previous \cref{rem:indunique} that
		$\id\otimes f:\Ind_\HG A\to \mult{\Ind_\HG B}$ is a well-defined non-degenerate left $G$-$*$-homomorphism. 
	\end{rem}
	
	As already mentioned, when $\HG:H\to G$ is a homomorphism of regular locally compact quantum groups giving a closed quantum subgroup, then 
	the construction of the induced coaction $\Ind_{\HG} A$ due to Vaes satisfies \cref{def:ind} by~\cite[Theorem~7.2]{Vaes:impr}. 
	Next, we give a few examples of this formal generalization of induced coactions when the homomorphism does not give a closed quantum subgroup. 
	
	\begin{eg}\label{eg:formalind}
		Let $\HG:H\to G$ be a homomorphism of locally compact quantum groups. 
		\begin{enumerate}
			\item
			Consider the case of a trivial homomorphism $1_{H\to 1}:H\to 1=G$. 
			Assume $H$ is a locally compact group and 
			$X$ is a locally compact Hausdorff space with a proper left action of $H$. 
			Then $C_0(H\bs X)$ gives $\Ind_{1_{H\to 1}}C_0(X)$. 
			Indeed, we observe that $\mult{C_0(H\bs X)}\xrarr{\sim} \WInd_{1_{H\to 1}}C_0(X)$ via the quotient map $X\rsurj H\bs X$, 
			which shows (i) of \cref{def:ind}. (ii) is trivial. 
			We show (iii). Take a net $(f_\lambda)_\lambda$ in the unit ball of $\WInd_{1_{H\to 1}}C_0(X)$ and $f\in\WInd_{1_{H\to 1}}C_0(X)$ 
			such that $\norm{f_\lambda g -fg} \xrarr{\lambda} 0$ for any $g\in C_0(X)$. 
			Since for any $0\leq h\in C_0(H\bs X)$ and any $\varepsilon>0$ there is some $0\leq h_\varepsilon\in C_c(X)$ such that 
			for all $x\in X$ we have 
			$\sup\limits_{t\in H}h_\varepsilon(tx)>h(x)-\varepsilon$, it follows $\norm{f_\lambda h -fh}\xrarr{\lambda} 0$. 
			\item
			As a tautological example, 
			for any homomorphism $\HG:H\to G$ of locally compact quantum groups, 
			$(\Delta_G\rpb \HG) C^r_0(G)$ gives $\Ind_{\HG}C^r_0(H)$. 
			The conditions (i) and (ii) in \cref{def:ind} are clear. 
			The restriction of the map $\Ad V^{\HG}_{12}(-)_1: \mult{\K(G)}\to \mult{\K(G)\otimes C^r_0(H)}$ 
			to the unit ball is a homeomorphism onto its image with respect to the strict topologies. 
			Indeed, the inverse is given by $(\id\otimes\omega)\Ad V^{\HG}_{12}{}^*$ for some state $\omega$ on $C^r_0(H)$. 
			Since $\WInd_{\HG}C^r_0(H)\subset (\Delta_G\rpb \HG) L^\infty(G)$
			by the following \cref{lem:vNind}, 
			(iii) of \cref{def:ind} follows from the next commutative diagram 
			\begin{align*}
				\xymatrix{
					\WInd_{\HG}C^r_0(H) \ar[rr]^-{\Ad V^G_{12}(-)_{13}}
					\ar[d]_-{(\id\otimes\omega)\Ad V^{\HG}_{12}{}^*}&& 
					\M(\K(G)\otimes C^r_0(G)\otimes C^r_0(H))\\
					\M(\K(G))\ar[r]^-{\Ad V^{G}(-)_1}& 
					\M(\K(G)\otimes C^r_0(G)) 
					\ar[r]^-{\Ad V^{\HG}_{23}(-)_{12}}&
					\M(\K(G)\otimes \Delta\rpb\HG C^r_0(G)) , 
					\ar[u]_-{\cup}
				}
			\end{align*}
			whose vertical arrows are homeomorphic on the unit balls with respect to the strict topologies. 
		\end{enumerate}
	\end{eg}
	
	\begin{lem}\label{lem:vNind}
		For a homomorphism $\HG:H\to G$ of locally compact quantum groups, it holds
		\begin{align*}
			(L^\infty(G)\barotimes L^\infty(H))\cap 
			\coequalizer{\Delta_{G}\rpb \HG\barotimes \id}{\id\barotimes \Delta_H}
			= (\Delta_{G}\rpb \HG)(L^\infty(G)) . 
		\end{align*}
	\end{lem}
	
	\begin{proof}
		We write $M$ for the left hand side. 
		It is clear that $M\supset (\Delta_{G}\rpb \HG)(L^\infty(G))$. 
		As for the converse inclusion, we observe $\id\barotimes \Delta_H$ 
		is a right von Neumann algebraic $H$-coaction on $M$. 
		Therefore 
		by~\cite[Theorem~2.6.1]{Vaes:unitaryimplement} 
		it holds 
		$(1\otimes1\otimes\B(L^2(H)))\vee (\id\otimes \Delta_H)M=M\barotimes\B(L^2(H))$, 
		but its left hand side is contained in $((\Delta_{G}\rpb \HG)L^\infty(G))\barotimes\B(L^2(H))$ by the definition of $M$. 
	\end{proof}
	
	Next, we are going to show the induction in stages \cref{thm:indstages}, 
	which gives a foundation of the remaining part. 
	We begin with the following lemma. 
	
	\begin{lem}\label{lem:indtensorcpt}
		Let $\HG:H\to G$ be a homomorphism of locally compact quantum groups, $(A,\alpha)$ be a C*-algebra 
		with a left $H$-coaction which is continuous in the weak sense, 
		and $\HC$ be a Hilbert space. 
		If $\Ind_\HG A$ exists, then 
		$\Ind_\HG (A\otimes \K(\HC))$ 
		also exists and is given by $(\Ind_\HG A)\otimes \K(\HC)$. 
	\end{lem}
	
	\begin{proof}
		It is clear that 
		$(\Ind_\HG A)\otimes \K(\HC)$ satisfies (i) and (ii) in \cref{def:ind} for $\Ind_\HG (A\otimes \K(\HC))$. 
		If $\HC$ is finite dimensional, (iii) also holds since $\WInd_\HG (A\otimes \K(\HC))=(\WInd_\HG A)\otimes \K(\HC)$. 
		We consider the general case and for each finite dimensional subspace $\V\subset \HC$ we write $1_{\V}$ for the unit of $\K(\V)$. 
		
		Fix arbitrarily a net $(x_\lambda)_{\lambda\in \Lambda}$ and an element $x$ in the unit ball of $\WInd_{\HG}(A\otimes \K(\HC))$ 
		such that $x_\lambda \xrightarrow{\lambda\in \Lambda} x$ 
		with respect to the strict topology of $\mult{\K(G)\otimes A\otimes \K(\HC)}$. 
		For each finite dimensional subspace $\V\subset \HC$ 
		we have 
		$(1_{\V})_3(x_\lambda-x)^*(x_\lambda-x)(1_{\V})_3
		\xrarr{\lambda}0$ 
		strictly in $\mult{\K(G)\otimes A\otimes \K(\V)}$ 
		and
		\begin{align*}
			&
			V^G_{12} (1_{\V})_4
			(x_\lambda-x)_{134}^*(x_\lambda-x)_{134}
			(1_{\V})_4 V^G_{12}{}^*\xrarr{\lambda} 0
		\end{align*}
		strictly in $\mult{\K(G)\otimes (\Ind_{\HG}A)\otimes \K(\V)}$ by the finite dimensional case. 
		Thus for any $z\in \K(G)\otimes (\Ind_{\HG}A) \otimes \K(\HC)$, it holds 
		\begin{align*}
			&
			\norm{\bigl(\Ad V^G_{12}(x_\lambda -x)_{134}\bigr)z}
			=\norm{z^* V^G_{12} (x_\lambda-x)^*_{134}(x_\lambda-x)_{134} V^G_{12}{}^* z}^{\frac{1}{2}}
			\xrarr{\lambda} 0 . 
		\end{align*}
		By replacing $x$ and $x_\lambda$ with $x^*$ and $x_{\lambda}^*$, respectively, we also get 
		$\norm{\bigl(\Ad V^G_{12}(x_{\lambda}^* -x^*)_{134}\bigr)z}\xrarr{\lambda}0$. 
		
		It remains to show $\Ad V^G_{12}(x)_{134}\in \mult{\K(G)\otimes (\Ind_\HG A)\otimes \K(\HC)}$ for any $x\in\WInd_{\HG} (A\otimes \K(\HC))$. 
		Let $\Lambda$ be the poset of all finite dimensional subspace of $\HC$ with respect to the inclusions, 
		and for each $\V\in\Lambda$ put 
		$x_\V := (1_{\V})_3x(1_{\V})_3$. 
		Then $(x_{\V})_{\V}$ is a norm-bounded net in $\WInd_{\HG} (A\otimes \K(\HC))$ 
		strictly converging to $x$ in $\mult{\K(G)\otimes A\otimes \K(\HC)}$. 
		For any $z\in \K(G)\otimes (\Ind_{\HG}A) \otimes \K(\HC)$, 
		the argument above shows 
		$(\Ad V^G_{12}x_{134})z$ 
		is a norm-limit of 
		$(\Ad V^G_{12}(x_{\V})_{134})z \in \K(G)\otimes (\Ind_{\HG}A) \otimes \K(\HC)$. 
		Similarly it holds 
		$(\Ad V^G_{12}x^*_{134})z \in \K(G)\otimes (\Ind_{\HG}A) \otimes \K(\HC)$, 
		and thus we get 
		$\Ad V^G_{12}x_{134}\in \mult{\K(G)\otimes (\Ind_\HG A)\otimes \K(\HC)}$. 
	\end{proof}
	
	Later we will see the stronger compatibility of induced coactions with tensor products holds for the class of proper homomorphism 
	in \cref{cor:properind} and \cref{prop:projfml}. 
	
	\begin{thm}\label{thm:indstages}
		Let $\HG:H\to G$, $\GF:G\to F$ be homomorphisms of locally compact quantum groups and $(A,\alpha)$ be a C*-algebra with a left $H$-coaction which is continuous in the weak sense. 
		Suppose $\Ind_{\HG}A$ and $\Ind_{\GF}(\Ind_{\HG}A)$ exist. 
		Then $\Ind_{\GF\HG}A$ also exists and 
		$\Ad V^{\GF}_{12}(-)_{13}$ gives a well-defined left $H$-$*$-isomorphism 
		$\Ind_{\GF\HG}A\xrarr{\sim} \Ind_{\GF}(\Ind_{\HG}A)$. 
	\end{thm}
	
	\begin{proof}
		It follows from \cref{lem:vNind} that 
		\begin{align*}
			&
			\mult{\K(F)\otimes \K(G)\otimes A} 
			\cap (L^\infty(F)'\barotimes L^\infty(G)'\otimes 1_A)' 
			\cap 
			\coequalizer{\Ad V^{\GF}_{12}(-)_{134}}{\Ad V^G_{23}(-)_{124}}
			\\&\subset 
			\Ad V^{\GF}_{12}\left(
			\mult{\K(F)\otimes 1_{\K(G)}\otimes A}
			\cap 
			(L^\infty(F)'\otimes1_{\K(G)}\otimes 1_{A})'\right) .
		\end{align*} 
		Indeed, for any element $x$ in the left hand side and any $\omega\in A^*$ we have $(\Ad V^{\GF}{}^* \otimes \omega)x\in L^\infty(F)\otimes1_{\K(G)}$, 
		which shows 
		$\Ad V^{\GF}_{12}{}^*x \in (L^\infty(F)'\barotimes \B(L^2(G)) \otimes 1_A)'$. 
		Thus there is some C*-subalgebra $C \subset\mult{\K(F)\otimes A}\cap (L^\infty(F)'\otimes1_A)'$ such that 
		$\Ind_{\GF}(\Ind_{\HG}A)=\Ad V^{\GF}_{12}C_{13}$. 
		Since for any $z\in \mult{\K(F)\otimes A}$ it holds 
		\begin{align}\label{prf:thm:indstages1}
			&
			\Ad (V^{F}_{12}V^{\GF}_{13})(z)_{14}
			=\Ad (V^{F}_{12}V^{\GF}_{13}V^{\GF}_{23})(z)_{14}
			=\Ad (V^{\GF}_{23}V^{F}_{12})(z)_{14} ,
		\end{align}
		it follows from (ii) of \cref{def:ind} for $\Ind_{\GF}(\Ind_{\HG}A)$ that 
		$\Ad V^{F}_{12}(-)_{13}$ gives a well-defined left $F$-coaction on $C$ 
		such that $\Ad V^{\GF}_{12}:C\xrarr{\sim} \Ind_{\GF}(\Ind_{\HG}A)$ is a left $F$-$*$-isomorphism. 
		In particular, this left $F$-coaction on $C$ is continuous in the weak sense. 
		Therefore it is enough to show that $C$ satisfies (i) and (iii) for \cref{def:ind} of $\Ind_{\GF\HG}A$. 
		
		Since 
		$\Ind_{\GF}(\Ind_{\HG}A)\subset\mult{\K(F)\otimes \K(G)\otimes A}$ is non-degenerate, 
		\begin{align*}
		&C
		=\clin{(\id\otimes \K(G)^*\otimes \id)
			\Bigl(\Ad V^{\GF}_{12}{}^*(\Ind_{\GF}(\Ind_{\HG}A))\Bigr)}
		\subset \mult{\K(F)\otimes A}
		\end{align*} 
		is also non-degenerate. 
		For any $z\in C$ we have 
		$\Ad V^{\GF}_{12}(z)_{13}\in\mult{\K(F)\otimes \Ind_{\HG}A}$, 
		and thus 
		\begin{align*}
			&
			\Ad (V^{\GF\HG}_{13})(z)_{14}
			=
			\Ad (V^{\GF}_{12}{}^* V^{\HG}_{23}V^{\GF}_{12}V^{\HG}_{23}{}^* )(z)_{14}
			\\&
			=
			\Ad V^{\GF}_{12}{}^* \bigl((\id\otimes\id\otimes \alpha) (\Ad V^{\GF}_{12}(z)_{13})\bigr)
			=
			\bigl((\id\otimes \alpha)(z)\bigr)_{134}. 
		\end{align*}
		It follows $C\subset\coequalizer{\Ad V^{\GF\HG}_{12}(-)_{13}}{\id\otimes \alpha}$. 
		Therefore $C$ satisfies (i). 
		
		We show (iii). 
		By the strict continuity of $\Ad V^{\GF}_{12}(-)_{13}:\mult{\K(F)\otimes A}\to\mult{\K(F)\otimes\K(G)\otimes A}$ 
		and by \cref{prf:thm:indstages1}, it suffices to show that
		\begin{align*}
			&
			\Ad V^{F}_{12}(-)_{134}:\Ad V^{\GF}_{12}(\WInd_{\GF\HG}A)_{13}
			\to \mult{\K(F)\otimes (\Ad V^{\GF}_{12}C_{13})}
			= \mult{\K(F)\otimes \Ind_{\GF}(\Ind_{\HG}A)}
		\end{align*}
		is a well-defined $*$-homomorphism and continuous on the unit ball from the strict topology of 
		$\mult{\K(F)\otimes \K(G)\otimes A}$ 
		to the strict topology of $\mult{\K(F)\otimes \Ind_{\GF}(\Ind_{\HG}A)}$. 
		Moreover, by (iii) of $\Ind_{\GF}(\Ind_{\HG}A)$, it is enough to show that 
		\begin{align*}
			&
			\Ad V^{\GF}_{12}(\WInd_{\GF\HG}A)_{13} 
			\subset \WInd_{\GF}(\Ind_{\HG}A) 
			\quad\mathrm{in}\quad \mult{\K(F)\otimes \K(G)\otimes A}, 
		\end{align*}
		and 
		that this inclusion is continuous on the unit ball 
		from the strict topology of $\mult{\K(F)\otimes \K(G)\otimes A}$ 
		to the strict topology of $\mult{\K(F)\otimes \Ind_{\HG}A}$. 
		
		Fix a normal state $\omega$ on $\B(L^2(G))$. 
		Since $\Ad V^{\GF}_{12}(-)_{134}=\Ad V^{G}_{23}(-)_{124}$ on $\Ad V^{\GF}_{12}(\WInd_{\GF\HG}A)_{13}$, 
		we have 
		\begin{align}\label{prf:thm:indstages2}
			&\id_{\Ad V^{\GF}_{12}(\WInd_{\GF\HG}A)_{13}} 
			= (\id\otimes \omega\otimes \id\otimes\id)\Ad (V^{\GF}_{12}{}^*V^{G}_{23})(-)_{124} . 
		\end{align}
		We also have 
		\begin{align*}
			&\Ad V^{\GF}_{12}(\WInd_{\GF\HG}A)_{13}
			\subset \bigl(\WInd_{\HG}(A\otimes \K(F))\bigr)_{231}
			\subset \mult{\K(F)\otimes \K(G)\otimes A} 
		\end{align*}
		when we equip $\K(F)$ with a trivial coaction, 
		because for any $z\in\WInd_{\GF\HG}A$ it holds 
		\begin{align*}
			&\Ad(V^{\HG}_{23}V^{\GF}_{12}) (z)_{14} 
			=
			\Ad(V^{\GF}_{12}V^{\GF\HG}_{13}V^{\HG}_{23}) (z)_{14}
			=
			\Ad V^{\GF}_{12}((\id\otimes \alpha)(z))_{134}
			\\&
			=(\id\otimes\id\otimes \alpha)\Ad V^{\GF}_{12}(z)_{13}. 
		\end{align*}
		By \cref{lem:indtensorcpt}, it follows from 
		(iii) for $\Ind_{\HG}(A\otimes \K(F))=(\Ind_{\HG}A)\otimes \K(F)$
		that 
		\begin{align*}
			\Ad V^{G}_{23}(-)_{124}:\Ad V^{\GF}_{12}(\WInd_{\GF\HG}A)_{13} \to \mult{\K(F)\otimes \K(G)\otimes \Ind_{\HG}A}
		\end{align*}
		is well-defined and continuous on the unit ball from the strict topology of $\M(\K(F)\otimes \K(G)\otimes A)$ to the strict topology of $\M(\K(F)\otimes\K(G)\otimes\Ind_\HG A)$. 
		
		Therefore by composing this map with the strictly continuous map 
		\begin{align*}
			(\id \otimes \omega\otimes \id\otimes\id)\Ad V^{\GF}_{12}{}^* :
			\mult{\K(F)\otimes \K(G)\otimes \Ind_{\HG}A}\to
			\mult{\K(F)\otimes \Ind_{\HG}A}, 
		\end{align*}
		we obtain that the right hand side of \cref{prf:thm:indstages2} is a well-defined map into $\mult{\K(F)\otimes \Ind_{\HG}A}$ with the desired continuity. 
		Now it is easy to check 
		$\Ad V^{\GF}_{12}(\WInd_{\GF\HG}A)_{13}\subset \WInd_{\GF}(\Ind_{\HG}A)$. 
	\end{proof}
	
	\subsection{Quantum homogeneous spaces as quotients}\label{ssec:hmgquot}
	Before going to the next section, 
	we briefly discuss the coincidence of the two possibilities for quantum homogeneous spaces. 
	In the following, for a C*-algebra $A$ with a right coaction $\alpha$ of a locally compact quantum group $G$ which is continuous in the weak sense, 
	we write $\M(A)^G:=\{x\in \M(A)\,|\, \alpha(x)=x\otimes1\}$. 
	When the induced coaction 
	$D\subset \M(A)\isom \M(A\otimes\C)$ of $A$ along $1_{G\to 1}$ exists, we write $A^G:=D$. 
	In other words, 
	$A^G\subset \M(A)$ is the non-degenerate C*-subalgebra such that $\M(A^G)=\M(A)^G$ and their strict topologies coincide on the unit balls. 
	If $A^G$ exists, it is a good candidate for the fixed point algebra of $A$ with the right coaction $G$ (see (1) of \cref{eg:formalind}). 
	For a homomorphism $\HG:H\to G$ of locally compact quantum group, we want to know when $C^r_0(G)^H$ and $\Ind_\HG \C$ exist and coincide 
	(see also the condition (d${}^+$) in \cref{rem:hmgquot}). 
	As another consequence of \cref{lem:indtensorcpt}, we record a criterion for their coincidence. 
	
	\begin{prop}\label{prop:hmgquot}
		Let $\HG:H\to G$ be a homomorphism of locally compact quantum groups and 
		suppose $C^r_0(G)^H$ exists. 
		If the restriction of $\Delta_G$ gives a left $G$-coaction on $C^r_0(G)^H$ which is continuous in the weak sense, 
		then $\Ind_\HG \C$ also exists and $C^r_0(G)^H=\Ind_\HG \C$. 
	\end{prop}
	
	\begin{proof}
		It is clear that $C^r_0(G)^H$ satisfies (i) and (ii) of \cref{def:ind} for $\Ind_\HG\C$. 
		We show (iii). 
		If we consider $C^r_0(G)^H$ as the induced coaction of the left $H^{\op}$-C*-algebra 
		$(C^r_0(G), \sigma(\Delta_G\rpb\HG))$ along $1_{H^{\op}\to 1}$, 
		by \cref{lem:indtensorcpt} we get 
		$\K(G)\otimes (C^r_0(G)^H)=(\K(G)\otimes C^r_0(G))^H$. 
		Here we have equipped $\K(G)$ with a trivial coaction. 
		Therefore $\M(\K(G)\otimes C^r_0(G))^H =\M(\K(G)\otimes C^r_0(G)^H)$  and on the unit balls these strict topologies coincide by (iii). 
		Now it is easy to see $\Ad V^G_{12}(-)_1: \WInd_\HG\C\to \M(\K(G)\otimes C^r_0(G))^H =\M(\K(G)\otimes C^r_0(G)^H)$ 
		is well-defined and continuous on the unit ball with respect to the strict topologies of $\M(\K(G))$ and $\M(\K(G)\otimes C^r_0(G)^H)$. 
	\end{proof}
	
	We do not know this coincidence holds in general, but it is true in many tractable cases. 
	
	\begin{eg}\label{eg:hmgquot}
		Let $\HG:H\to G$ be a homomorphism of locally compact quantum groups giving a closed quantum subgroup. 
			Then $C^r_0(G/H)=\Ind_{\HG}\C$ and $C^r_0(G)^H$ exist and coincide in the following cases. 
			\begin{itemize}
				\item
				When $G$ is a locally compact group by (1) of \cref{eg:formalind}. 
				\item
				Suppose $G,H$ are regular and $\HG$ gives an open quantum subgroup. 
				By the proof of Lemma~5.7 and Proposition~6.2 in~\cite{Kalantar-Kasprzak-Skalski:open}, 
				up to some unitary equivalence 
				$L^2(G)\isom \bigoplus\limits_{\lambda\in\Lambda}\HC_\lambda\otimes\HC'_\lambda$ for some Hilbert spaces $\HC_\lambda, \HC'_\lambda$, 
				it holds $C^r_0(G/H)\isom \bigoplus\limits_{\lambda\in\Lambda} 1_{\HC_\lambda}\otimes\K(\HC'_\lambda)$. 
				Therefore 
				$L^\infty(G/H)=\M(C^r_0(G/H))$, and the strict topologies with respect to 
				$\M(\K(G))$ and $\M(C^r_0(G/H))$ coincide on the unit balls. 
				Now it follows from the non-degenerate inclusion 
				$\M(C^r_0(G/H))\subset \M(C^r_0(G))^H\subset L^\infty(G/H)=\M(C^r_0(G/H))$ 
				that $C^r_0(G/H)$ satisfies the definition of $C^r_0(G)^H$. 
				Especially this holds when $G$ is discrete. 
				\item
				This is also the case when $H$ is a compact. 
				This is essentially proved in~\cite[Theorem~5.1]{Soltan:noncptaction}. 
				See also the remark after \cref{thm:cptind}. 
			\end{itemize}
	\end{eg}
	
	\section{Induced coactions along a proper homomorphism}\label{sec:proper}
	
	Next, we define proper homomorphisms of locally compact quantum groups, in a similar spirit as closed or open quantum subgroups. 
	
	\begin{defn}\label{def:proper}
		Let $G,H$ be locally compact quantum groups and $\HG:H\to G$ be a homomorphism. 
		We say $\HG$ is \emph{proper} if 
		$\hat{\HG}^r$ exists and $\hat{\HG}^r$ extends to a normal unital
		$*$-homomorphism 
		$L^{\infty}(\hat{H})\to L^{\infty}(\hat{G})$. 
		We still write $\hat{\HG}^r$ for this extension. 
	\end{defn}
	
	Let $\HG:H\to G$ be a proper homomorphism of locally compact quantum groups in the sense of \cref{def:proper} above. 
	Then we can define a von Neumann bialgebra 
	$F$ by $(L^{\infty}(F),\Delta_{F}):= ( \hat{\HG}^r(L^{\infty}(\hat{H})), \Delta_{\hat{G}} )$. 
	This is a well-defined locally compact quantum group by~\cite[Proposition~10.5]{Baaj-Vaes:doublecrossed}, 
	because $\hat{\HG}^r(L^{\infty}(\hat{H}))$ is closed under $R^{\hat{G}}$ and the scaling automorphism $\tau^{\hat{G}}_t$ by~\cite[Proposition~5.45]{Kustermans-Vaes:lcqg}. 
	By~\cite[Theorem~3.3]{Daws-Kasprzak-Skalski-Soltan:closed} and~\cite[Proposition~2.3]{Kalantar-Kasprzak-Skalski:open}, 
	$L^{\infty}(F)\subset L^{\infty}(\hat{G})$ identifies $\Ran\HG:=\hat{F}$ as a closed quantum subgroup of $G$, and 
	the surjection $\hat{\HG}^r:L^{\infty}(\hat{H})\to L^{\infty}(F)$ identifies $\hat{\Ran\HG}$ as an open quantum subgroup of $\hat{H}$
	(actually the similar argument as in~\cite[Proposition~2.3]{Kalantar-Kasprzak-Skalski:open} shows that 
	any homomorphism $L^\infty(\hat{H})\to L^\infty(\hat{G})$ of von Neumann bialgebras is $\hat{\HG}^r$ for some proper homomorphism $\HG:H\to G$). 
	Therefore we have obtained the decomposition of $\HG$ into 
	$\xymatrix{H\ar[r]^-{\bar{\HG}}& \Ran\HG\ar[r]^-{\ran\HG}& G}$, 
	where 
	$\bar{\HG}: H\to \Ran \HG$ is a homomorphism whose dual gives an open quantum subgroup, and  
	$\ran\HG: \Ran\HG\to G$ is a homomorphism giving a closed quantum subgroup. 
	This decomposition of $\HG$ is clearly unique (up to a natural isomorphism which makes the suitable diagram commute). 
	Moreover if $H$ is regular, then $\Ran\HG$ is also regular 
	since its dual is regular by the last remark of \cref{ssec:clopen}. 
	
	We also note the result~\cite[Theorem~3.2]{Kalantar-Kasprzak-Skalski-Soltan:openind} that 
	if $\HG:H\to G$ gives an open quantum subgroup, then $\hat{\HG}^{u}:C^u_0(\hat{H})\to C^u_0(\hat{G})$ is injective. 
	Therefore for a proper homomorphism $\HG:H\to G$, 
	it follows that $\bar{\HG}^u$ is injective. 
	Since it also holds that $(\ran\HG)^u(C^u_0(G))=C^u_0(\Ran\HG)$ by the remark after \cref{def:clopen},  
	we see $C^u_0(\Ran\HG)$ can be canonically identified with the image $\HG^u(C^u_0(G))\subset C^u_0(H)$. 
	
	Finally, we remark 
	it directly follows from~\cite[Theorem~4.2]{Daws-Kasprzak-Skalski-Soltan:closed} and~\cite[Corollary~7.4]{Kalantar-Kasprzak-Skalski:open} 
	via the decompositions of homomorphisms, 
	that when $G,H$ are locally compact groups, 
	a continuous group homomorphism $\HG:H\to G$ is proper in the sense of \cref{def:proper} 
	if and only if it is proper as a continuous map of topological spaces. 
	In this sense, \cref{def:proper} can be regarded as a quantum counterpart of properness (from a measure theoretic viewpoint). 
	See~\cite{Kasprzak-Khosravi-Soltan:integrable} for other aspects of this kind of homomorphisms. 
	
	\subsection{The general case}\label{ssec:properind}
	
	We prove the existence of induced coactions along a proper homomorphism of regular locally compact quantum groups. 
	First, we give a detailed description of induced coactions along the dual of an open quantum subgroup. 
	
	\begin{lem}\label{lem:propersurjind}
		Let $\HG:H\to G$ be a homomorphism of locally compact quantum groups with $\hat{\HG}$ giving an open quantum subgroup, 
		and $(A,\alpha)$ be a C*-algebra with a left $H$-coaction which is continuous in the weak sense. 
		Then $\Ind_\HG A$ exists. 
		Moreover, when we let $p=p^{\hat{\HG}}$ 
		be the coisometry associated with $\hat{\HG}:\hat{G}\to\hat{H}$ as in \cref{ssec:clopen} 
		and put 
		\[A^\HG:=\clin{(\K(G)^*\otimes \id_A)(p_1\alpha(A)p_1^*)} \subset \mult{A} , \] 
		then the followings hold. 
		\begin{enumerate}
			\item
			$A^\HG$ is a well-defined C*-algebra, and 
			$\alpha^\HG:=p_1\alpha(-)p_1^* :A^\HG\to \mult{C^r_0(G)\otimes A^\HG}$ is a well-defined left $G$-coaction on $A^\HG$ 
			which is continuous in the weak sense. 
			\item
			$\Ind_{\HG}A=\alpha^\HG (A^\HG) =p_1\alpha(A)p_1^*$, 
			and $(\HG^r\otimes\id_A)(\Ind_\HG A)=\alpha(A^\HG)$. 
			\item
			$(\alpha^\HG)^{-1}(p_1\alpha(-)p_1^*): A\to A^\HG$ is a well-defined conditional expectation. 
			\item
			$A^\HG=\{a\in A\,|\, \alpha(a)\in \M(\HG^rC^r_0(G)\otimes A)\}$. 
		\end{enumerate}
	\end{lem}
	
	In the proof, we write $P=p^*p\in \PC\ZC L^\infty (\hat{H})=\PC\ZC (L^\infty (\hat{H})')$. 
	We recall from \cref{ssec:clopen} that 
	\begin{align*}
		&[W^H,P\otimes P]=0 ,\quad
		(p\otimes p)W^H=W^{G}(p\otimes p) ,\quad
		p_2W^H=W^{\HG}p_2 ,\quad
		\mathrm{and}\quad p\hat{J}^H=\hat{J}^Gp.
	\end{align*}
	It follows from that 
	$[V^H,P\otimes P]=0$, 
	$(p\otimes p)V^H=V^G(p\otimes p)$ and $p_1 V^H= V^{\HG}p_1$, 
	because $V^H=(\hat{J}^H\otimes\hat{J}^H)\hat{W}^H(\hat{J}^H\otimes\hat{J}^H)$ 
	and so are $V^G$ and $V^{\HG}$. 
	
	\begin{proof}
		First, we note 
		$A^\HG$ is a non-degenerate closed $*$-linear subspace of $A=\clin{(\K(H)^*\otimes \id)\alpha(A)}$ 
		by definition. 
		Since $p_1(-)p_1:C^r_0(H)\to C^r_0(G)\isom \HG^rC^r_0(G)$ is a conditional expectation, it holds that 
		$p_1\alpha(-)p_1^*:A\to \mult{C^r_0(G)\otimes A}$ is a well-defined non-degenerate c.c.p.~map. 
		We also have as maps from $A$ to 
		$\mult{\K(G)\otimes C^r_0(H)\otimes A}$ that 
		\begin{align}\label{prf:lem:propersurjind1}
			&
			\begin{aligned}
				&
				(\id\otimes \alpha)\circ(p_1\alpha(-)p_1^*)
				= p_1 V^{H}_{12} \alpha(-)_{13} V^{H}_{12}{}^* p_1^*
				= V^{\HG}_{12} p_1 \alpha(-)_{13} p_1^* V^{\HG}_{12}{}^* 
				\\&
				= (\id\otimes \HG^{r}\otimes \id)(\Delta_{G}\otimes \id)(p_1\alpha(-)p_1^*) . 
			\end{aligned}
		\end{align}
		It follows 
		$\alpha(A^\HG) 
		\subset \alpha(A)\cap 
		\mult{\HG^{r}C^r_0(G)\otimes A}$. 
		Next, we show $p_1\alpha(A)p_1^* \subset p_1\alpha(A^\HG)p_1^*$, which will imply 
		\begin{align}\label{prf:lem:propersurjind2}
			p_1\alpha(A^\HG)p_1^* 
			= p_1\left(\alpha(A)\cap \mult{\HG^{r}C^r_0(G)\otimes A}\right)p_1^* 
			= p_1\alpha(A)p_1^* . 
		\end{align}
		Since $\B(L^2(H))_*\restriction_{\HG^{r}L^{\infty}(G)} = \B(L^2(G))_*\restriction_{L^{\infty}(G)}$ 
		by the injectivity and normality of $\HG^{r}$, we get 
		\begin{align*}
			&
			\clin{p_1 \alpha(A) p_1^*}
			=\clin{(\K(H)^*\otimes \id\otimes\id)\left(
				p_2 W^{H}_{12}{}^* \alpha(A)_{23} W^{H}_{12} p_2^* \right)}
			\\&
			=\bigl[(\K(H)^*\otimes \id\otimes\id)\bigl(
				W^{\HG}_{12}{}^* p_2 \alpha(A)_{23} p_2^* W^{\HG}_{12} \bigr)\bigr]
			\\&
			=\clin{(\K(G)^*\otimes \id\otimes\id)\left(
				W^{G}_{12}{}^* p_2 \alpha(A)_{23} p_2^* W^{G}_{12} \right)}
			\\&
			=
			\clin{(\K(G)^*\otimes \id\otimes\id)\left(
				(p\otimes p)_{12} W^{H}_{12}{}^* \alpha(A)_{23} W^{H}_{12} (p\otimes p)_{12}^*\right)}
			\\&
			=\clin{(\K(G)^*\otimes \id\otimes\id)\bigl(
				(p\otimes p)_{12} \bigl((\id\otimes \alpha)\alpha(A)\bigr) (p\otimes p)_{12}^*\bigr)}
			= \clin{p_1\alpha(A^{\HG})p_1^*} . 
		\end{align*}
		Here we note 
		$p_1(-)p_1^*$ is an injective $*$-homomorphism on $\M(\HG^rC^r_0(G)\otimes A)$, which contains $\alpha(A^\HG)$. 
		Thus 
		$p_1\alpha(-)p_1$ is an isometry on $A^\HG$, 
		and we have 
		$\clin{p_1\alpha(A^{\HG})p_1^*}
		=p_1\alpha(A^{\HG})p_1^*$. 
		Therefore \cref{prf:lem:propersurjind2} follows. 
		By the second part of \cref{prf:lem:propersurjind2} 
		we see $A^\HG$ is actually a C*-algebra. 
		It also follows that $\alpha^\HG:=p_1\alpha(-)p_1$ 
		is a well-defined injective non-degenerate $*$-homomorphism of 
		$A^\HG\to\M(C^r_0(G)\otimes A)$. 
		Now it is easy to see (3) and (4). 
		
		We show (1). 
		By composing $p_2(-)p_2^*$ after \cref{prf:lem:propersurjind1}, 
		it follows $\alpha^\HG:A^{\HG} \to \mult{C^r_0(G)\otimes A^{\HG}}$ is a well-defined left $G$-coaction on $A^{\HG}$. 
		This is continuous in the weak sense since by using \cref{prf:lem:propersurjind2} we get 
		\begin{align*}
			&
			\clin{(\K(G)^*\otimes\id)(p_1\alpha(A^\HG)p_1^*)}
			=
			\clin{(\K(G)^*\otimes\id)(p_1\alpha(A)p_1^*)}
			=
			A^\HG . 
		\end{align*}
		
		Now it suffices to show that $p_1\alpha(A^{\HG})p_1^*$ satisfies \cref{def:ind} for $\Ind_{\HG}A$. 
		(i) and (ii) follow from \cref{prf:lem:propersurjind1} and (1). 
		We show (iii). 
		Since on $\WInd_{\HG} A$ we have 
		\begin{align}\label{prf:lem:propersurjind3}
			\Ad V^G_{12}(-)_{13} =p_2(\Ad V^{\HG}_{12}(-)_{13})p_2^* =p_2\bigl((\id\otimes \alpha)(-)\bigr)p_2^* , 
		\end{align}
		it holds 
		\begin{align*}
			&
			\bigl[(\K(G)^*\otimes \id)(\WInd_{\HG} A)\bigr]
			=
			\bigl[(\K(G)^*\otimes \K(G)^*\otimes \id)\bigl(
				V^G_{12}(\WInd_{\HG} A)_{13}V^G_{12}{}^* \bigr)\bigr]
			\\&=
			\bigl[(\K(G)^*\otimes \K(G)^*\otimes \id)\bigl(
				p_2\bigl((\id\otimes \alpha)\WInd_{\HG} A\bigr)p_2^* \bigr)\bigr]
			\\&\subset
			\clin{(\K(G)^*\otimes \id)\bigl(
				p_1 (\alpha\mult{A})p_1^* \bigr)}
			\subset\mult{A^\HG}, 
		\end{align*} 
		where for the last inclusion we have used the fact by (3) that 
		$p_1 \alpha(-)p_1^*:\mult{A}\to \mult{\alpha^\HG(A^\HG)}$ 
		is a well-defined c.p.~map which is strictly continuous on the unit ball. 
		It follows $\WInd_{\HG} A\subset \mult{\K(G)\otimes A^\HG}$ because 
		$A^\HG\subset A$ is a non-degenerate C*-subalgebra. 
		Again by \cref{prf:lem:propersurjind3}, 
		the map 
		$\Ad V^{G}_{12}(-)_{13}:\WInd_\HG A\to \M(\K(G)\otimes \alpha^\HG(A^\HG))$ 
		is well-defined with the desired continuity. 
	\end{proof}
	
	Now we have all pieces to show the existence of induced coactions along a proper homomorphism. 
	
	\begin{thm}\label{thm:properind}
		Let $\HG:H\to G$ be a proper homomorphism of regular locally compact quantum groups. 
		Then $\Ind_\HG A$ exists for any left $H$-C*-algebra $(A,\alpha)$. 
	\end{thm}
	
	\begin{proof}
		We decompose $\HG=(\ran\HG)\bar{\HG}$ as in the remark after \cref{def:proper}. 
		Note that $\Ran\HG$ is also regular. 
		Then by \cref{thm:indstages}, the statement is a direct consequence of~\cite[Theorem~7.2]{Vaes:impr} for $\ran\HG$ 
		and \cref{lem:propersurjind} for $\bar{\HG}$. 
	\end{proof}
	
	Next, we record miscellaneous properties of induced coactions along 
	a proper homomorphism $\HG:H\to G$ of regular locally compact groups, 
	which will be used later. 
	Again we decompose $\HG=(\ran\HG)\bar{\HG}$. 
	We recall that $I^{\ran\HG}$ is the 
	$([C^r_0(\hat{G})C^r_0(G/\Ran\HG)],C^r_0(\hat{\Ran\HG}))$-imprimitivity bimodule 
	as in the beginning of \cref{sec:formal}. 
	
	\begin{cor}\label{cor:properind}
		Let $\HG:H\to G$ be a proper homomorphism of regular locally compact quantum groups, 
		$(A,\alpha)$ be a left $H$-C*-algebra, 
		and $B$ be a C*-algebra (with a trivial coaction). 
		When we put $p:=p^{\hat{\bar{\HG}}}$ for the coisometry associated with $\hat{\bar{\HG}}$, 
		the followings hold as subsets of $\mult{\K(G)\otimes A}$ or $\mult{\K(G)\otimes A\otimes B}$. 
		\begin{enumerate}
			\item
			$\Ind_{\HG}A
			=\Ind_{\ran \HG}A^{\bar{\HG}}
			=\Ind_{\HG}\bar{\HG}^* A^{\bar{\HG}}$. 
			\item
			$\clin{C^r_0(G)_1 \Ind_\HG A}= C^r_0(G)\otimes A^{\bar{\HG}}$. 
			\item
			$\bigl[C^r_0(\hat{G})_1 \Ind_\HG A\bigr]
			=\bigl[I^{\ran\HG}_1p_1 \alpha(A) p_1^* I^{\ran\HG}_1{}^*\bigr]$. 
			\item
			$\Ind_{\HG}(A\otimes B)=(\Ind_{\HG}A)\otimes B$. 
		\end{enumerate}
	\end{cor}
	
	\begin{proof}
		As for (1), 
		by \cref{thm:indstages} and \cref{lem:propersurjind} we have 
		\begin{align*}
			&
			\Ad V^{\ran \HG}_{12}(\Ind_{\HG}A)_{13}
			=\Ind_{\ran \HG}\Ind_{\bar{\HG}}A 
			= \Ind_{\ran \HG}\alpha^{\bar{\HG}}(A^{\bar{\HG}})
			=\Ad V^{\ran \HG}_{12}(\Ind_{\ran \HG}A^{\bar{\HG}})_{13} , 
		\end{align*}
		and it follows $\Ind_{\HG}A
		=\Ind_{\ran \HG}A^{\bar{\HG}}$. 
		Thus $\Ind_{\HG}\bar{\HG}^* A^{\bar{\HG}} = \Ind_{\ran \HG} (\bar{\HG}^* A^{\bar{\HG}})^{\bar{\HG}} = \Ind_{\ran \HG} A^{\bar{\HG}}$
		by (4) of \cref{lem:propersurjind}. 
		
		Next, we show (2). 
		By (1) and regularity of $V^G$, we have 
		\begin{align*}
			&
			\clin{C^r_0(G)_1 \Ind_\HG A} 
			= \bigl[C^r_0(G)_1 \Ind_{\ran \HG}A^{\bar{\HG}}\bigr] 
			\\&
			= \bigl[(\K(G)^*\otimes \id\otimes\id)\bigl(
				C^r_0(G)_2 V^G_{12} (\Ind_{\ran \HG}A^{\bar{\HG}})_{13} V^G_{12}{}^* \bigr)\bigr] 
			\\&
			= \bigl[(\K(G)^*\otimes \id\otimes\id)\bigl(
				C^r_0(G)_2 (\Ind_{\ran \HG}A^{\bar{\HG}})_{13} \bigr) \bigr]
			\\&
			= \bigl[ C^r_0(G) \odot \bigl((\K(G)^*\otimes \id)\Ind_{\ran \HG}A^{\bar{\HG}}\bigr)\bigr] . 
		\end{align*}
		
		Thus it is enough to show 
		$\clin{(\K(G)^*\otimes \id_A)\Ind_{\HG}A}
		=A$ when $\HG$ gives a closed quantum subgroup. 
		Here we note that when $\HG$ gives a closed quantum subgroup, 
		$\Ran \HG$ is canonically identified with $H$ itself and $A^{\bar{\HG}}=A^{\id_H}=A$. 
		By using \cref{eq:recallclosedind1}, we see 
		\begin{align*}
			&
			\clin{(\K(G)^*\otimes \id) 
				\Ind_{\HG} A }
			=
			\bigl[(\K(G)^*\otimes \id) 
				(C^r_0(\hat{G})_1 \Ind_{\HG} A) \bigr]
			\\&=
			\bigl[(\K(G)^*\otimes \id) 
				\bigl(I^{\HG}_1\alpha(A) I^{\HG}_1{}^* \bigr)\bigr]
			=
			\clin{(\K(H)^*\otimes \id) \alpha(A)}
			=A . 
		\end{align*}
		
		(3) holds since 
		\begin{align*}
			&
			\bigl[C^r_0(\hat{G})_1 \Ind_\HG A\bigr]
			=\bigl[C^r_0(\hat{G})_1 \Ind_{\ran\HG} A^\HG\bigr]
			=\bigl[I^\HG_1 p_1 \alpha(A^\HG) p_1^* I^\HG_1{}^*\bigr]
			=\bigl[I^\HG_1 p_1 \alpha(A) p_1^* I^\HG_1{}^*\bigr] . 
		\end{align*}
		
		Now it is easy to see (4) by taking 
		$\bigl[(\B(L^2(\hat{G}))_*\otimes\id\otimes\id\otimes\id)\bigl(\Ad V^G_{12}(-)_{134}\bigr)\bigr]$ to 
		the following equality assured by (3), 
		$\bigl[C^r_0(\hat{G})_1\Ind_\HG (A\otimes B)\bigr]
		=\bigl[C^r_0(\hat{G})_1\Ind_\HG A\bigr]\otimes B$. 
	\end{proof}
	
	\subsection{The compact case}\label{ssec:cptind}
	
	We give a description of induced coactions along a homomorphism $\HG:K\to G$ from a compact quantum group $K=H$ to a locally compact quantum group $G$ by using cotensor products. 
	Although the content of this section might be well-known to the experts, we contain it since we could not find appropriate references. 
	First, we fix some notation about compact quantum groups. 
	For details, see~\cite{Neshveyev-Tuset:book} for example. 
	Then we introduce cotensor products of C*-algebras by following~\cite{Podles:quantumsymmetry},~\cite{deRijdt-vanderVennet:moneq},~\cite{Voigt:bcfo}. 
	Especially, we closely follow the notation from Section~2 of~\cite{Voigt:bcfo}. 
	
	Consider a compact quantum group $K$ 
	with its faithful Haar state $h$ on $C^r_0(K)$. 
	A pair $\rho=(\HC^\rho,u^\rho)$ of a finite dimensional Hilbert space $\HC^\rho$
	and a unitary matrix $u^\rho\in \U(\K(\HC^\rho)\otimes C^r_0(K))$ 
	is called a finite dimensional unitary (right) representation of $K$ if 
	when we express $u^\rho=(u^\rho_{i,j})_{i,j}$ with respect to some orthonormal basis of $\HC^\pi$, it holds 
	$\Delta_K(u^\rho_{i,j})
	=\sum\limits_{k=1}^{\dim\HC^\rho}u^\rho_{i,k}u^\rho_{k,j}$ for all $i,j$. 
	We write $\Rep(K)$ for the C*-tensor category of finite dimensional unitary representations, 
	$\Irr(K)\subset \Rep(K)$ for the set of representatives of isomorphism classes of irreducible objects, 
	and $\1=\1_K\in \Irr(K)$ for the trivial $1$-dimensional representation of $K$. 
	Any $\rho\in\Rep(K)$ can be written as a direct sum of elements in $\Irr(K)$, and for $\pi\in\Irr(K)$ 
	we write $\pi\leq \rho$ if $\rho$ contains $\pi$ as a direct summand. 
	For each $\pi\in\Irr(K)$, we define the positive definite matrix $F^\pi=(F^\pi_{i,j})_{i,j}\in\B(\HC^\pi)$
	as the unique invertible intertwiner from the representation $\pi$ to its double contragredient representation $\pi^{cc}$ (which need not be unitary) 
	with $\mathrm{tr}F^\pi=\mathrm{tr}((F^\pi)^{-1})$. 
	Then the quantum dimension of $\pi$ is denoted by $\dim_q \pi:=\mathrm{tr}F^\pi$. 
	As for $F^\pi$, all we need here is the following orthogonality relation. 
	For all $\pi,\varpi\in\Irr(K)$, $1\leq i, j \leq\dim\HC^\pi$, and $1\leq k,l \leq\dim\HC^\varpi $, it holds 
	\begin{align*}
		h(u^\pi_{i,j} u^\varpi_{k,l}{}^*)
		&=
		(\dim_q\pi)^{-1}\delta_{\pi,\varpi} \delta_{i,k} F^\pi_{l,j} . 
	\end{align*} 
	Hereafter for each $\pi\in\Irr(K)$, we choose an orthonormal basis 
	$(e^\pi_i)_i$ of $\HC^\pi$ such that it diagonalizes $(F^\pi_{i,j})_{i,j}$. 
	We define $\OC(K):=\bigoplus\limits_{\pi\in\Irr(K)}\OC(K)_{\pi}\subset C^r_0(K)$, 
	where 
	$\OC(K)_\rho:=(\B(\HC^\rho)_* \otimes \id)(u^\rho)$ for each $\rho\in\Rep(K)$. 
	It is a consequence of Peter--Weyl theorem 
	that $\OC(K)$ is a norm-dense Hopf $*$-subalgebra of $C^r_0(K)$ 
	and 
	$(u^\pi_{i,j}\,|\,\pi\in\Irr(K),1\leq i,j\leq\dim\HC^\pi)$ forms its linear basis. 
	
	Fix a right $K$-C*-algebra $(A,\alpha)$. 
	Its spectral subspace of each $\rho\in \Rep(K)$ is denoted by 
	$A_\rho:=\{a\in A \,|\, \alpha(a)\in A\odot\OC(K)_\rho \}$. 
	For $a\in A_\rho$, we take the elements $a_{i,j}^\pi\in A$ for 
	$\pi\in \Irr(K)$ with $\pi\leq \rho$ and $1\leq i,j\leq\dim\HC^\pi$ 
	such that 
	$\alpha(a) = \sum\limits_{\rho\geq \pi\in\Irr(K)} \sum\limits_{i,j=1}^{\dim\HC^\pi}a^\pi_{i,j}\otimes u^\pi_{j,i}$. 
	We observe 
	\begin{align*}
		\sum\limits_{\rho\geq \pi\in\Irr(K)}
		\sum\limits_{i,j=1}^{\dim\HC^\pi}\alpha(a^\pi_{i,j})\otimes u^{\pi}_{j,i}
		=\sum\limits_{\rho\geq \pi\in\Irr(K)}
		\sum\limits_{i,j,k=1}^{\dim\HC^\pi}a^\pi_{i,k}\otimes u^{\pi}_{k,j} \otimes u^{\pi}_{j,i} . 
	\end{align*}
	Then we see that 
	$\alpha(a^\pi_{i,j})=\sum\limits_{k=1}^{\dim\HC^\pi}a^\pi_{i,k}\otimes u^\pi_{k,j}$, 
	and 
	$a=\sum\limits_{\rho\geq \pi\in\Irr(K)} \sum\limits_{i=1}^{\dim\HC^\pi}a^\pi_{i,i}$, 
	by the linear independence of $\{u^\pi_{i,j}\}_{\pi,i,j}$ and the injectivity of $\alpha$. 
	Especially $a^\pi_{i,j}\in A_\pi$. 
	
	It holds that $(\id_A\otimes h)\alpha:A\twoheadrightarrow A_{\1}$ is a conditional expectation onto the non-degenerate C*-subalgebra $A_{\1}\subset A$. 
	More generally for any $\rho\in\Rep(G)$, 
	a calculation using $a^\pi_{i,j}$ shows that 
	the continuous $A_{\1}$-bimodule map defined by 
	\begin{align}\label{eq:coten1}
		&
		p^A_\rho:A\ni a\mapsto 
		\sum\limits_{\rho\geq\pi\in\Irr(K)}
		\sum\limits_{i=1}^{\dim\HC^\pi} \frac{\dim_q\pi}{F^\pi_{i,i}}
		(\id_A \otimes h)(\alpha(a)(1_A\otimes u^{\pi}_{i,i}{}^*)) 
		\in A, 
	\end{align}
	satisfies $p^A_\rho\restriction_{A_\rho}=\id_{A_\rho}$ and $p^A_\rho(A)=A_\rho$. 
	It turns out that the $*$-subalgebra 
	$\bigoplus\limits_{\pi\in\Irr(K)}A_\pi \subset A$ is dense 
	since $\clin{(\id\otimes h)(\alpha(A)\OC(K)_2)}=A$. 
	
	Similar things hold for a left $K$-C*-algebra $(B,\beta)$. 
	For $\rho\in\Rep(K)$, we let 
	${}_\rho B:=\{b\in B \,|\, \beta(b)\in \OC(K)_\rho\odot B \}$. 
	For $b\in {}_\rho B$, 
	the elements $b^\pi_{i,j}\in {}_\pi B$ with 
	$\beta(b) = \sum\limits_{\rho\geq \pi\in\Irr(K)}
	\sum\limits_{i,j=1}^{\dim\HC^\pi}u^{\pi}_{j,i}\otimes b^\pi_{i,j}$, 
	also satisfy 
	$\beta(b^\pi_{i,j})=\sum\limits_{k=1}^{\dim\HC^\pi}u^\pi_{i,k}\otimes b^\pi_{k,j}$ and 
	$b=\sum\limits_{\rho\geq \pi\in\Irr(K)} \sum\limits_{i=1}^{\dim\HC^\pi}b^\pi_{i,i}$. 
	The $*$-subalgebra 
	$\bigoplus\limits_{\pi\in\Irr(K)}{}_\pi B \subset B$ is dense. 
	
	Now we define the cotensor product by 
	$A\cotensor{K}B 
	:= 
	\Bigl[\bigoplus\limits_{\pi\in\Irr(K)}A\cotensor{\pi}B\Bigr]
	\subset A\otimes B$, 
	where we let 
	\begin{align*}
		&
		A\cotensor{\rho}B:= 
		\bigl\{ x\in A\otimes B \,\bigl|\, (\alpha\otimes\id_B)(x)=(\id_A\otimes\beta)(x)\in \bigl((A\otimes B)\odot \OC(K)_\rho\bigr)_{132} \bigr\}, 
	\end{align*}
	for each $\rho\in\Rep(K)$. By a similar calculation as that for $p^A_\rho$, we can show the following lemma. 
	\begin{lem}\label{lem:cotenproj}
		Let $K$ be a compact quantum group, 
		$(A,\alpha)$ be a right $K$-C*-algebra, $(B,\beta)$ be a left $K$-C*-algebra, and $\rho\in \Rep(K)$. 
		For any $x\in A\otimes B$, we define 
		$r^{A,B}_\rho(x) =r_\rho(x)\in A\otimes B$ 
		by 
		\begin{align*}
			&
			r_{\rho}(x):= 
			\sum\limits_{\rho\geq\pi\in\Irr(K)}
			\sum\limits_{i=1}^{\dim\HC^\pi}\frac{\dim_q \pi}{F^\pi_{i,i}}
			(\id_A\otimes h\otimes h\otimes \id_B)
			\Bigl(\bigl((\alpha\otimes\beta)(x)\bigr) (1_A\otimes \Delta_K(u^\pi_{i,i})^* \otimes 1_B) \Bigr). 
		\end{align*} 
		Then $r_\rho
		:A\otimes B\to A\cotensor{\rho}B$ 
		is a well-defined continuous surjective idempotent $A_{\1} \otimes {}_{\1} B$-bimodule map. 
	\end{lem}
	
	\begin{proof}
		It is clear that $r_\rho : A\otimes B\to A\otimes B$ is a well-defined continuous $A_{\1} \otimes {}_{\1} B$-bimodule map. 
		It holds $r_\rho=\id$ on $A\cotensor{\rho}B$, because 
		for any $x\in A\cotensor{\rho}B$ we have 
		$(\alpha\otimes \beta)(x)
		= (\id\otimes\Delta_K\otimes\id)(\alpha\otimes\id)(x)$ 
		and thus 
		\begin{align*}
			&
			r_\rho(x)=\sum\limits_{\rho\geq\pi\in\Irr(K)}
			\sum\limits_{i=1}^{\dim\HC^\pi}\frac{\dim_q \pi}{F^\pi _{i,i}}
			\bigl(\id\otimes ((h\otimes h)\Delta_K)\otimes \id\bigr)
			\Bigl(\bigl((\alpha\otimes\id)(x)\bigr) (1\otimes u^\pi_{i,i}{^*} \otimes 1) \Bigr)
			\\&
			= p^{A\otimes 1^*B}_\rho (x)
			= x , 
		\end{align*}
		where $A\otimes 1^*B$ is the right $K$-C*-algebra $(A\otimes B,\sigma_{12}(\id\otimes\alpha)\sigma)$. 
		For all $\pi, \varpi\in \Irr(G)$ with either $\pi\neq \varpi$ or $\pi\not\leq\rho$, 
		we see $r_\rho(A_{\pi}\otimes {}_{\varpi} B)=0$. 
		For any $a\in A_\rho$ and $b\in {}_\rho B$, by using $a^\pi_{i,j}$ and $b^\pi_{i,j}$ we calculate 
		\begin{align*}
			& r_\rho (a\otimes b) 
			= \sum\limits_{\rho\geq \pi\in\Irr(K)} \sum\limits_{i,j,k,l,m,n=1}^{\dim\HC^\pi}\frac{\dim_q \pi}{F^\pi _{i,i}}
			(\id\otimes h\otimes h\otimes \id)
			\Bigl((a^\pi_{l,k}\otimes u^\pi_{k,l}\otimes u^\pi_{n,m}\otimes b^\pi_{m,n}) 
			(u^\pi_{i,j}{^*} \otimes u^\pi_{j,i}{^*})_{23} \Bigr)
			\\
			& = \sum\limits_{\rho\geq \pi\in\Irr(K)} \sum\limits_{i,j=1}^{\dim\HC^\pi}\frac{F^\pi _{j,j}}{\dim_q \pi}
			a^\pi_{j,i}\otimes b^\pi_{i,j}. 
		\end{align*}
		Thus we get 
		$r_\rho (A_\rho\odot {}_\rho B)\subset A\cotensor{\rho}B$. 
		Therefore it holds $r_\rho (A\otimes B)\subset A\cotensor{\rho}B$ 
		by the continuity of $r_\rho$ and the density of 
		$\bigoplus\limits_{\pi\in\Irr(K)}A_\pi \odot \bigoplus\limits_{\varpi\in\Irr(K)}{}_\varpi B \subset A\otimes B$, 
		which completes the proof. 
	\end{proof}
	
	\begin{prop}\label{prop:coten}
		Let $G$ be a locally compact quantum group, $K$ be a compact quantum group, 
		$(A,\alpha)$ be a right $K$-C*-algebra 
		and $(B,\beta)$ be a left $K$-C*-algebra. 
		Suppose $\gamma$ is a left $G$-coaction on $A$ 
		which is continuous in the strong sense, such that 
		$(\gamma\otimes\id)\alpha=(\id\otimes\alpha)\gamma$. 
		Then $A\cotensor{K}B \subset A\otimes B$ 
		is a non-degenerate C*-subalgebra, and 
		$(A\cotensor{K}B , \gamma\otimes\id)$ is a well-defined left $G$-C*-algebra. 
	\end{prop}
	
	\begin{proof}
		First, we note $A_{\1}\otimes {}_{\1}B\subset A\otimes B$ is a non-degenerate C*-subalgebra. 
		We can check 
		\begin{align*}
			&
			A\cotensor{K}B 
			=
			\bigl[\bigl\{x\in A\otimes B \,\bigl|\, (\alpha\otimes\id)(x)=(\id\otimes\beta)(x)\in \bigl((A\otimes B)\odot\OC(K)\bigr)_{132} \bigr\}\bigr]
			\supset A_{\1}\otimes {}_{\1}B 
		\end{align*}
		by \cref{lem:cotenproj}. 
		It follows that $A\cotensor{K}B \subset A\otimes B$ 
		is a non-degenerate C*-subalgebra. 
		
		For each $\rho\in\Rep(K)$ it holds 
		$(\gamma\otimes\id)r_\rho=(\id\otimes r_\rho)(\gamma\otimes\id)$ 
		on $A\otimes B$ by definition, 
		and we see 
		\begin{align*}
			&
			\bigl[ C^r_0(G)_1 \bigl((\gamma\otimes \id)(A\cotensor{\rho}B)\bigr) \bigr]
			=
			\clin{C^r_0(G)_1 \bigl((\id\otimes r_\rho)(\gamma\otimes\id) (A\otimes B)\bigr)}
			\\&
			=
			\clin{(\id\otimes r_\rho)(C^r_0(G) \otimes A\otimes B)}
			= 
			\bigl[C^r_0(G)\odot (A\cotensor{\rho}B)\bigr] . 
		\end{align*}
		Thus it follows $\gamma\otimes\id:A\cotensor{K}B\to \mult{C^r_0(G)\otimes(A\cotensor{K}B)}$ 
		is a well-defined non-degenerate $*$-homomorphism 
		which gives a continuous left $G$-coaction. 
	\end{proof}
	
	\begin{thm}\label{thm:cptind}
		Consider a homomorphism $\HG:K\to G$ of locally compact quantum groups and assume $K$ is compact. Let $(A,\alpha)$ be a left $K$-C*-algebra. 
		Then $\Ind_\HG A$ exists and is given by $C^r_0(G)\rpb\HG\cotensor{K}A$. 
		Its left $G$-coaction $\Delta_G\otimes\id$ is continuous in the strong sense.  
	\end{thm}
	
	Especially, by letting $A=\C$ we get the last case of \cref{eg:hmgquot}. 
	Note that we do not need regularity on $G$. 
	We also note that a homomorphism $\HG:K\to G$ of locally compact quantum groups is always proper when $K$ is compact, because $C^r_0(\hat{K})$ is a direct sum of finite dimensional C*-algebras. 
	
	\begin{proof}
		By \cref{prop:coten}, $C^r_0(G)\cotensor{K}A$ satisfies (i) and (ii) of \cref{def:ind} for $\Ind_\HG A$, and its left $G$-coaction is continuous. 
		We are going to show 
		$(\Delta_G\otimes\id)\WInd_{\HG}A \subset \mult{\K(G)\otimes (C^r_0(G)\cotensor{K}A)}$, 
		from which (iii) follows because 
		$\Ad V^{G}_{12}(-)_{13}:\mult{\K(G)\otimes A}\to \mult{\K(G)\otimes C^r_0(G)\otimes A}$ 
		is strictly continuous and 
		$C^r_0(G)\cotensor{K}A\subset C^r_0(G)\otimes A$ is a non-degenerate C*-subalgebra. 
		
		Via the identification 
		$L^2(G)=\bigoplus\limits_{\pi\in\Irr(K)}(\HC^{\pi})^{\oplus\dim\HC^\pi}$
		by Peter--Weyl theorem, 
		for each $\pi\in\Irr(K)$ and $1\leq i,j\leq \dim\HC^\pi$ 
		we put 
		$E^\pi_{i,j}:= \left(e^\pi_{i} \bra e^\pi_{j} , -\ket \right)^{\oplus\dim\HC^\pi} \in \K(K)$. 
		Then we can write 
		\begin{align*}
			&
			V^K=
			\biggl(\sum\limits_{i,j=1}^{\dim\HC^\pi}E^\pi_{i,j}\otimes u^\pi_{i,j} 
			\,\biggl|\, \pi\in\Irr(K)\biggr) 
			\in\UM\left(\hat{J}^{K}C^r_0(\hat{K})\hat{J}^{K}\otimes C^r_0(K)\right) . 
		\end{align*}
		For a finite subset $\Pi\subset \Irr(K)$ we let 
		$P^\Pi:=
		\sum\limits_{\pi\in\Pi}\sum\limits_{i=1}^{\dim\HC^\pi}E^{\pi}_{i,i} \in \PC\ZC (\hat{J}^{K}C^r_0(\hat{K})\hat{J}^{K})$. 
		Now we recall the homomorphism of C*-bialgebras $\hat{\HG}^r{}':
		\hat{J}^KC^r_0(\hat{K})\hat{J}^K\to \M(\hat{J}^G C^r_0(\hat{G})\hat{J}^G)$
		with $V^{\HG} = (\hat{\HG}^r{}'\otimes\id)(V^K)$. 
		Then for any $x\in \WInd_\HG A$ it holds 
		$\hat{\HG}^r{}'(P^{\Pi})_1 x \hat{\HG}^r{}'(P^{\Pi})_1
		\xrarr{\Pi\subset \Irr(K)} x$ strictly in $\mult{\K(G)\otimes A}$. 
		Moreover we have 
		\begin{align}\label{prf:thm:cptind1}
			&
			\begin{aligned}
				&
				\Ad V^{\HG}_{23}\Ad V^{G}_{12} \left( \hat{\HG}^r{}'(P^{\Pi})_1 x_{14} \hat{\HG}^r{}'(P^{\Pi})_1 \right)
				=
				\Ad V^{G}_{12} \Ad V^{\HG}_{13} \left( \hat{\HG}^r{}'(P^{\Pi})_1 x_{14} \hat{\HG}^r{}'(P^{\Pi})_1 \right)
				\\&
				=
				V^{G}_{12} 
				\left( (\hat{\HG}^r{}' \otimes\id) (P^{\Pi}_1V^K) \right)_{13} 
				x_{14} 
				\left( (\hat{\HG}^r{}' \otimes\id) (V^K{}^*P^{\Pi}_1) \right)_{13} 
				V^{G}_{12}{}^*
				\\&
				=
				V^{G}_{12} 
				\hat{\HG}^r{}' (P^{\Pi})_1 
				\left( (\id\otimes\alpha)x \right)_{134} 
				\hat{\HG}^r{}' (P^{\Pi})_1 
				V^{G}_{12}{}^*
				=
				(\id\otimes\id\otimes\alpha)\Ad V^{G}_{12} (\hat{\HG}^r{}'(P^{\Pi})_1 x_{13} \hat{\HG}^r{}'(P^{\Pi})_1) . 
			\end{aligned}
		\end{align}
		Put $\rho:=\bigoplus\limits_{\pi\in\Pi}\pi\in\Rep(K)$ 
		and $\tilde{\rho}:=\rho\,\otop\,\overline{\rho}\in \Rep(K)$, 
		where $\overline{\rho}\in\Rep(K)$ is the unitary conjugate of $\rho$. 
		Since $P^\Pi_1 V^K \in \M(\K(G))\otimes \OC(K)_\rho$, \cref{prf:thm:cptind1} is an element in 
		$\bigl(\M(\K(G)\otimes C^r_0(G)\otimes A) \odot \OC(K)_{\tilde{\rho}}\bigr)_{1243}$. 
		Thus for any $c\in\K(G)\otimes (C^r_0(G)_{\1})$ and $a\in{}_{\1}A$, it holds that 
		\begin{align*}
			& 
			(c\otimes a)V^{G}_{12} \bigl( \hat{\HG}^r{}'(P^{\Pi})_1 x \hat{\HG}^r{}'(P^{\Pi})_1 \bigr)_{13} V^{G}_{12}{}^*
			\in
			\bigl(\K(G)\otimes C^r_0(G)\bigr)\cotensor{\tilde{\rho}} A
			=
			\bigl[\K(G)\odot\bigl(C^r_0(G)\cotensor{\tilde{\rho}} A\bigr)\bigr] , 
		\end{align*}
		where we have equipped $\K(G)$ with a trivial coaction, 
		and used $r_{\tilde{\rho}}^{\K(G)\otimes C^r_0(G),\, A}
		=\id\otimes r_{\tilde{\rho}}^{C^r_0(G),\, A}$. 
		It follows from the strict continuity of 
		$\Ad V^{G}_{12}(-)_{13}:\mult{\K(G)\otimes A}\to \mult{\K(G)\otimes C^r_0(G)\otimes A}$ that 
		$(c\otimes a) V^{G}_{12} x_{13} V^{G}_{12}{}^* \in 
		\K(G)\otimes(C^r_0(G)\cotensor{K} A)$, 
		and thus $\Ad V^{G}_{12} x_{13} \in 
		\mult{\K(G)\otimes (C^r_0(G)\cotensor{K} A)}$ 
		by the non-degeneracy of $C^r_0(G)_{\1}\otimes {}_{\1}A\subset C^r_0(G)\cotensor{K}A$. 
	\end{proof}
	
	\section{Analogue of base change}\label{sec:basechange}
	
	We go back to the case of induced coactions along general proper homomorphisms. 
	From now we restrict our attention to regular locally compact quantum groups to assure the existence of induced coactions. 
	Throughout this section, we consider the following situation, 
	which changes quantum groups coacting on C*-algebras. 
	
	\begin{condition}\label{cond:square}
		Consider the following diagram of regular locally compact quantum groups and homomorphisms, 
		\begin{align}\label{diagram:square}
			\vcenter{\xymatrix{
				F \ar[r]^-{\FE}\ar[d]_-{\FH} & E\,\ar[d]^-{\EG} \\
				H \ar[r]^-{\HG} & G, 
			}} 
		\end{align}
		which commutes in the sense of $\HG\FH=\EG \FE$, 
		such that there is a $*$-isomorphism $\delta:C^r_0(G)\xrarr{\sim}\Ind_{\FE}\FH^*C^r_0(H)$ 
		which makes the following diagram commutes, 
		\begin{align}\label{diagram:delta}
			\vcenter{\xymatrix{
				C^u_0(G) \ar[r]^-{\Delta^u_{G}}\ar@{->>}[d]_-{\lambda_G} & 
				\mult{C^u_0(G)\otimes C^u_0(G)} 
				\ar[r]^-{\EG^{u}\otimes \HG^{u}} & 
				\mult{C^u_0(E)\otimes C^u_0(H)}\, \ar@{->>}[d]^-{\lambda_E\otimes\lambda_H} \\
				C^r_0(G) \ar[rr]^-{\delta} & & 
				\mult{C^r_0(E)\otimes C^r_0(H)}.
			}}
		\end{align}
	\end{condition}
	
	Here we remark that by using the fact $\clin{C^r_0(G)_1(\lambda_G\otimes\id)\Delta^u_G (C^u_0(G))}=C^r_0(G)\otimes C^u_0(G)$, 
	it can be seen that the diagram \cref{diagram:delta} commutes if and only if 
	$(\id\otimes\delta)\Delta_G=(\Delta_G\rpb\EG\otimes \id)\Delta_G\rpb\HG$ as maps from $C^r_0(G)$. 
	Similarly, this is also equivalent to 
	$(\delta\otimes\id)\Delta_G=(\id\otimes\HG^*\Delta_G)\EG^*\Delta_G$ as maps from $C^r_0(G)$. 
	
	We want to compare $\Ind_\HG A$ and $\Ind_\FE \FH^*A$ for a left $H$-C*-algebra $A$. 
	First, we are going to construct the map to compare two induced coactions formally. 
	It immediately turns out that it is isomorphic in some special case. 
	
	\begin{lem}\label{lem:formalinclusion}
		Consider the square \cref{diagram:square} with \cref{cond:square}, and let $(A,\alpha)$ be a left $H$-C*-algebra. 
		If $\Ind_{\HG}A$ and $\Ind_{\FE}\FH^*A$ exist, 
		then the followings hold. 
		\begin{enumerate}
			\item
			$(\id_{\K(E)}\otimes\alpha)^{-1}(\delta\otimes \id_A): \EG^*\Ind_{\HG}A\to \mult{\Ind_{\FE}\FH^*A}$ 
			is a well-defined injective non-degenerate left $E$-$*$-homomorphism. 
			\item
			Suppose $\EG$ gives an open quantum subgroup (or more generally $\hat{\EG}$ is proper). 
			When $A=\C$, the map in (1) is a $*$-isomorphism of 
			$\Ind_{\HG}\C \xrarr{\sim} \Ind_{\FE}\C$. 
		\end{enumerate}
	\end{lem}
	
	\begin{proof}
		We show (1). 
		Since its injectivity and equivariance are automatic, we prove well-definedness and non-degeneracy. 
		Also, since $(\id\otimes\alpha):\Ind_\FE\FH^* A\xrarr{\sim} \Ind_{\FE}\FH^*(\alpha(A))$ is a left $E$-$*$-isomorphism, 
		it is enough to show the map 
		$(\delta\otimes \id): \EG^*\Ind_{\HG}A\to \mult{\Ind_{\FE}\FH^*(\alpha A)}$ 
		is a well-defined non-degenerate $*$-homomorphism. 
		
		We have a non-degenerate inclusion
		\begin{align*}
			&
			(\delta\otimes \id)\Ind_\HG A
			=
			\clin{(\K(G)^*\otimes\id\otimes\id\otimes\id)(\id\otimes \delta\otimes \id)(\Delta_G\otimes \id)\Ind_\HG A}
			\\&
			\subset
			\clin{(\K(G)^*\otimes \id\otimes\id\otimes\id)
				\bigl((\Delta_G\rpb\EG\otimes \id)(\id\otimes \alpha)
				\mult{C^r_0(G)\otimes A}\bigr)}
			\\&
			\subset
			\mult{C^r_0(E)\otimes \alpha(A)}. 
		\end{align*}
		It follows 
		$(\delta\otimes \id)\Ind_\HG A
		\subset \WInd_\HG \FH^*(\alpha A)$ is a non-degenerate C*-subalgebra, and 
		$\Delta_E\otimes\id\otimes\id$ gives well-defined continuous $E$-coaction on it. 
		Therefore the argument of \cref{rem:indunique} shows 
		\begin{align*}
			&
			\clin{((\delta\otimes \id)\Ind_\HG A)\Ind_{\FE}\FH^*(\alpha A)} 
			= \Ind_{\FE}\FH^*(\alpha A) . 
		\end{align*}
		
		Next, we show (2). By (1) it is enough to show 
		$\clin{\delta(\Ind_{\HG} \C) \bigl((\Ind_{\FE}\C)\otimes 1_{\K(H)}\bigr)} 
		= \delta(\Ind_{\HG} \C)$. 
		First, we observe that 
		the restriction of $\delta$ gives a $*$-isomorphism 
		\begin{align*}
			&
			\delta:\M(C^r_0(G))^H\xrarr{\sim} \M(\Ind_\FE C^r_0(H))\cap \bigl(\M(C^r_0(E))^F\otimes 1_{\K(H)}\bigr) . 
		\end{align*}
		Then by \cref{rem:indmulthom} for $\C\to \mult{C^r_0(H)}$ we get an injective non-degenerate $*$-homomorphism 
		\begin{align*}
			&
			f:\Ind_{\FE}\C\xrarr{-\otimes1} 
			\M(\Ind_\FE C^r_0(H))\cap \bigl(\M(C^r_0(E))^F\otimes 1\bigr)
			\xrarr{\delta^{-1}} 
			\M(C^r_0(G)) .  
		\end{align*}
		By construction $f$ is a left $E$-$*$-homomorphism. 
		Thus by the normality of $\EG^r$ and (iii) of \cref{def:ind} for $\Ind_\HG\C$ we obtain 
		\begin{align}\label{eq:lem:formalinclusion}
			&
			\begin{aligned}
				&
				\clin{(\Ind_{\HG} \C) \bigl(f(\Ind_{\FE}\C)\bigr)} 
				= 
				\clin{(\Ind_{\HG} \C)  
					\bigl((\K(E)^*\otimes f)\Delta_E(\Ind_{\FE}\C) \bigr)} 
				\\&
				= 
				\clin{(\Ind_{\HG} \C) 
					\bigl((\K(E)^*\otimes \id)(\EG^r\otimes\id)\Delta_G 
					f(\Ind_{\FE}\C) \bigr)} 
				\\&
				\subset 
				\clin{(\Ind_{\HG} \C) 
					\bigl( (\K(G)^*\otimes \id)\Delta_G 
					f(\Ind_{\FE}\C) \bigr)} 
				\subset \Ind_{\HG} \C . 
			\end{aligned}
		\end{align}
		As for the converse inclusion, let $(e_\lambda)_\lambda$ be an approximate unit of $\Ind_{\FE}\C$ and $\omega$ be a normal state on $\B(L^2(E))$. 
		Then we have 
		$(\omega\EG^r\otimes\id)\Delta_G f(e_\lambda)\xrarr{\lambda}1$ strictly in $\mult{C^r_0(G)}$, 
		and thus the two inclusions in \cref{eq:lem:formalinclusion} are equalities. 
		By taking $\delta$ the desired claim follows. 
	\end{proof}
	
	Next, we suppose $\HG$ and $\FE$ are proper, 
	and state the analogue of base change formula.
	
	\begin{thm}\label{thm:basechange}
		Consider the square \cref{diagram:square} with \cref{cond:square}, and let $\HG$ and $\FE$ be proper. 
		Then the followings are equivalent. 
		\begin{enumerate}
			\item[(a)]
			For any left $H$-C*-algebra $(A,\alpha)$, 
			it holds $(\id_{\K(E)}\otimes\alpha)^{-1}(\delta\otimes \id_A): \EG^*\Ind_\HG A\to \Ind_{\FE}\FH^*A$ 
			is a well-defined left $E$-$*$-isomorphism. 
			\item[(b)]
			(a) holds for $A=\C$. 
		\end{enumerate}
	\end{thm}
	
	\begin{rem}\label{rem:basechangecondition}
		As in \cref{thm:basechange} above, consider the situation of \cref{cond:square} with $\HG$ and $\FE$ proper. 
		Then the condition (b) in the theorem is equivalent to the following conditions. 
		\begin{enumerate}
			\item[(c)]
			\textit{There is some non-degenerate C*-subalgebra $D\subset\mult{C^r_0(G)}$ which is a left $G$-C*-algebra by $\Delta_G$, such that $\delta(D)=(\Ind_\FE \C)\otimes 1_{\K(H)}$. }
			\item[(d)]
			\textit{$(\Ind_{\FE}\C)\otimes 1_{\K(H)} 
				\subset \delta\mult{\Ind_{\HG}\C}$ is a non-degenerate C*-subalgebra. }
		\end{enumerate}
		Indeed, clearly (b) implies (c). 
		(c) implies (d) because 
		by (iii) of \cref{def:ind} for $\Ind_{\HG}\C$ it holds 
		\begin{align*}
			&
			D=\clin{(\K(G)^*\otimes \id_{C^r_0(G)})\Delta_{G}(D)}
			\subset \clin{(\K(G)^*\otimes \id_{C^r_0(G)})\mult{\K(G)\otimes \Ind_{\HG}\C}} . 
		\end{align*}
		It is proved at the beginning of the proof of (2) in \cref{lem:formalinclusion} that (d) implies (b). 
	\end{rem}
	
	To prove \cref{thm:basechange}, we decompose the homomorphisms as we did in \cref{thm:properind}, 
	and reduce the statement to the cases of closed quantum subgroups and duals of open quantum subgroups. 
	In (4) of \cref{lem:basechangedecomp} next, we prove the latter case at the same time with the reduction step. 
	
	\begin{lem}\label{lem:basechangedecomp}
		Consider the square \cref{diagram:square} with \cref{cond:square}, and let $\HG$ and $\FE$ be proper. Then the followings hold. 
		\begin{enumerate}
			\item
			There is a unique homomorphism $\bar{\FH}:\Ran \HG\to\Ran \FE$ such that the two squares in the following diagram are commutative, 
			\begin{align}\label{diagram:decompsuare}
				\vcenter{\xymatrix{
					F \ar[r]^-{\bar{\FE}}\ar[d]_-{\FH} & 
					\Ran \FE \ar[r]^-{\ran \FE}\ar@{-->}[d]^-{\bar{\FH}} & 
					E\, \ar[d]^-{\EG} \\
					H \ar[r]^-{\bar{\HG}} & \Ran \HG \ar[r]^-{\ran \HG} & G. 
				}}
			\end{align}
			\item
			The left square of \cref{diagram:decompsuare} satisfies \cref{cond:square}. 
			The well-defined left $\Ran\FE$-$*$-isomorphism 
			\begin{align*}
				&
				(\id_{C^r_0(\Ran \FE)}\otimes \bar{\HG}^{r})(\bar{\FH}^*\Delta_{\Ran \HG}) : 
				C^r_0(\Ran \HG)\xrarr{\sim}\Ind_{\bar{\FE}}\FH^*C^r_0(H) 
			\end{align*}
			makes the diagram \cref{diagram:delta} for this square commute. 
			This square satisfies (b) of \cref{thm:basechange} 
			automatically. 
			\item
			The right square of \cref{diagram:decompsuare} satisfies \cref{cond:square}. 
			The well-defined left $E$-$*$-isomorphism  
			\begin{align*}
				&
				\bar{\delta}:=(\id_{C^r_0(E)}\otimes\bar{\HG}^r)^{-1}\delta : 
				C^r_0(G)\xrarr{\sim}\Ind_{\ran \FE}\bar{\FH}^*C^r_0(\Ran \HG)
			\end{align*}
			makes the diagram \cref{diagram:delta} for this square commute. 
			This square satisfies (b) of \cref{thm:basechange} 
			if the original square \cref{diagram:square} satisfies (b). 
			\item
			If (a) of \cref{thm:basechange} holds for the right square of \cref{diagram:decompsuare}, 
			then (a) for the original square \cref{diagram:square} also holds. 
		\end{enumerate}
	\end{lem}
	
	\begin{proof}
		(1) follows from the fact that 
		$C^u_0(\Ran \HG)$ can be canonically identified with the image of $\HG^{u}:C^u_0(G)\to C^u_0(H)$, 
		and so is $C^u_0(\Ran \FE)$. 
		
		We show (2). 
		The commutativity of the diagram \cref{diagram:delta} for this square is easy. 
		By using (2) of \cref{eg:formalind} and \cref{cor:properind} we calculate 
		\begin{align*}
			&
			C^r_0(H)^{\bar{\HG}} 
			= 
			\clin{(\K(G)^*\otimes \id)\bigl(\Ind_\HG C^r_0(H)\bigr)}
			= 
			\clin{(\K(G)^*\otimes \id)\bigl((\Delta_G\rpb \HG)C^r_0(G)\bigr)}
			\\&
			=
			\clin{(\K(G)^*\otimes \K(E)^*\otimes \id)
				\bigl(\Ad V^{\EG}_{12}
				((\Delta_G\rpb \HG)C^r_0(G))_{13}\bigr)}
			\\&
			=
			\clin{(\K(G)^*\otimes \K(E)^*\otimes \id)
				(\id\otimes \delta )
				\Delta_G C^r_0(G)}
			=
			\clin{(\K(E)^*\otimes \id)
				(\delta C^r_0(G))}
			\\&
			=
			\clin{(\K(E)^*\otimes \id)
				\bigl(\Ind_{\FE} \FH^* C^r_0(H)\bigr)}
			=
			(\FH^* C^r_0(H))^{\bar{\FE}}. 
		\end{align*}
		Thus again a similar calculation shows 
		\begin{align}\label{prf:lem:basechangedecomp1}
			&
			\begin{aligned}
				&
				(\FH^* C^r_0(H))^{\bar{\FE}}
				=
				C^r_0(H)^{\bar{\HG}} 
				=
				\clin{(\K(\Ran \HG)^*\otimes \bar{\HG}^{r}) \bigl(\Delta_{\Ran\HG} C^r_0(\Ran \HG)\bigr)}
				=
				\bar{\HG}^{r} C^r_0(\Ran \HG) . 
			\end{aligned}
		\end{align}
		From this and \cref{lem:propersurjind} we obtain 
		\begin{align*}
			&
			(\bar{\FE}^{r}\otimes \bar{\HG}^{r})(\bar{\FH}^*\Delta_{\Ran \HG}) C^r_0(\Ran \HG) 
			=(\FH^*\Delta_H) \bar{\HG}^{r} C^r_0(\Ran \HG) 
			=
			(\FH^*\Delta_H) (\FH^* C^r_0(H)^{\bar{\FE}})
			\\&
			=
			(\bar{\FE}^{r}\otimes \id)
			((\FH^*\Delta_H)^{\bar{\FE}}) (\FH^* C^r_0(H)^{\bar{\FE}})
			=
			(\bar{\FE}^{r}\otimes \id) \Ind_{\bar{\FE}}\FH^* C^r_0(H) . 
		\end{align*}
		Therefore the desired $*$-isomorphism follows from the injectivity of $\bar{\FE}^{r}\otimes \id$. 
		Now (b) of \cref{thm:basechange} for this square holds because of (1) in \cref{lem:formalinclusion} and the fact 
		$\Ind_{\bar{\HG}}\C\isom \C\isom \Ind_{\bar{\FE}}\C$. 
		
		As for (3), 
		$\bar{\delta}$ is well-defined as a map of $C^r_0(G)\to \M(C^r_0(E)\otimes C^r_0(\Ran\HG))$ 
		because 
		\begin{align*}
			&
			\delta C^r_0(G)
			=
			(\lambda_E\otimes\lambda_{H})(\EG^u\otimes\HG^u)\Delta^u_G C^u_0(G)
			=
			(\id\otimes\bar{\HG}^r)(\lambda_E\otimes\lambda_{\Ran\HG})(\EG^u\otimes(\ran\HG)^u)\Delta^u_G C^u_0(G)
			\\&
			\subset
			(\id\otimes\bar{\HG}^r)\mult{C^r_0(E)\otimes C^r_0(\Ran \HG)} . 
		\end{align*}
		By definition, $\bar{\delta}$ makes the diagram \cref{diagram:delta} commute. 
		By using \cref{cor:properind} and \cref{prf:lem:basechangedecomp1}, we get
		\begin{align*}
			&
			(\id\otimes \bar{\HG}^{r})
			\Ind_{\ran \FE}\bar{\FH}^* C^r_0(\Ran \HG)
			=
			(\id\otimes \bar{\HG}^{r})
			\Ind_{\FE}(\bar{\FH}\bar{\FE})^* C^r_0(\Ran \HG)
			\\&
			=
			(\id\otimes \bar{\HG}^{r})
			\Ind_{\FE}(\bar{\HG}\FH)^* C^r_0(\Ran \HG)
			=
			\Ind_{\FE}\bar{\FE}^*(\FH^* C^r_0(H)^{\bar{\FE}})
			=
			\Ind_{\FE}\FH^* C^r_0(H) . 
		\end{align*}
		Therefore $\bar{\delta}:C^r_0(G)\to \Ind_{\ran \FE}\bar{\FH}^* C^r_0(\Ran \HG)$ is a well-defined $*$-isomorphism 
		because $\delta$ is a $*$-isomorphism by the assumption. 
		(b) of \cref{thm:basechange} for the original square \cref{diagram:square} implies 
		(b) for the right square of \cref{diagram:decompsuare} 
		because 
		$\Ind_{\ran \FE}\C = \Ind_{\FE}\C$ and $\Ind_{\ran \HG}\C = \Ind_{\HG}\C$ hold by \cref{cor:properind}. 
		
		Finally, we prove (4). 
		Take a left $H$-C*-algebra $(A,\alpha)$ arbitrarily. 
		By \cref{lem:propersurjind} and \cref{prf:lem:basechangedecomp1}, we calculate 
		\begin{align*}
			&
			(\FH^*\Delta_H\otimes\id)\alpha(A^{\bar{\HG}})
			=
			(\FH^*\Delta_H\otimes\id)
			\clin{
				\alpha(A)\cap\M(\bar{\HG}^rC^r_0(\Ran\HG)\otimes A) }
			\\&
			=
			(\FH^*\Delta_H\otimes\id)
			\bigl[\alpha(A)\cap\M(\FH^* C^r_0(H)^{\bar{\FE}}\otimes A)\bigr]
			\\&
			=
			(\bar{\FE}^r\otimes\id\otimes\id)
			\bigl(p_1 \bigl((\FH^*\Delta_H\otimes\id)\alpha(A)\bigr) p_1^*\bigr)
			=
			(\FH^*\Delta_H\otimes\id)\bigl(\FH^*(\alpha A)^{\bar{\FE}}\bigr) , 
		\end{align*}
		where we have put $p:=p^{\hat{\bar{\FE}}}$ associated with $\hat{\bar{\FE}}$. 
		Thus we get $\FH^*(\alpha(A^{\bar{\HG}}))=\bar{\FE}^*(\FH^*(\alpha A)^{\bar{\FE}})$ as left $F$-C*-algebras. 
		If (a) of \cref{thm:basechange} for the right square of \cref{diagram:decompsuare} holds, 
		then $(\bar{\delta}\otimes\id)(\Ind_{\ran\HG}B)=\Ind_{\ran\FE}\bar{\FH}^*(\beta B)$ for any left $(\Ran\HG)$-C*-algebra $(B,\beta)$. 
		Therefore we obtain 
		\begin{align*}
			&
			(\delta\otimes \id)\Ind_{\HG}A
			=
			(\delta\otimes \id)\Ind_{\ran \HG}A^{\bar{\HG}}
			=
			(\id\otimes \bar{\HG}^{r}\otimes \id)
			\Ind_{\ran \FE}\bar{\FH}^*(\alpha^{\bar{\HG}}(A^{\bar{\HG}}))
			\\&
			=
			(\id\otimes \bar{\HG}^{r}\otimes \id)
			\Ind_{\FE}(\bar{\FH}\bar{\FE})^*(\alpha^{\bar{\HG}}(A^{\bar{\HG}}))
			=
			(\id\otimes \bar{\HG}^{r}\otimes \id)
			\Ind_{\FE}(\bar{\HG}\FH)^*(\alpha^{\bar{\HG}}(A^{\bar{\HG}}))
			\\&
			=
			\Ind_{\FE}\FH^*(\alpha(A^{\bar{\HG}}))
			=
			\Ind_{\FE}\bar{\FE}^* (\FH^*(\alpha A)^{\bar{\FE}})
			=
			\Ind_{\FE}\FH^*(\alpha A) 
			=
			(\id\otimes\alpha)\Ind_{\FE}\FH^*A . 
		\end{align*}
		Now (a) for the original square \cref{diagram:square} follows since $\delta\otimes \id_A$ is an injective non-degenerate left $E$-$*$-homomorphism. 
	\end{proof}
	
	We proceed to prove \cref{thm:basechange}. 
	The remaining case is the case when $\HG$ and $\FE$ give closed quantum subgroups. 
	In order to prove this, we will concretely realize a left $F$-C*-algebra $B$ with $\Ind_\FE B\isom \EG^*\Ind_{\HG}A$ 
	by following the proof of~\cite[Theorem~9.2]{Vaes:impr}, 
	and reduce the conclusion to $\FH^*\alpha(A)\isom B$. 
	
	\begin{prf:thm:basechange}
		We only have to prove (b) implies (a). 
		By (3) and (4) of \cref{lem:basechangedecomp}, it suffices to consider 
		the case when $\HG$ and $\FE$ give closed quantum subgroups. 
		Consider the C*-algebras  
		\begin{align*}
			\begin{aligned}
				&
				C^\HG:=\clin{C^r_0(\hat{E})_1 \bigl((\delta\otimes \id_A) \Ind_{\HG}A\bigr)}, 
				\quad\mathrm{and}\quad
				C^\FE:=\clin{C^r_0(\hat{E})_1 (\Ind_{\FE}\FH^*(\alpha A))}, 
			\end{aligned}
		\end{align*}
		which are well-defined right $\hat{E}$-C*-algebras with the coactions given by the restrictions of 
		$\Ad \hat{V}^{E}_{14}(-)_{123}$. 
		The restrictions of $\Ad V^E_{12}(-)_{134}$ give 
		right $\hat{E}$-$*$-isomorphisms from $C^\HG$ and $C^\FE$ to 
		$E\redltimes \bigl((\delta\otimes \id) \Ind_{\HG}A\bigr)$ and 
		$E\redltimes \Ind_{\FE}\FH^*(\alpha A)$, respectively. 
		
		Let us assume (b). 
		Then we can consider a non-degenerate $*$-homomorphism 
		\begin{align*}
			C^r_0(E/F)\xrarr{-\otimes1_{\K(H)}}
			\delta \mult{C^r_0(G/H)} \xrarr{-\otimes 1_A}
			\M(C^\HG) , 
		\end{align*}
		where the last map is well-defined and non-degenerate because 
		$C^r_0(G/H)\otimes 1_A\subset \M(\Ind_{\HG}A)$ is a non-degenerate C*-subalgebra by (iii) for \cref{def:ind} of $\Ind_\HG A$. 
		It follows $\clin{C^r_0(\hat{E})C^r_0(E/F)}_1\subset \mult{C^\HG}$ is a non-degenerate C*-subalgebra, and thus 
		$\clin{I^{\FE}_1I^{\FE}_1{}^* C^\HG I^{\FE}_1I^{\FE}_1{}^*} =C^\HG$. 
		
		It turns out $\clin{I^{\FE}_1{}^* C^\HG I^{\FE}_1}$ is a right $\hat{E}$-C*-algebra with $\Ad \hat{V}^{\FE}_{14}(-)_{123}$, 
		because $\hat{V}^{E}(-)_1\hat{V}^{\FE}{}^*$ gives a well-defined continuous right $\hat{E}$-coaction on $I^\FE$. 
		Therefore $\Ad \hat{V}^{F}_{14}(-)_{123}$ is a well-defined continuous right $\hat{F}$-coaction on $\clin{I^{\FE}_1{}^* C^\HG I^{\FE}_1}$ 
		by the argument from~\cite[Theorem~9.2]{Vaes:impr} (see \cref{lem:coactionclosedsubqg} below). 
		Since $W^F_{12}\in\U\mult{C^r_0(F)\otimes [I^{\FE}_1{}^* C^\HG I^{\FE}_1]}$ satisfies 
		$(\Delta_F\otimes \id)W^F=W^F_{13}W^F_{23}$ and $\Ad \hat{V}^{F}_{25}(W^F_{12})=W^F_{15}W^F_{12}$, 
		we can use (the proof of)~\cite[Theorem~6.7]{Vaes:impr}. 
		As a consequence, it holds that 
		\begin{align*}
			&
			B:=\Bigl[(\K(F)^*\otimes \id_{\K(F)\otimes \K(H)\otimes A}) 
				\Bigl(\Ad (W^{F}_{12}{}^* \hat{V}^{F}_{21}) 
				\clin{I^{\FE}_1{}^* C^\HG I^{\FE}_1}_{234}\Bigr)\Bigr]
			\subset \mult{\K(F)\otimes \K(H)\otimes A}
		\end{align*}
		with $\Ad W^{F}_{12}{}^*(-)_{234}$ is a well-defined left $F$-C*-algebra, and that 
		$F\redltimes B = 
		\Ad (W^{F}_{12}{}^* \hat{V}^{F}_{21})\clin{I^{\FE}_1{}^* C^\HG I^{\FE}_1}_{234}$. 
		
		We note that up to a scalar it holds 
		$V^F \Sigma (\hat{J}^{F}J^{F})_1 = W^F{}^* \hat{V}^{F}_{21}$, 
		since up to a scalar $1=(\Sigma (J^F\hat{J}^{F})_2V^F)^3$ 
		by~\cite[(2.2)]{Baaj-Skandalis-Vaes:nonsemireg}. 
		Combining this with $I^\FE_1 V^F{}^* = ((\hat{\FE}^{r}{}'\otimes \id)V^F)^* I^\FE_1 = V^\FE{}^* I^\FE_1$, we have 
		\begin{align*}
			&
			\Ad (W^{F}_{12}{}^* \hat{V}^{F}_{21})\clin{I^{\FE}_1{}^* C^\HG I^{\FE}_1}_{234}
			=\Ad V^F_{12}
			\clin{I^{\FE}_1{}^* C^\HG I^{\FE}_1}_{134}
			=\clin{I^{\FE}_1{}^* V^{\FE}_{12} C^\HG_{134} V^{\FE}_{12}{}^* I^{\FE}_1} . 
		\end{align*}
		Since $\Ad V^{\FE}_{12}(C^\HG_{134})=\Ad W^{\FH}_{23}{}^*(C^\HG_{134})$ by the definition of $\delta$, 
		we obtain from the arguments so far, 
		\begin{align*}
			&
			\clin{I^\FE_1 C^r_0(\hat{F})_1 W^{F}_{12}{}^* B_{234} W^{F}_{12} I^\FE_1{}^*}
			=
			\clin{I^\FE_1 I^{\FE}_1{}^* W^{\FH}_{23}{}^* C^\HG_{134} W^{\FH}_{23} I^{\FE}_1 I^{\FE}_1{}^*}
			=
			\Ad W^{\FH}_{23}{}^*(C^\HG_{134}). 
		\end{align*}
		
		On the other hand, we also have 
		\begin{align*}
			&
			\clin{I^\FE_1 C^r_0(\hat{F})_1 (\Ad(W^{F}_{12}{}^* W^{\FH}_{23}{}^*)\alpha(A)_{34}) I^\FE_1{}^*}
			\\&
			=
			\clin{I^\FE_1 C^r_0(\hat{F})_1 (\Ad(W^{\FH}_{23}{}^* W^{\FH}_{13}{}^*)\alpha(A)_{34}) I^\FE_1{}^*}
			=
			\Ad W^{\FH}_{23}{}^* (C^\FE_{134}) . 
		\end{align*}
		Therefore if it holds 
		\begin{align}\label{prf:thm:basechange2}
			B=\Ad W^{\FH}_{12}{}^* \alpha(A)_{23}
			\quad\mathrm{in}\quad \mult{\K(F)\otimes \K(H)\otimes A}, 
		\end{align}
		then $C^\HG =C^\FE\subset\mult{\K(E)\otimes \K(H)\otimes A}$, and 
		by taking $\clin{(\K(E)^*\otimes \id_{\K(E)\otimes \K(H)\otimes A})\bigl(\Ad V^{E}_{12}(-)_{134}\bigr)}$ 
		it follows $(\delta\otimes \id)\Ind_\HG A=\Ind_\FE \FH^*(\alpha A)$, which will complete the proof. 
		But we calculate by writing $A_\K:=\K(F)\otimes \K(H)\otimes A$, 
		\begin{align*}
			&
			B
			=
			\clin{(\K(F)^*\otimes \id_{A_\K})\left(
				C^r_0(\hat{F})_1 \bigl(\Ad W^{F}_{12}{}^*(B)_{234}\bigr)\right)}
			\\&=
			\clin{(\K(E)^*\otimes \id_{A_\K})\left(
				I^{\FE}_1 C^r_0(\hat{F})_1 \bigl(\Ad W^{F}_{12}{}^*(B)_{234}\bigr) I^{\FE}_1{}^* \right)}
			\\&=
			\clin{(\K(E)^*\otimes \id_{A_\K})\left(
				\Ad W^{\FH}_{23}{}^* C^{\HG}_{134} \right)}
			=
			\clin{(\K(E)^*\otimes \id_{A_\K})\Bigl(
				\Ad W^{\FH}_{23}{}^* 
				((\delta\otimes \id_A)\Ind_{\HG}A)_{134} \Bigr)}
			\\&=
			\clin{(\K(G)^*\otimes \K(E)^*\otimes \id_{A_\K})\Bigl(
				\Ad W^{\FH}_{34}{}^* 
				\bigl((\id_{\K(G)}\otimes \delta\otimes \id_A)(\Delta_G\otimes \id_A)\Ind_{\HG}A\bigr)_{1245} \Bigr)}
			\\&=
			\clin{(\K(G)^*\otimes \K(E)^*\otimes \id_{A_\K})\left(
				\Ad (V^{\EG}_{12}W^{\FH}_{34}{}^*V^{\HG}_{14}) 
				(\Ind_{\HG}A)_{15} \right)}
			\\&=
			\clin{(\K(G)^*\otimes \id_{A_\K})\left(
				\Ad (W^{\FH}_{23}{}^*V^{\HG}_{13}) 
				(\Ind_{\HG}A)_{14} \right)}
			\\&=
			\Ad W^{\FH}_{12}{}^* 
			\clin{(\K(G)^*\otimes \id_{\K(H)\otimes A})\bigl(\Ind_\HG \alpha(A)\bigr)}_{23}
			=\Ad W^{\FH}_{12}{}^*\alpha(A)_{23} , 
		\end{align*}
		and thus \cref{prf:thm:basechange2} holds as desired. 
	$\sq$\end{prf:thm:basechange}
	
	In the proof above, we have used an argument taken from~\cite[Theorem~9.2]{Vaes:impr}. 
	We record this for the convenience of the reader. 
	
	\begin{lem}\label{lem:coactionclosedsubqg}
		Let $\HG:H\to G$ be a homomorphism of regular locally compact quantum groups which gives a closed quantum subgroup, 
		$\HC$ be a Hilbert space, 
		and $A\subset \B(\HC)$ be a non-degenerate C*-subalgebra. 
		Let $\alpha:A\to \B(\HC)\barotimes L^\infty(\hat{H})$ be a $*$-homomorphism such that $(\id\barotimes\hat{\HG}^r)\alpha$ is a continuous right $\hat{G}$-coaction on $A$. 
		Then $\alpha$ is a continuous right $\hat{H}$-coaction on $A$. 
	\end{lem}
	
	\begin{proof}
		Clearly $\alpha$ is injective on $A$. 
		By putting 
		\begin{align*}
			&
			\J:=\{v\in\B(L^2(G),L^2(H))\,|\, xv=v\hat{\HG}^r(x)\mathrm{\ for\ all\ }x\in L^\infty(\hat{H})\} , 
		\end{align*}
		we have 
		\begin{align*}
			&
			\clin{\alpha(A) L^2(H)_2} 
			=
			\clin{\alpha(A) (\J L^2(G))_2} 
			= \bigl[\J_2 \bigl((\id\barotimes\hat{\HG}^r)\alpha(A)\bigr) L^2(G)_2\bigr] 
			\\&
			= \clin{\J_2 (A\otimes L^2(G))} 
			= A\otimes L^2(H). 
		\end{align*}
		Thus $\alpha:A\to \M(A\otimes\K(H))$ 
		is a well-defined non-degenerate $*$-homomorphism. 
		By the injectivity of $\hat{\HG}^r$ it holds 
		$(\alpha\otimes\id_{\K(H)})\alpha=(\id_{\B(\HC)}\barotimes\Delta_{\hat{H}})\alpha$. 
		Then it follows from the regularity of $\hat{H}$ that 
		\begin{align*}
			&
			[\alpha(A)C^r_0(\hat{H})_2 ]
			= \bigl[(\id\otimes\id\otimes\K(H)^*)
				\bigl(\hat{W}^H_{23}{}^*\alpha(A)_{13}\hat{W}^H_{23} C^r_0(\hat{H})_2 \bigr)\bigr]
			= A\otimes C^r_0(\hat{H}), 
		\end{align*}
		which shows the desired statement. 
	\end{proof}
	
	\begin{rem}\label{rem:hmgquot}
		Consider the square \cref{diagram:square} with \cref{cond:square} and let $\HG$ and $\FE$ be proper as in \cref{thm:basechange}. 
		To use \cref{thm:basechange}, the condition (b) looks still technical. 
		But in many practical cases, this follows automatically from other assumptions, as we list here. 
		\begin{itemize}
			\item
			As observed by~\cite{Vaes:impr} and~\cite{Nest-Voigt:eqpd}, (c) of \cref{rem:basechangecondition} 
			holds when the square \cref{diagram:square} is given by (generalized) quantum doubles. 
			See \cref{eg:ydbasechange} below for detail. 
			This is our motivating example of \cref{thm:basechange}. 
			\item
			(b) holds when $\EG$ gives an open quantum subgroup by (2) of \cref{lem:formalinclusion}. 
			\item
			(b) holds when $\hat{\HG}$ gives an open quantum subgroup by (2) of \cref{lem:basechangedecomp}. 
			This is also contained in the following case. 
			\item
			(b) holds if the next condition holds. 
			\begin{itemize}
				\item[(d${}^+$)]
				\textit{ $\Ind_{\HG}\C$ also gives the 
				induced coaction $C^r_0(G)^H$ of the right $H$-C*-algebra $C^r_0(G)\rpb \HG$ along the homomorphism $1_{H\to1}$. }
			\end{itemize}
			To see this, we observe as in the beginning of the proof of (2) in \cref{lem:formalinclusion} that  
			\begin{align*}
				&
				(\Ind_{\FE}\C) \otimes 1_{\K(H)} 
				\subset 
				\delta\M(C^r_0(G))^H
				= \delta\mult{C^r_0(G)^H} . 
			\end{align*}
			This inclusion is non-degenerate by (iii) of \cref{def:ind} for $C^r_0(G)^H$, 
			which shows (d) of \cref{rem:basechangecondition}. 
			As we saw in \cref{eg:hmgquot}, (d${}^+$) holds 
			when $G$ is a locally compact group, $\HG$ gives an open quantum subgroup, or $H$ is a compact quantum group. 
		\end{itemize}
	\end{rem}
	
	We give some examples. From now, we use the notation and terminologies from \cref{sec:appendix} except \cref{prop:twiten} and \cref{eg:twiten}, which concern twisted tensor products. 
	
	\begin{eg}\label{eg:doublebasechange}
		Let $F,G,H$ be regular locally compact quantum groups, 
		$\m\in\Aut(C^r_0(G)\otimes C^r_0(F))$ extend to a matching on $G$ and $F$, and 
		$\n\in\Aut(C^r_0(H)\otimes C^r_0(F))$ extend to a matching on $H$ and $F$. 
		Consider a proper homomorphism $\HG:H\to G$ satisfying \cref{eq:prop:doublehom1}. 
		Then we have a homomorphism 
		$\HG^{\op}\bowtie\id_F:\DD(\n)\to \DD(\m)$ 
		by \cref{prop:doublehom} and 
		the commutative square 
		\begin{align*}
			\xymatrix{
				H^{\op}\ar[r]^-{\HG^{\op}}\ar[d] & G^{\op}\ar[d] \\
				\DD(\n) \ar[r]^-{\HG^{\op}\bowtie\id_F} & \DD(\m)
			}
		\end{align*}
		satisfies \cref{cond:square}. 
		Indeed, $\delta$ is given by $(\sigma\HG^*\Delta_G)\otimes \id$ 
		and $C^r_0(\DD(\n))\isom C^r_0(H^{\op})\otimes C^r_0(\hat{F})$ with $\Delta_H^{\cop}\otimes\id$ as left $H^{\op}$-C*-algebras. 
		For example, when $\HG$ gives an open quantum subgroup it follows 
		that $\HG^{\op}\bowtie\id$ also gives an open quantum subgroup by \cref{prop:doublehom}, 
		and thus (a) of \cref{thm:basechange} holds by \cref{rem:hmgquot}. 
	\end{eg}
	
	\begin{eg}\label{eg:ydbasechange}
		Consider a homomorphism $\HG:H\to G$ of strongly regular locally compact quantum groups that gives a closed quantum subgroup and
		a $H$-Yetter--Drinfeld C*-algebra $(A,\alpha,\hat{\alpha})$. 
		It is shown by~\cite[Proposition~3.4]{Nest-Voigt:eqpd} that
		$\Ind_\HG ((\id\bowtie\hat{\HG})^*A)$ has a canonical structure of an $G$-Yetter--Drinfeld C*-algebra. 
		Note that with \cref{rem:strreg} taken into account, this statement only requires regularity. 
		
		More generally, let $\HG:H\to G$, $\GF:G\to F$ be homomorphisms of regular locally compact quantum groups with $\HG$ being proper. 
		By the previous example, the square 
		\begin{align*}
			\xymatrix{
				H\ar[r]^-{\HG}\ar[d] & G\ar[d] \\
				\DD(\GF\HG) \ar[r]^-{\HG\bowtie\id_{\hat{F}}} & \DD(\GF)
			}
		\end{align*}
		satisfies \cref{cond:square}. 
		By \cref{eg:doublehom}, $\HG\bowtie\id_{\hat{F}}:\DD(\GF\HG)\to \DD(\GF)$ is a well-defined proper homomorphism. 
		We want to show this satisfies (a) in \cref{thm:basechange}. 
		
		By the proof of~\cite[Theorem~8.2]{Vaes:impr}, it holds $\Ind_\HG\C =C^r_0(G/\Ran\HG)$ with $\Ad \hat{W}^G{}^*(-)_2$ is a well-defined left $\hat{G}$-C*-algebra. 
		By taking $\hat{\GF}^*$, it follows 
		$(C^r_0(G/\Ran\HG),\Delta_G,\Ad \hat{W}^{\GF}{}^*(-)_2)$ is a well-defined $\GF$-Yetter--Drinfeld C*-algebra. 
		Also, it is easy to see 
		$C^r_0(G/\Ran\HG)\otimes 1_{\K(F)}\subset \M(C^r_0(\DD(\GF)))$ is a left $\DD(\GF)$-$*$-homomorphism. 
		This means (c) of \cref{rem:basechangecondition}, because 
		\begin{align*}
			&
			\delta(C^r_0(G/\Ran\HG)\otimes 1_{\K(F)})
			=\Ad V^{\HG}_{12}(C^r_0(G/\Ran\HG))_1
			=C^r_0(G/\Ran\HG)_1
			\subset \delta\M(C^r_0(\DD(\GF))) . 
		\end{align*}
		
		Therefore (a) holds, and for any $\GF\HG$-Yetter--Drinfeld C*-algebra $(A,\alpha,\hat{\alpha})$, 
		$\Ind_{\HG\bowtie\id}A \isom \Ind_\HG A$ as left $G$-C*-algebras. 
		Via this $*$-isomorphism, 
		$(\Ind_\HG A,\Delta_G, \sigma_{12}\Ad W^{\GF}_{12}(\id\otimes\hat{\alpha}))$ 
		is a $\GF$-Yetter--Drinfeld C*-algebra. 
	\end{eg}
	
	\begin{eg}\label{eg:exactbasechange}
		Consider a sequence of regular locally compact quantum groups with $K$ being compact, 
		\begin{align*}
			\xymatrix{
				1\ar[r] &
				K\ar[r]^-{\iota} &
				H\ar[r]^-{\HG} &
				G\ar[r] & 1
			}, 
		\end{align*}
		which is exact in the sense of $\iota,\hat{\HG}$ giving closed quantum subgroup and $C^r_0(H/K) =\HG^{r}C^r_0(G)$ 
		(see (3) of \cref{eg:doublehom}). 
		Then 
		\begin{align*}
			\xymatrix{
				K\ar[r]\ar[d]_-{\iota} & 1 \ar[d]\\
				H\ar[r]^-{\HG} & G
			}
		\end{align*}
		satisfies \cref{cond:square}, and condition (d) of \cref{thm:basechange} by \cref{rem:hmgquot}. 
		Therefore for any left $H$-C*-algebra $A$, 
		it holds $A^\HG=\Ind_{K\to 1} \iota^*A = {}_{\1_K}\iota^*A$. 
	\end{eg}
	
	\section{Morita and \texorpdfstring{$K$}{K}-theoretic properties} \label{sec:MK}
	In this section, we consider the functors between several kinds of the categories of coactions which are induced by induced coactions along a proper homomorphism $\HG$. 
	We prove several properties of such functors which are known when $\HG$ gives a closed quantum subgroup in~\cite{Nest-Voigt:eqpd} and~\cite{Vaes:impr}. 
	We say left $G$-$(A,B)$-correspondences $(\E,\pi)$ and $(\F,\varpi)$, 
	are \emph{left $G$-unitarily isomorphic} 
	if there is a unitary $U\in\U\LC_B(\E,\F)$ 
	satisfying $(\Ad U)\pi=\varpi$, such that 
	$\Ad \left(\begin{array}{cc}
		U & 0 \\ 0 & 1_B \\
	\end{array}\right):\K_B(\E\oplus B)\xrarr{\sim}\K_B(\F\oplus B)$ is a left $G$-$*$-isomorphism. 
	We will also use the terminologies from \cref{sec:appendix}, especially those about twisted tensor products. 
	
	\begin{defn}\label{def:categories}
		Let $G$ be a regular locally compact quantum group. 
		\begin{enumerate}
			\item
			Let $\Calg^G$ be the category of left $G$-C*-algebras as objects, and $G$-$*$-homomorphisms $A\to B$ as morphisms from $A$ to $B$ for $A,B\in\Calg^G$. 
			\item
			Let $\Cor^G$ be the category of left $G$-C*-algebras as objects, and morphisms as follows. 
			For $A,B\in\Cor^G$, a morphism from $A$ to $B$ is 
			a left $G$-unitary isomorphism class of 
			left $G$-$(A,B)$-correspondences $(\E,\pi)$ with non-degenerate $\pi$. 
			The composition of morphisms is given by inner tensor products. 
			\item
			Suppose $C^r_0(G)$ is separable. 
			Let $\KK^G$ be the category of separable left $G$-C*-algebras as objects, 
			and the $G$-equivariant bivariant $K$-group $\KK^G(A,B)$ as the set of morphisms from $A$ to $B$ for $A,B\in\KK^G$. 
			The composition of morphisms is given by Kasparov products. 
			We refer~\cite{Baaj-Skandalis:eqkk} for details. 
		\end{enumerate} 
	\end{defn}
	
	We will use the following facts from \cite{Nest-Voigt:eqpd} about $\KK^G$ for a regular locally compact quantum group $G$ with $C^r_0(G)$ separable. 
	Let $F:\Calg^G\to \A$ be a functor to an additive category $\A$ with following three properties. 
	\begin{itemize}
		\item
		For $A,B\in\Calg^G$ and a left $G$-$*$-homomorphism 
		$f:A\to B\otimes C([0,1])$, it holds $F(f_0)=F(f_1):F(A)\to F(B)$, 
		where $f_t(a)$ denotes $(f(a))(t)$ for all $a\in A$ and $t\in [0,1]$. 
		\item
		For a split exact sequence of left $G$-$*$-homomorphisms 
		$\xymatrix@C=1.2em{
			0\ar[r]& I\ar[r]^-{\iota}& A\ar@<0.3ex>[r]& 
			B\ar[r]\ar@<0.3ex>[l]^-{s} & 0
		}$
		with $I,A,B\in \Calg^G$, 
		it holds $F(\iota)\oplus F(s):F(I)\oplus F(B)\xrarr{\sim}F(A)$ is isomorphic. 
		\item
		For $A\in\Calg^G$ and left $G$-Hilbert spaces $\HC,\HC'$, 
		it holds $F(\iota\otimes \id_A): F(\K_A(\HC\otimes A))\xrarr{\sim} F\bigl(\K_A((\HC\oplus\HC')\otimes A)\bigr)$ is isomorphic, 
		where we write $\iota:\K(\HC)\to \K(\HC\oplus\HC')$ for the canonical inclusion. 
	\end{itemize}
	Then there is an additive functor $\bar{F}:\KK^G\to \A$ such that 
	$\bar{F}\iota\isom F$, 
	where we write $\iota:\Calg^G\to\KK^G$ for the canonical functor. 
	Moreover, such $\bar{F}$ is unique up to a natural isomorphism and $\iota$ itself satisfies the three properties above. 
	Also, there is a canonical structure of a triangulated category on $\KK^G$. 
	
	\begin{rem}\label{rem:functors}
		Consider a homomorphism $\HG:H\to G$ of regular locally compact quantum groups, and $\CC\in\{\Calg, \Cor, \KK\}$, where we additionally assume 
		separability of $C^r_0(G)$ and $C^r_0(H)$ when $\CC=\KK$. 
		\begin{enumerate}
			\item
			We still write $\HG^*:\CC^G\to\CC^H$ for the functor given by restrictions of left coactions along $\HG$. 
			\item
			When $\HG$ is proper, induced coactions canonically induces the well-defined functor $\HG_*:\Calg^H\to\Calg^G$. 
			When $\HG$ gives a closed quantum subgroup, this is shown in the last paragraph of Section~2 in~\cite{Nest-Voigt:eqpd}. 
			When $\hat{\HG}$ gives an open quantum subgroup, this follows 
			from $f(A^{\bar{\HG}})\subset B^{\bar{\HG}}$ 
			for any left $H$-$*$-homomorphism $f:A\to B$ of left $H$-C*-algebras. 
			The general case holds by \cref{thm:indstages}. 
			We can also check that this functor preserves split exact sequences. 
			\item
			Suppose $\CC=\Cor$ or $\KK$. 
			We also write $\HG_*:\CC^H\to\CC^G$ 
			for the well-defined functor given by 
			$\hat{G}^{\op}\redltimes(\hat{\HG}^{\op}{}^*(H\redltimes-))$. 
			This construction is compatible with compositions of homomorphisms by Baaj--Skandalis duality. 
		\end{enumerate}
	\end{rem}
	
	As for (3), we note that when $\HG$ does not give a closed quantum subgroup 
	induced coactions along $\HG$ need not induce a functor $\CC^H\to\CC^G$ in a direct way (see the remark before \cref{prop:impr}). 
	
	\subsection{Induced coactions and twisted tensor products}
	
	We want to regard the functors $\HG_*:\CC^H\to\CC^G$ of (3) in \cref{rem:functors} 
	as the functors induced by induced coactions. 
	This is justified because it holds 
	$\HG_*(A)\isom \Ind_\HG(\hat{G}^{\op}\redltimes G\redltimes A)$ 
	as a special case of \cref{prop:projfml} next by (2) of \cref{eg:twiten}. 
	
	\begin{prop}\label{prop:projfml}
		Let $\HG:H\to G$, $\GF:G\to F$, $\FE:F\to E$ be homomorphisms of regular locally compact quantum groups and suppose $\HG$ is proper. 
		Then for any $\GF\HG$-Yetter--Drinfeld C*-algebra $A$ and any $\FE$-Yetter--Drinfeld C*-algebra $B$, 
		it holds 
		$\Ad W^{\GF}_{12}{}^*\sigma_{12}:(\Ind_{\HG}A)\outensor{F}B \xrarr{\sim}\Ind_{\HG}(A\outensor{F}B)$ 
		is a well-defined left $\DD(\FE\GF)$-$*$-isomorphism. 
	\end{prop}
	
	When $\HG$ gives a closed quantum subgroup, this is proved in~\cite[Theorem~3.6]{Nest-Voigt:eqpd}. 
	Their proof can be easily adapted to our situation, 
	but nevertheless we include the proof 
	since their statement is made in a slightly different way. 
	The crucial point for the proof is an argument from~\cite[Theorem~8.2]{Vaes:impr} that 
	for any left $G$-C*-algebra $(B,\beta)$, it holds 
	\begin{align}\label{eq:prop:projfml1}
		\bigl[I^\HG_1(\HG^*\beta)(B)\bigr]=\bigl[\beta(B)I^\HG_1\bigr]. 
	\end{align} 
	We again note that the argument is still valid when $G$ and $H$ are regular. 
	
	\begin{proof}
		We put $p:=p^{\hat{\bar{\HG}}}$ associated with $\hat{\bar{\HG}}$. 
		Therefore by \cref{cor:properind}, we have 
		\begin{align*}
			&
			\bigl[C^r_0(\hat{G})_1 \Ad W^{\GF}_{12}{}^*\sigma_{12} \bigl((\Ind_{\HG}A)\outensor{F}B\bigr)\bigr]
			\\&
			=
			\bigl[C^r_0(\hat{G})_1 
				\bigl( \Ad W^{\GF}_{12}{}^* \Ad W^{\GF}_{12} ((\id\otimes\hat{\alpha}) \Ind_{\HG}A) \bigr)_{123}
				\bigl(\Ad W^{\GF}_{12}{}^* \beta(B)_{24}\bigr)\bigr]
			\\&
			=
			\bigl[\bigl( (\id\otimes\hat{\alpha}) 
				(I^{\ran\HG}_1 p_1\alpha(A)p_1^* I^{\ran\HG}_1{}^*) \bigr)_{123}
				\bigl((\id\otimes\beta) (\GF^*\beta)B\bigr)_{124}\bigr]
			\\&
			=
			\bigl[I^{\ran\HG}_1 p_1 
				\bigl( \Ad W^{\GF\HG}_{12}{}^* 
				\sigma_{12} (\id\otimes\alpha) \hat{\alpha}(A) \bigr)_{123} 
				p_1^* I^{\ran\HG}_1{}^* 
				\bigl( (\id\otimes\beta) (\GF^*\beta)B \bigr)_{124}\bigr]
			\\&
			=
			\bigl[I^{\ran\HG}_1 p_1 
				\bigl( \Ad W^{\GF\HG}_{12}{}^* 
				\sigma_{12} (\id\otimes\alpha) \hat{\alpha}(A) \bigr)_{123} 
				p_1^* \bigl((\id\otimes\beta)((\ran\HG)^*\GF^*\beta)(B)\bigr)_{124} I^{\ran\HG}_1{}^*\bigr]
			\\&
			=
			\bigl[I^{\ran\HG}_1 p_1 \bigl(\Ad W^{\GF\HG}_{12}{}^* 
				\sigma_{12} (\id\otimes\alpha) \hat{\alpha}(A) \bigr)_{123} 
				\bigl(\Ad W^{\GF\HG}_{12}{}^*\beta(B)_{24}\bigr) 
				p_1^* I^{\ran\HG}_1{}^*\bigr]
			\\&
			=\sigma_{23}\bigl[C^r_0(\hat{G})_1 \Ind_{\HG}(A\outensor{F}B)\bigr]. 
		\end{align*}
		Here for the last three equalities we have used \cref{eq:prop:projfml1}, 
		and 
		$p_1W^{\GF\HG}=p_1(\bar{\HG}^r\otimes\id)W^{\GF\ran\HG}=W^{\GF\ran\HG}p_1$. 
		
		By taking $\clin{(\K(G)^*\otimes\id)\Ad V^{G}_{12}(-)_{1345}}$, 
		it follows 
		$\Ad W^{\GF}_{12}{}^*\sigma_{12} ((\Ind_{\HG}A)\outensor{F}B)
		=\Ind_{\HG}(A\outensor{F}B)$. 
		It is easy to check that this equality preserves left $G$- and $\hat{E}$-coactions. 
	\end{proof}
	
	\subsection{Imprimitivity} \label{ssec:impr}
	
	For a proper homomorphism $\HG:H\to G$ of regular locally compact quantum groups, 
	sometimes we want to know when the functor $\HG_*$ becomes genuine $\Ind_\HG$, for the purpose of calculation. 
	For left $G$-C*-algebras $(A,\alpha)$, $(B,\beta)$ and a left $G$-$(A,B)$-correspondence $(\E,\pi)$ with non-degenerate $\pi$, 
	the canonical non-degenerate left $H$-$*$-homomorphism $\iota:\K_B(\E)\oplus B\to \K_B(\E\oplus B)$ induces 
	an injective non-degenerate left $G$-$*$-homomorphism 
	$\id\otimes\iota:\Ind_\HG \K_B(\E)\oplus \Ind_\HG B \to \Ind_\HG \K_B(\E\oplus B)$. 
	Via this map we identify $\Ind_\HG \K_B(\E)\oplus \Ind_\HG B \subset \Ind_\HG \K_B(\E\oplus B)$. 
	We define 
	$\Ind_\HG \E:= 1_{\Ind_\HG \K_B(\E)} (\Ind_{\HG} \K_B(\E\oplus B)) 1_{\Ind_\HG B}$, 
	and this is a well-defined right $(\Ind_\HG B)$-Hilbert module with a canonical continuous left $G$-coaction 
	(which is different from an induced corepresentation in nature). 
	By definition we also have a non-degenerate left $G$-$*$-homomorphism 
	$\mu:\Ind_\HG \K_B(\E)\to \LC_B(\Ind_\HG \E)$, 
	and $\Ind_\HG (\E,\pi):=(\Ind_\HG \E, \mu(\id\otimes\pi))$ is a well-defined left $G$-$(A,B)$-correspondence with non-degenerate $\Ind_\HG\pi:=\mu(\id\otimes\pi)$. 
	
	When $\HG$ gives a closed quantum subgroup, 
	the imprimitivity theorem for reduced crossed products~\cite[Theorem~7.3]{Vaes:impr} assures $\HG_*A$ is naturally isomorphic to $\Ind_\HG A$ for any $A\in\Cor^H$. 
	For morphisms $(\E,\pi):A\to B$ and $(\F,\varpi):B\to C$ in $\Cor^H$, 
	we can check that 
	$(\Ind_\HG \E) \tensor{\Ind_\HG\varpi}(\Ind_\HG \F)=\Ind_\HG (\E\tensor{\varpi}\F)$ 
	up to a left $G$-unitary equivalence, 
	by the construction of $\Ind_\HG$ using $I^\HG$. 
	
	For a general proper $\HG$, this is no longer true since $\Ind_\HG$ need not preserve equivariant Morita equivalence. 
	For example, consider $\HG:=1_{K\to 1}:K\to 1$ for a compact quantum group $K$. 
	Then the left $K$-C*-algebras $\C$ and $(\K(K), \Ad W^{K}_{12}{}^*(-)_2)$ are left $K$-Morita equivalent, 
	but $\Ind_\HG \C=\C$ and 
	$\Ind_\HG\K(K)\isom \Ind_\HG C^r_0(K)\outensor{\hat{K}^{\op}}C^r_0(\hat{K})\isom C^r_0(\hat{K})$ 
	are not Morita equivalent unless $K$ is trivial. 
	To get the well-behavior of $\Ind_\HG$, we consider a relative variant of freeness as follows. 
	
	\begin{prop}\label{prop:impr}
		Consider a proper homomorphism $\HG:H\to G$ of regular locally compact quantum groups. 
		Let $\Cor^H_{\HG}$ be the full subcategory of $\Cor^H$ whose objects are left $H$-C*-algebras satisfying 
		\begin{align*}
			&
			\clin{\alpha(A)((\lambda_H\HG^{u}C^u_0(G))\otimes A)}=C^r_0(H)\otimes A . 
		\end{align*}
		Then 
		$\Ind_\HG$ gives a well-defined functor $\Cor^H_{\HG}\to\Cor^G$ 
		which is naturally isomorphic to $\HG_*\restriction_{\Cor^H_\HG}$. 
	\end{prop}
	
	We note in the situation of the proposition above, 
	$(A,\alpha)\in\Cor^H$ is an object of $\Cor^H_\HG$ if and only if 
	it is an object of $\Cor^H_{\bar{\HG}}$ for $\bar{\HG}:H\to \Ran\HG$, which is equivalent to say 
	$\clin{\alpha(A)(\bar{\HG}^{r}(C^r_0(\Ran \HG))\otimes A)}=C^r_0(H)\otimes A$. 
	We also note that  
	for a regular locally compact quantum group $F$ and a left $F$-C*-algebra $B$ it holds 
	$\bigl[(F\redltimes B)_{12}
	\Ad \hat{V}^{F}_{13}(F\redltimes B)_{12}\bigr] 
	=(F\redltimes B)\otimes C^r_0(\hat{F})$. 
	In particular, $\hat{H}^{\op}\redltimes H\redltimes A\in \Cor^H_{\HG}$ for any left $H$-C*-algebra $A$. 
	For a right Hilbert $A$-module $\E$ and an $(A,B)$-correspondence $(\F,\varpi)$, 
	we write $\E \underset{\varpi}{\odot} \F$ for the tensor product of bimodule in a purely algebraic sense. 
	
	\begin{prf}
		As we saw, the result holds when $\HG$ gives a closed quantum subgroup. 
		By decomposing $\HG=(\ran \HG)\bar{\HG}$ and using \cref{thm:indstages}, 
		it suffices to consider the case when $\hat{\HG}$ gives an open quantum subgroup. 
		
		Let $A,B\in\Cor^H$. 
		Take a left $H$-Hilbert $A$-module $(\E,\kappa_\E)$ and a morphism $(\F,\varpi):A\to B$ in $\Cor^H$ with its left $H$-coaction $\kappa_\F$. 
		It is easy to see $(\E^\HG,\kappa_\E^\HG)$ is a left $G$-Hilbert $A^\HG$-module, and $(\F^\HG,\varpi):A^\HG\to B^\HG$ is a morphism in $\Cor^H$. 
		We can naturally identify 
		$\E^\HG\tensor{\varpi}\F^\HG \subset (\E\tensor{\varpi}\F)^{\HG}$. 
		
		Now suppose $A\in\Cor^H_{\HG}$. 
		It is enough to show that 
		$(\E\tensor{\varpi}\F)^{\HG}= \E^\HG\tensor{\varpi}\F^\HG$ 
		holds as subsets of $\E\tensor{\varpi}\F$. 
		Indeed, from this it will follow that $\Ind_\HG$ defines a well-defined functor 
		$\Cor^H_\HG\to\Cor^G$, 
		which must be naturally isomorphic to 
		$\HG_*\isom \Ind_\HG (\hat{H}^{\op}\redltimes H\redltimes (-))$ 
		via Baaj--Skandalis duality 
		because $\hat{H}^{\op}\redltimes H\redltimes A\in \Cor^H_\HG$ for any $A\in \Cor^H$. 
		Since $A\in\Cor^H_{\HG}$, we have 
		\begin{align*}
			&
			\clin{\kappa_\E(\E)(\HG^{r}C^r_0(G)\otimes A)}
			=\clin{\kappa_\E(\E)\alpha(A)(\HG^{r}C^r_0(G)\otimes A)}
			\\&
			=\clin{\kappa_\E(\E)(C^r_0(H)\otimes A)}
			=C^r_0(H)\otimes \E . 
		\end{align*}
		Thus when we put $p:=p^{\hat{\HG}}$, 
		it follows 
		\begin{align*}
			&
			\Bigl[\E^\HG \underset{\varpi}{\odot} \F\Bigr]
			=
			\Bigl[(\K(G)^*\otimes \id_{\E\tensor{\varpi}\F})
				\Bigl(\Bigl(C^r_0(G)\otimes \E^\HG\Bigr)
				\underset{\id_{C^r_0(G)}\otimes \varpi}{\odot}
				\Bigl(C^r_0(G)\otimes (\varpi(A)\F)\Bigr)\Bigr)\Bigr]
			\\&
			=
			\Bigl[(\K(G)^*\otimes \id_{\E\tensor{\varpi}\F})
				\Bigl(\Bigl(p_1\kappa_\E(\E)p_1^* 
				(C^r_0(G)\otimes A)\Bigr)
				\underset{\id_{C^r_0(G)}\otimes \varpi}{\odot}
				(C^r_0(G)\otimes \F)\Bigr)\Bigr]
			\\&
			=
			\Bigl[(\K(G)^*\otimes \id_{\E\tensor{\varpi}\F})
				\Bigl((pC^r_0(H)p^*\otimes \E) 
				\underset{\id_{C^r_0(G)}\otimes \varpi}{\odot}
				(C^r_0(G)\otimes \F)\Bigr)\Bigr]
			=\E\tensor{\varpi}\F . 
		\end{align*}
		Therefore when we write $\kappa$ for the left $H$-coaction of $\E\tensor{\varpi}\F$, we obtain 
		\begin{align*}
			&
			(\E\tensor{\varpi}\F)^\HG
			=
			\Bigl[(\K(G)^*\otimes \id_{\E\tensor{\varpi}\F})\Bigl(
				p_1\kappa(\E\tensor{\varpi}\F)p_1^*\Bigr)\Bigr]
			=
			\Bigl[(\K(G)^*\otimes \id_{\E\tensor{\varpi}\F})\Bigl(
				p_1\kappa(\E^\HG\underset{\varpi}{\odot}\F)p_1^* \Bigr)\Bigr]
			\\&
			=
			\Bigl[(\K(G)^*\otimes \id_{\E\tensor{\varpi}\F})\Bigl(
				(C^r_0(G)_1p_1\kappa_\E(\E^\HG)p_1^*) \underset{\id_{C^r_0(G)}\otimes \varpi}{\odot} 
				(p_1 (\HG^rC^r_0(G))_1 \kappa_{\F}(\F) p_1^*) \Bigr)\Bigr] 
			\\&
			=\E^\HG\tensor{\varpi}\F^\HG . 
			\tag*{$\square$}
		\end{align*}
	\end{prf}
	
	\begin{rem}\label{rem:impr}
		Consider a proper homomorphism $\HG:H\to G$ of regular locally compact quantum groups. 
		Let $\Cor^H_\bullet$ be the full subcategory of $\Cor^H$ whose objects are left $H$-C*-algebra $A$, such that 
		$\HG$ fits into some diagram of the form \cref{diagram:square} with proper $\FE$, 
		satisfying \cref{cond:square} 
		and $\FH^*A\in\Cor^F_\FE$. 
		We write \cref{diagram:square}$_A$ for such a diagram. 
		Then $\Ind_\HG$ still gives a well-defined functor $\Cor^H_\bullet\to \Cor^G$ which is naturally isomorphic to 
		$\HG_*\restriction_{\Cor^H_\bullet}$. 
		We can see this by combining the argument of the analogue of base change and the proof of \cref{prop:impr}. 
		Indeed, by (2) of \cref{lem:basechangedecomp}, 
		we may assume $\HG$ and $\FE$ in \cref{diagram:square}$_A$ for each $A\in\Cor^H_\bullet$ give duals of open quantum subgroups. 
		Take a morphism $(\F,\varpi):A\to B$ in $\Cor^H$ 
		with $A\in\Cor^H_\bullet$, 
		and a left $H$-Hilbert $A$-module $\E$. 
		Then we can see $\theta$ in \cref{diagram:square}$_A$ satisfies 
		$(\E\tensor{\varpi}\F)^{\HG}=(\E\tensor{\varpi}\F)^{\FE}
		=\E^\FE\tensor{\varpi}\F^\FE = \E^\HG\tensor{\varpi}\F^\HG$
		by the argument of \cref{prop:impr} for $\Cor^F_\FE$ and (b) of \cref{thm:basechange}. 
		From this the claim follows. 
	\end{rem}
	
	\subsection{Adjointness}\label{ssec:adjoint}
	
	As the last topic of this section, we record the adjointness of induced coactions and restriction. 
	For a while, we fix a proper homomorphism $\HG:H\to G$ of regular locally compact quantum groups 
	which has a \emph{cocompact} image in the sense of $\Ind_\HG \C$ being unital. 
	Suppose $\HG^r$ exists, and 
	let $\CC\in\{\Calg,\Cor,\KK\}$, where we additionally assume separability of $C^r_0(G)$ and $C^r_0(H)$ when $\CC=\KK$. 
	
	When $\HG$ gives a cocompact closed quantum subgroup and $G$ is coamenable, it is proved in~\cite[Proposition~4.7]{Nest-Voigt:eqpd} that 
	there is an adjunction $\HG^*\dashv\HG_*$ as functors between $\KK^G$ and $\KK^H$, 
	as the quantum version of the result for locally compact groups in~\cite{Meyer-Nest:bctri}. 
	We note that their proof is quite analogous to Frobenius reciprocity for unitary representations of a compact group and its closed subgroup 
	(see~\cite{Folland:book} for example). 
	By slightly modifying their construction of the counit denoted by $\kappa$ below, we can loosen the assumption on the coamenability of $G$ into the existence of $\HG^r$ as follows. 
	
	First, let $\CC=\Calg$. 
	Consider a proper homomorphism $\HG:H\to G$ with the conditions above. 
	For all $(A,\alpha)\in\Calg^G$ and $(B,\beta)\in\Calg^H$ we have 
	\begin{align*}
		&
		(\id\otimes\alpha)\Ind_\HG \HG^* A
		= \Ind_\HG (\C_\HG\outensor{G}A) 
		= \Ad W^{G}_{12}{}^*\sigma_{12}((\Ind_\HG \C_\HG)\outensor{G}A)
		\\&
		\supset \Ad W^{G}_{12}{}^*(1_{\K(G)}\otimes\alpha(A))
		= (\id\otimes\alpha)\alpha(A) 
	\end{align*}
	by \cref{prop:projfml}, and 
	\begin{align*}
		&
		(\HG^{r}\otimes \id)\Ind_\HG B
		=\clin{(\K(H)^*\otimes \HG^r\otimes \id) \bigl((\HG^*\Delta_G\otimes \id)\Ind_\HG B\bigr)}
		\\&
		=\clin{(\K(H)^*\otimes \id\otimes\id) \bigl((\HG^{r}\otimes \beta)\Ind_\HG B\bigr)}
		\\&
		=\bigl[(\K(H)^*\otimes \id\otimes\id) (\id\otimes \beta) \bigl((\HG^rC^r_0(G))\otimes B^{\bar{\HG}}\bigr)\bigr]
		\subset \beta(B) 
	\end{align*}
	by \cref{cor:properind}. 
	Thus the well-defined non-degenerate left $G$- and $H$-$*$-homomorphisms 
	$\eta_A : A \isom \alpha(A) \subset \Ind_\HG \HG^* A$ and
	$\kappa_B : \HG^*\Ind_\HG B\xrarr{\HG^{r}\otimes \id_B} \beta(B)\isom B$ 
	give the natural transformations
	$\eta:\id_{\Calg^G}\to \HG_*\HG^*$ and 
	$\kappa:\HG^*\HG_* \to\id_{\Calg^H}$, respectively. 
	For all $(A,\alpha)\in\Calg^G$ and $(B,\beta)\in\Calg^H$, it holds 
	\begin{align*}
		&
		\kappa_{\HG^*A}\circ\HG^*(\eta_{A}) 
		= (\HG^*\alpha)^{-1}(\HG^{r}\otimes \id_A)\circ\alpha=\id_A 
		:\HG^*A\to \HG^*A, 
	\end{align*}
	and
	\begin{align*}
		&
		\Ind_\HG(\kappa_B)\circ\eta_{\Ind_\HG B}
		= (\id_{\K(G)}\otimes(\beta^{-1}(\HG^r\otimes\id_B))) 
		\circ (\Delta_G\otimes\id_B) 
		\\&
		=(\id_{\K(G)}\otimes\beta)^{-1} 
		\circ (\Delta_G\rpb\HG\otimes\id_B) 
		=\id_{\Ind_\HG B} 
		: \Ind_\HG B\to \Ind_\HG B. 
	\end{align*}
	Therefore $\eta$ and $\kappa$ are the unit and counit which assure $\HG^*\dashv\HG_*$ as functors between $\Calg^G$ and $\Calg^H$. 
	
	Now suppose that $\CC=\Cor$ or $\KK$, and 
	that $\HG$ gives a closed quantum subgroup. 
	Then \cref{prop:impr} and the universality of $\KK^H$ show 
	$\HG^*\dashv \HG_*$ as functors between $\CC^G$ and $\CC^H$, as claimed above. 
	
	Next, we seek such adjointness 
	when $\hat{\HG}$ gives an open quantum subgroups and $\CC=\Cor$ or $\KK$. 
	When $\CC=\KK$, 
	it is proved as a special case of~\cite[Theorem~4.8]{Voigt:cpxss} that $\hat{\HG}_*\dashv \hat{\HG}^*$ as functors between $\KK^G$ and $\KK^H$. 
	We remark that the proof still works for $\CC=\Cor$. 
	Thus by Baaj--Skandalis duality it follows $\HG^*\dashv \HG_*$ as functors between $\CC^G$ and $\CC^H$, where $\CC=\Cor$ or $\KK$. 
	
	By combining the arguments so far, we obtain the following. 
	
	\begin{prop}\label{prop:adjoint}
		Let $\HG:H\to G$ be a proper homomorphism of regular locally compact quantum groups with $\Ind_\HG \C$ being unital, 
		and $\CC\in\{\Calg,\Cor,\KK\}$, where we additionally assume separability of $C^r_0(G)$ and $C^r_0(H)$ when $\CC=\KK$. 
		Suppose $\HG^r$ exists. 
		Then $\HG^*\dashv \HG_*$ as functors between $\CC^G$ and $\CC^H$. 
	\qed\end{prop}
	
	As already mentioned, this adjointness in the case of a closed quantum subgroup can be considered as an analogue of Frobenius reciprocity. 
	On the other hand, let $\HG:=1_{K\to 1}:K\to 1$ for a compact quantum group $K$ with separable $C^r_0(K)$, 
	as a special case when $\hat{\HG}$ gives an open quantum subgroup. 
	Then $(1_{K\to1})^*\dashv (1_{K\to1})_*$ means that 
	$KK^K(A,B)\isom KK(A,K\redltimes B)$ holds for any $A\in\KK$ and $B\in\KK^K$, 
	which is exactly (the quantum version of) Green--Julg theorem,~\cite[Th{\'e}or{\`e}me 5.10]{Vergnioux:thesis}. 
	
	Consider a homomorphism $\HG:H\to G$ of regular locally compact quantum groups which gives a closed quantum subgroup. 
	In~\cite{Nest-Voigt:eqpd}, they considered induced coactions along $\HG$ to give a functor $\CC^{\DD(H)}\to \CC^{\DD(G)}$. 
	In our notation, this is given by 
	$(\HG\bowtie\id)_* (\id\bowtie\hat{\HG})^*:\CC^{\DD(H)}\to \CC^{\DD(\HG)}\to\CC^{\DD(G)}$. 
	As an immediate consequence from the argument of \cref{prop:adjoint}, we have the following adjointability. 
	
	\begin{cor}
		Let $\HG :H\to G$ be a homomorphism of regular locally compact quantum groups giving an open quantum subgroup, and
		$\CC\in\{\Calg,\Cor,\KK\}$, where we additionally assume separability of $C^r_0(G)$ and $C^r_0(H)$ when $\CC=\KK$. 
		Then there is a right adjoint for 
		$(\HG\bowtie\id)_* (\id\bowtie\hat{\HG})^*:\CC^{\DD(H)}\to \CC^{\DD(G)}$. 
	\end{cor}
	
	\begin{proof}
		Note that $\HG\bowtie\id$ and the dual of $\id\bowtie\hat{\HG}$ give open quantum subgroups by \cref{eg:doublehom}. 
		Thus the statement follows from the fact that 
		$(\HG\bowtie\id)_*\dashv(\HG\bowtie\id)^*$ and 
		$(\id\bowtie\hat{\HG})^*\dashv(\id\bowtie\hat{\HG})_*$. 
	\end{proof}
	
	\section{Example: double crossed products with a discrete abelian group}\label{sec:eg}
	
	In this section, we stick to the following situation. 
	
	\begin{condition}\label{cond:twistab}
		Let $F,G,H$ be regular locally compact quantum groups and $\Gamma$ be a discrete abelian group. 
		Consider the following sequences of homomorphisms of locally compact quantum groups. 
		\begin{align*}
			\xymatrix@R=0.4em{
				1\ar[r] & 
				\hat{\Gamma} \ar@<0.3ex>[r]^-{\EG} & 
				H \ar[r]^-{\HG}\ar@<0.3ex>[l]^-{\pi} & 
				G \ar[r] & 
				1, 
				\\
				1\ar[r] & 
				\hat{\Gamma} \ar@<0.3ex>[r]^-{\FH} & 
				H \ar[r]^-{\GF}\ar@<0.3ex>[l]^-{\pi} & 
				F \ar[r] & 
				1. 
			}
		\end{align*}
		Moreover, we assume they are ``split exact" (with common $\pi$) in the sense of satisfying the three conditions as follows. 
		\begin{itemize}
			\item
			$\EG, \hat{\HG}, \FH, \hat{\GF}$ give closed quantum subgroups. 
			\item
			$(C^r_0(H)\rpb\EG)^{\hat{\Gamma}}=\HG^rC^r_0(G)$, 
			and $(C^r_0(H)\rpb\FH)^{\hat{\Gamma}}=\GF^rC^r_0(F)$. 
			\item
			$\pi\EG=\id_{\hat{\Gamma}}=\pi\FH$. 
		\end{itemize}
	\end{condition}
	
	\begin{rem}
		Assume the situation of \cref{cond:twistab}. 
		By (3) of \cref{eg:doublehom} we have 
		$\HG^rC^r_0(G)=
		\Ind_{\hat{\Gamma}\to 1}\EG^*C^r_0(H)$, 
		and similarly $\GF^rC^r_0(F)=
		\Ind_{\hat{\Gamma}\to 1}\FH^*C^r_0(H)$. 
		By using this and the commutativity of $C_0(\Gamma)$, 
		it follows for any $x\in (\HG^rC^r_0(G))\otimes C_0(\Gamma)$, 
		\begin{align*}
			&
			\Ad W^{\EG}_{12}{}^*\Ad W^\pi_{23} x_{23}
			=\Ad (\hat{W}^{\Gamma}_{13}W^\pi_{23}W^{\EG}_{12}{}^*) x_{23}
			=\Ad W^\pi_{23} x_{23}, 
			\\
			&
			\Ad W^{\EG}_{12}{}^*\Ad W^\pi_{23}{}^* x_{23}
			=\Ad (W^\pi_{23}{}^*\hat{W}^{\Gamma}_{13}{}^*W^{\EG}_{12}{}^*) x_{23}
			=\Ad W^\pi_{23}{}^* x_{23}. 
		\end{align*}
		Therefore $\Ad W^\pi$ gives a $*$-automorphism on $(\HG^r C^r_0(G))\otimes C_0(\Gamma)$, 
		and induces a matching $\m$ on $G^{\op}$ and $\Gamma$. 
		By the same reason we also have a matching $\n$ on $G^{\op}$ and $\Gamma$ coming from a $*$-automorphism $\Ad W^\pi$ on 
		$(\GF^rC^r_0(F))\otimes C_0(\Gamma)$. 
		By (3) of \cref{eg:doublehom} we see $\hat{\HG\bowtie\id}$ and $\hat{\GF\bowtie\id}$ 
		are well-defined homomorphisms which give open quantum subgroups. 
	\end{rem}
	
	When $F,G,H$ are compact, such a situation arises from a more general treatment in~\cite{Bichon-Neshveyev-Yamashita:grtwi}. 
	Here we give an example coming from finite groups. 
	
	\begin{eg}\label{eg:dihedral}
		Let $n\in\Z_{>0}$. 
		Consider an action of $\Gamma:=\Z/2\Z$ on a group $\hat{G}:=\Z/2n\Z$ given by 
		$\Z/2\Z\ni 1\mapsto \bigl(1\mapsto -1\bigr)\in \Aut(\Z/2n\Z)$. 
		The semi-direct product $\hat{H}:=\Z/2n\Z\rtimes\Z/2\Z$ is the dihedral group $D_{2n}=\bra a,b \,|\, a^{2n}=b^{2}=1, bab^{-1}=a^{-1} \ket$.  
		This has a subgroup $\hat{F}:= D_n\isom \bra a^2, b\ket\subset D_{2n}$. 
		Then we have the following two split exact sequences, 
		\begin{align*}
			\xymatrix@R=0.4em{
				1\ar[r] & 
				\Z/2n\Z \ar[r]^-{\subset} & 
				D_{2n} \ar@<0.3ex>[r] & 
				\Z/2\Z \ar[r]\ar@<0.3ex>[l]^-{\hat{\pi}} & 
				1, 
				\\
				1\ar[r] & 
				D_n \ar[r]^-{\subset} & 
				D_{2n} \ar@<0.3ex>[r] & 
				\Z/2\Z \ar[r]\ar@<0.3ex>[l]^-{\hat{\pi}} & 
				1, 
			}
		\end{align*}
		where $\hat{\pi}:\Z/2\Z\ni 1\mapsto ab\in D_{2n}$. 
		By taking these duals, we get a pair of sequences satisfying \cref{cond:twistab}. 
		Then a calculation shows $\DD(\m)\isom D_{2n}$. 
		We can also see that 
		$C^r_0(\DD(\n))\isom C^*(D_{n})\otimes C(\Z/2\Z)$. 
		We calculate the representation rings of their duals as 
		$R(\hat{D_{2n}})\isom K_0(C(D_{2n}))\isom \Z^{\oplus 4n}$ 
		and $R(\hat{\DD(\n)})\isom K_0(C^*(D_{n})\otimes C(\Z/2\Z))
		\isom R(D_n)^{\oplus 2}$. 
		Here the rank of $R(D_n)$ over $\Z$ is $\frac{n-1}{2}+2$ if $n$ is odd, and $\frac{n}{2}+3$ if $n$ is even. 
		In particular, 
		$\hat{\DD(\m)}$ and $\hat{\DD(\n)}$ 
		are not monoidally equivalent if $n\geq 3$. 
	\end{eg}
	
	Similarly we can also get a situation of \cref{cond:twistab} with $(\hat{G},\hat{H})=(D_{2n}, Q_{4n})$, where 
	$Q_{4n}:=\bra a,c \,|\, a^{2n}=1, a^{n}=c^2, cac^{-1}=a^{-1} \ket$ is the generalized quaternion group. 
	Indeed, by letting 
	\begin{align*}
		&
		\hat{H}:=\bra a,b,c \,|\, a^{2n}=b^2=1, a^{n}=c^2, bab^{-1}=cac^{-1}=a^{-1}, bcb^{-1}=c^{-1} \ket , 
	\end{align*}
	we see $D_{2n}\isom \bra a,b\ket\subset \hat{H}$ and $Q_{4n} \isom \bra a,c\ket \subset \hat{H}$. 
	We can give $\pi:H\to \hat{\Gamma}:=\hat{\Z/2\Z}$ by $\hat{\pi}:\Z/2\Z\ni 1\mapsto bc\in \hat{H}$. 
	
	When $F,G$ are compact quantum groups, various aspects of them are compared in 
	\cite{Bichon-Neshveyev-Yamashita:grtwi},~\cite{Bichon-Neshveyev-Yamashita:grmod}. 
	In terms of double crossed products and induced coactions, 
	we are going to compare them in our settings of \cref{cond:twistab}. 
	See~\cite[Section~3]{Bichon-Neshveyev-Yamashita:grmod} for the closely related result in a more algebraic situation. 
	We will also see such a comparison can be done in the levels of $\CC=\Cor$ and $\KK$. 
	
	\begin{thm}\label{thm:twistab}
		Consider the situation of \cref{cond:twistab} 
		and let $\CC\in\{\Calg,\Cor,\KK\}$, where we additionally assume separability of $C^r_0(H)$ when $\CC=\KK$. 
		Consider the following functors 
		\begin{align*}
			\xymatrix@C=3.5em@R=0.4em{
				T_\m: 
				\CC^{\DD(\pi)} \ar[r]^-{\Gamma\redltimes-} &
				\CC^{\DD(\pi)} 
				\ar[r]^-{(\HG\bowtie\id)_*} &
				\CC^{\DD(\m)}, 
				\\
				T_\n: 
				\CC^{\DD(\pi)} \ar[r]^-{\Gamma\redltimes-} &
				\CC^{\DD(\pi)} 
				\ar[r]^-{(\GF\bowtie\id)_*} &
				\CC^{\DD(\n)}. 
			}
		\end{align*}
		Here for any $\pi$-Yetter--Drinfeld C*-algebra $(A,\alpha,\hat{\alpha})$, we have equipped 
		$\Gamma\redltimes A$ with a structure of $\pi$-Yetter--Drinfeld C*-algebra via 
		the canonical $*$-isomorphism $\Gamma\redltimes A\isom A\outensor{\hat{\Gamma}}C(\hat{\Gamma})$, 
		by considering $C^r_0(\hat{\Gamma})$ 
		as a $\hat{\Gamma}$-Yetter--Drinfeld C*-algebra, 
		with the left $\hat{\Gamma}$-coaction $\Delta_{\hat{\Gamma}}$ and the trivial left $\Gamma$-coaction. 
		Then 
		$T_\n (\HG\bowtie\id_\Gamma)^*$ and $T_\m (\GF\bowtie\id_\Gamma)^*$ are mutually inverse categorical equivalences 
		between $\CC^{\DD(\m)}$ and $\CC^{\DD(\n)}$. 
		If $\CC=\KK$, they are triangulated. 
	\end{thm}
	
	In the proof, the unitary in $C(\hat{\Gamma})=C^*(\Gamma)$ corresponding to each $t\in\Gamma$ is denoted by $u^t:L^2(\Gamma)\ni f\mapsto f(t^{-1}(-))\in L^2(\Gamma)$. 
	Then we have $W^\Gamma=\sum\limits_{t\in\Gamma}\delta_t\otimes u^t\in \U(C_0(\Gamma)\otimes C^*(\Gamma))$. 
	We also canonically identify $\Irr(\hat{\Gamma})=\Gamma$. 
	
	\begin{proof}
		By symmetry, it is enough to show that 
		$T'_\n:=(\Gamma\redltimes-)^{\GF\bowtie\id}:\CC^{\DD(\pi)}\to \CC^{\DD(\n)}$ is a well-defined functor 
		which is naturally isomorphic to $T_\n$, and that 
		$\id_{\CC^{\DD(\m)}}\isom T'_\m(\GF\bowtie\id)^*T'_\n(\HG\bowtie\id)^*$ 
		when we define $T'_\m:\CC^{\DD(\pi)}\to \CC^{\DD(\m)}$ similarly.
		We also write $\tilde{T}_\n:=
		(\Gamma\redltimes-)\circ(\GF\bowtie\id)^*\circ T'_\n :\CC^{\DD(\pi)}\to \CC^{\DD(\pi)}$. 
		
		First, we prove them when $\CC=\Calg$. 
		Clearly $T'_\n$ is well-defined and $T'_\n\isom T_\n$. 
		For any $\pi$-Yetter--Drinfeld C*-algebra $(A,\alpha,\hat{\alpha})$, 
		it holds $A^{\GF\bowtie\id}=A^{\GF}=(\FH^*A)^{\hat{\Gamma}\to 1}$ 
		by \cref{eg:doublebasechange} and \cref{eg:exactbasechange}. 
		We note that the $\pi$-Yetter--Drinfeld structure of $\Gamma\redltimes A$ is given by the left $H$-coaction 
		$u^t_1{}^*\hat{\alpha}(a)\mapsto 
		(\pi^r(u^t{}^*)\otimes u^t{}^*)_{12} (\id\otimes\hat{\alpha})\alpha(a)$
		and the left $\Gamma$-coaction 
		$u^t_1{}^*\hat{\alpha}(a)\mapsto 
		u^t_2{}^*(\id\otimes\hat{\alpha})\hat{\alpha}(a)$, 
		where $t\in\Gamma$, $a\in A$. 
		From \cref{ssec:cptind}, we recall 
		${}_{t}(\FH^*A):=\{a\in A \,|\, 
		(\FH^*\alpha)(a)\in u^t\otimes A \}$
		for each $t\in\Gamma$, 
		and we can describe 
		\begin{align*}
			&
			T'_\n(A)=
			\biggl[\bigoplus\limits_{t\in\Gamma}
			u^t_1{}^*\hat{\alpha}({}_{t}(\FH^*A))\biggr] 
			\subset \M(\K(\Gamma)\otimes A) . 
		\end{align*}
		The left $\DD(\n)$-coaction of $T'_\n(A)$ is canonically determined by the left $\DD(\pi)$-coaction of $(\GF\bowtie\id)^*T'_\n(A)$ 
		as a subset of $\Gamma\redltimes A$. 
		Thus we obtain 
		\begin{align}\label{eq:thm:twistab1}
			&
			\tilde{T}_\n(A)=
			\biggl[\bigoplus\limits_{s,t\in\Gamma}
			(u^s\otimes u^t{}^*)_{12} (\id\otimes\hat{\alpha})\hat{\alpha}({}_{t}(\FH^*A))\biggr]
			\subset \M(\K(\Gamma)\otimes\K(\Gamma)\otimes A). 
		\end{align}
		We have a left $H$- and $\Gamma$-$*$-homomorphism 
		\begin{align}\label{eq:thm:twistab2}
			&
			(\epsilon_{\hat{\Gamma}}\otimes\id\otimes\id) \Ad W^{\Gamma}_{12} : 
			\tilde{T}_\n(A) \ni 
			(u^s\otimes u^t{}^*)_{12} (\id\otimes\hat{\alpha})\hat{\alpha}(a)
			\mapsto
			u^{st^{-1}}_1\hat{\alpha}(a)
			\in \Gamma\redltimes A , 
		\end{align}
		where $s,t\in\Gamma$, $a\in{}_t(\FH^* A)$, and 
		$\epsilon_{\hat{\Gamma}}$ is the counit of $C^*(\Gamma)$. 
		This map is actually a $*$-isomorphism, whose inverse can be given by 
		\begin{align*}
			&
			\Gamma\redltimes A 
			\xrarr{\id\otimes\hat{\alpha}} 
			\M\bigl((\Gamma\redltimes C_0(\Gamma))\otimes A\bigr) 
			\xrarr{\id\otimes\id\otimes (\FH^*\alpha)}
			\M\bigl((\Gamma\redltimes C_0(\Gamma)) 
			\otimes C^*(\Gamma) \otimes A\bigr) 
			\\
			&
			\xrarr{\Ad(\hat{W}^{\Gamma}_{13}{}^*\hat{W}^{\Gamma}_{23})} 
			\M\bigl(\K(\Gamma)\otimes \K(\Gamma)\otimes C^*(\Gamma)\otimes A \bigr)
			\xrarr{\id\otimes\id\otimes\epsilon_{\hat{\Gamma}}\otimes\id}
			\M\bigl(\K(\Gamma)\otimes\K(\Gamma)\otimes A\bigr) , 
		\end{align*}
		where we have used the fact that for any $s,t\in \Gamma$ it holds 
		\begin{align*}
			&
			\Ad (\hat{W}^{\Gamma}_{13}{}^*\hat{W}^{\Gamma}_{23}) (\Delta_{\Gamma}(\delta_s)\otimes u^t)
			=(u^t\otimes u^t{}^*\otimes 1)(\Delta_{\Gamma}(\delta_s)\otimes u^t) . 
		\end{align*}
		By combining \cref{eq:thm:twistab2} with the left $\DD(\m)$-$*$-isomorphism 
		$\hat{\alpha}:A\to (\Gamma\redltimes (\HG\bowtie\id)^*A)^{\HG\bowtie\id}$, 
		we get a natural isomorphism 
		$\eta: (\tilde{T}_\n(\HG\bowtie\id)^*(-))^{\HG\bowtie\id} \xrarr{\sim}\id_{\Calg^{\DD(\m)}}$. 
		Indeed, for any left $\DD(\m)$-$*$-homomorphism $f:A\to B$ in $\Calg^{\DD(\m)}$, we have 
		$(\tilde{T}_\n (\HG\bowtie\id)^*(f))^{\HG\bowtie\id} =\id_{\K(\Gamma)}\otimes\id_{\K(\Gamma)}\otimes f$. 
		
		Next, we show the case of $\CC=\Cor$. 
		We observe $\FH^*(\Gamma\redltimes A)\in \Cor^{\hat{\Gamma}}_{\hat{\Gamma}\to 1}$ 
		for any $\pi$-Yetter--Drinfeld C*-algebra $(A,\alpha,\hat{\alpha})$, since 
		\begin{align*}
			&
			\Bigl[\bigl(\Gamma\redltimes A\bigr)_{23} 
				\Bigl( \bigl( \Delta_{\hat{\Gamma}}C^r_0(\hat{\Gamma}) \bigr)_{12} 
				\bigl( (\id\otimes\hat{\alpha})(\FH^*\alpha)(A) \bigr) \Bigr)\Bigr]
			\\&
			=
			\bigl[\hat{\alpha}(A)_{23} 
				\bigl( C^r_0(\hat{\Gamma})\otimes C^r_0(\hat{\Gamma}) \bigr)_{12} 
				\bigl( (\id\otimes\hat{\alpha})(\FH^*\alpha)(A) \bigr) \bigr] 
			=
			C^r_0(\hat{\Gamma})\otimes (\Gamma\redltimes A). 
		\end{align*}
		Thus \cref{eg:exactbasechange} and \cref{rem:impr} imply $T'_\n$ is a well-defined functor with $T_\n\isom T'_\n$. 
		Now it is enough to show $\id_{\Cor^{\DD(\m)}}\isom T'_\m(\GF\bowtie\id)^*T'_\n(\HG\bowtie\id)^*$. 
		But for any morphism $(\F,\varpi):A\to B$ in $\Cor^{\DD(\m)}$, 
		we can describe $\tilde{T}_\n (\HG\bowtie\id)^*(\F)$ 
		as a corner of \cref{eq:thm:twistab1} for $\K_{B}(\F\oplus B)$. 
		By using this we can check that $\eta$ constructed in the case of $\CC=\Calg$ still gives a natural isomorphism 
		$(\tilde{T}_\n (\HG\bowtie\id)^*(-))^{\HG\bowtie\id} \xrarr{\sim}\id_{\Cor^{\DD(\m)}}$. 
		
		Finally, we consider the case of $\CC=\KK$. 
		Note that the separability of $C^r_0(H)$ implies separability of $C^r_0(G), C^r_0(F), C_0(\Gamma)$ by \cref{cond:twistab}. 
		As a functor of $\Calg^{\DD(\pi)}\to\Calg^{\DD(\n)}$, 
		$T'_\n$ preserves split exact sequences and tensor products with the closed interval $[0,1]$. 
		We also know from the previous case of $\CC=\Cor$ that $T'_\n$ preserves equivariant imprimitivity bimodules. 
		Therefore from the universal property of $\KK^{\DD(\pi)}$ it follows $T'_\n : \KK^{\DD(\pi)}\to\KK^{\DD(\n)}$ is well-defined. 
		It also follows from Baaj--Skandalis duality and $\eta$ in the case of $\CC=\Calg$ 
		that $T'_\n\isom T_\n$ and $(\tilde{T}_\n (\HG\bowtie\id)(-))^{\HG\bowtie\id} \isom\id_{\KK^{\DD(\m)}}$. 
	\end{proof}
	
	\begin{cor}\label{cor:fiberfunc}
		Consider the situation of \cref{cond:twistab}, 
		and suppose $H$ is compact and $\Gamma$ is finite. 
		Let $\CC\in\{\Cor,\KK\}$, where we additionally assume separability of $C^r_0(H)$ when $\CC=\KK$. 
		Then there is a categorical equivalence $\CC^{\hat{\DD(\m)}^{\op}}\simeq \CC^{\hat{\DD(\n)}^{\op}}$ 
		such that it fits into the horizontal part of 
		the following diagram, 
		which commutes up to some natural isomorphism, 
		\begin{align}\label{diagram:fiber}
			\vcenter{\xymatrix{
				\CC^{\hat{\DD(\m)}^{\op}} 
				\ar[drrr]_-{(1_{1\to\hat{\DD(\m)}^{\op}})^*\qquad} & &
				\CC^{\DD(\m)} \ar[ll]_-{\DD(\m)\redltimes-}
				\ar[rr]^-{T_\n(\HG\bowtie\id_\Gamma)^*} & &
				\CC^{\DD(\n)} \ar[rr]^-{\DD(\n)\redltimes-} & & 
				\CC^{\hat{\DD(\n)}^{\op}} \ar[dlll]^-{\qquad(1_{1\to\hat{\DD(\n)}^{\op}})^*} \\
				&&& \CC^{1} &&& . 
			}}
		\end{align}
	\end{cor}
	
	Note that by \cref{eg:dihedral}, $\hat{\DD(\m)}^{\op}$ and $\hat{\DD(\n)}^{\op}$ need not be isomorphic, 
	even in the finite dimensional case. 
	In this sense, the corollary can be interpreted as a kind of no-go theorem for reconstructing quantum groups using $\Cor$ and $\KK$. 
	We remark this categorical equivalence need not send 
	$\C\in \CC^{\hat{\DD(\m)}^{\op}}$ to $\C\in \CC^{\hat{\DD(\n)}^{\op}}$, since 
	the representation rings of $\hat{\DD(\m)}^{\op}$ and $\hat{\DD(\n)}^{\op}$ can be different. 
	
	\begin{proof}
		By Baaj--Skandalis duality and \cref{thm:twistab}, 
		the statement reduces to the commutativity of the diagram \cref{diagram:fiber} up to some natural isomorphism. 
		It suffices to show that the left adjoints of two composed functors $\CC^{\DD(\m)}\to \CC^{1}$ in the \cref{diagram:fiber} exist and are naturally isomorphic. 
		
		Note that $(1_{1\to\hat{\DD(\m)}^{\op}})^*(\DD(\m)\redltimes-)\isom (1_{\DD(\m)\to 1})_* : \CC^{\DD(\m)}\to \CC^{1}$ admits its left adjoint, and 
		this is given by 
		$(1_{\DD(\m)\to 1})^*:\CC^1\to \CC^{\DD(\n)}$ 
		because of Green--Julg theorem. 
		By symmetry we also have $(1_{\DD(\n)\to 1})^* \dashv (1_{1\to\hat{\DD(\n)}^{\op}})^*(\DD(\n)\redltimes-)$. 
		Therefore it is enough to show $
		(1_{\DD(\m)\to 1})^*\isom 
		T'_\m(\GF\bowtie\id)^*(1_{\DD(\n)\to 1})^* 
		: \CC^1\to \CC^{\DD(\m)}$ 
		by \cref{thm:twistab}. 
		But by the construction of $T'_\m$, it is easy to see that 
		for any $A\in\CC^{\DD(\pi)}$ whose coaction is trivial, 
		$T'_\m(A)\in \CC^{\DD(\m)}$ also has a trivial coaction and canonically $*$-isomorphic to $A$. 
	\end{proof}
	
	Finally, we briefly discuss the implications for the quantum analogue of torsion phenomena and Baum--Connes conjecture. 
	Let $\Lambda$ be a discrete quantum group. 
	In the context of the quantum analogue of Baum--Connes conjecture, 
	finite dimensional left $\hat{\Lambda}$-C*-algebras up to left $\hat{\Lambda}$-Morita equivalence 
	are considered to reflect the torsion phenomena of $\Lambda$ (or $\Lambda^{\op}$). 
	In~\cite{Meyer:homological}, torsion-freeness for $\Lambda$ is defined 
	to state the quantum analogue of Baum--Connes conjecture 
	for such a discrete quantum group. 
	Recently a formulation of the quantum analogue of Baum--Connes conjecture 
	for $\Lambda$ which need not be torsion-free 
	is proposed in~\cite{Arano-Skalski:bcconj}. 
	Both works use certain decompositions of the triangulated category $\KK^{\hat{\Lambda}}$. 
	
	Now consider the situation of \cref{cond:twistab} with $H$ being compact and $\Gamma$ being finite, 
	and we have $\CC^{\DD(\m)}\simeq\CC^{\DD(\n)}$ as in \cref{thm:twistab}. 
	It is easy to see $\Cor^{\DD(\m)}\simeq\Cor^{\DD(\n)}$ preserve finite dimensionality of equivariant C*-algebras, 
	and there is a bijective correspondence between 
	the equivariant Morita equivalence classes of finite dimensional left $\DD(\m)$-C*-algebras, and those for $\DD(\n)$, 
	which preserves finite direct sums. 
	Also when $C^r_0(H)$ is separable, 
	we can see the triangulated equivalence $\KK^{\DD(\m)}\simeq \KK^{\DD(\n)}$ preserves 
	the decompositions of the two triangulated categories given by \cite[Corollary~5.1]{Arano-Skalski:bcconj}.

	\appendix
	\section{Constructions for locally compact quantum groups and coactions}\label{sec:appendix}
	In this section we gather some constructions which are not directly related to the general theory of induced coactions. 
	For a C*-algebra $A$, we write $\Aut(A):=\{f:A\to A\,|\, f$\ is a $*$-automorphism $\}$. 
	\subsection{Double crossed products and homomorphisms}\label{ssec:doublecrossed}
	We collect some definitions and properties about double crossed products of locally compact quantum groups by following~\cite{Baaj-Vaes:doublecrossed}. 
	
	\begin{defn}\label{def:matching}
		Let $G,H$ be locally compact quantum groups and 
		$\m\in \Aut(L^\infty(H)\barotimes L^{\infty}(G))$. 
		\begin{enumerate}
			\item
			The map $\m$ is called a \emph{matching} of $H$ and $G$ if 
			\begin{align*}
				&
				(\Delta_H\barotimes \id)\m=\m_{23}\m_{13}(\Delta_H\barotimes \id), 
				\quad\mathrm{and}\quad
				(\id\barotimes \Delta_G)\m=\m_{13}\m_{12}(\id\barotimes \Delta_G). 
			\end{align*}
			\item
			When $\m$ is a matching of $H$ and $G$, 
			the well-defined von Neumann bialgebra
			\[\DD(\m):=
			( L^\infty(H)\barotimes  L^\infty(G), 
			(\id\barotimes (\sigma\m)\barotimes \id)
			(\Delta_H^{\cop}\barotimes \Delta_G) )\] 
			is called the \emph{double crossed product associated with $\m$}. 
		\end{enumerate}
	\end{defn}
	
	When $\m$ is a matching of locally compact quantum groups $H$ and $G$, the von Neumann bialgebra $\DD(\m)$ is a locally compact quantum group by~\cite[Theorem~5.3]{Baaj-Vaes:doublecrossed}. 
	If $G,H$ are regular, then $\DD(\m)$ is also regular by~\cite[Proposition~9.3]{Baaj-Vaes:doublecrossed}. 
	If $G,H$ are regular and the restriction of $\m$ gives a $*$-automorphism on $C^r_0(H)\otimes C^r_0(G)$, 
	then $C^r_0(\DD(\m))=C^r_0(H)\otimes C^r_0(G)$ 
	(as subsets of $L^\infty(H)\barotimes L^{\infty}(G)$) 
	by~\cite[Proposition~9.2]{Baaj-Vaes:doublecrossed}. 
	
	\begin{eg}\label{eg:quantumdouble}
		Let $\HG:H\to G$ be a homomorphism of locally compact quantum groups. 
		Then $\m:=\Ad W^{\HG} \in \Aut(C^r_0(H)\otimes C^r_0(\hat{G}))$ 
		gives a matching of $H^{\op}$ and $\hat{G}$. 
		Then we have a locally compact quantum group $\DD(\Ad W^{\HG})$ called the \emph{generalized quantum double}, 
		for which we simply write $\DD(\HG)$. 
		When $\HG=\id_G$ we also write $\DD(G)=\DD(\id_G)$ for short, which is called the quantum double of $G$. 
		If $G,H$ are regular, then 
		$C^r_0(\DD(\HG))=C^r_0(H)\otimes C^r_0(\hat{G})$. 
		See Section~8 of~\cite{Baaj-Vaes:doublecrossed} for detail. 
	\end{eg}
	
	Before stating the next proposition, we note some facts needed from~\cite{Baaj-Vaes:doublecrossed}. 
	For a matching $\m$ of locally compact quantum groups $H$ and $G$, there is a locally compact quantum group 
	$\G=(H\ltimes L^\infty(G), \Delta_{\G})$ called the \emph{bicrossed product} associated with $\m$, 
	where the left coaction of $H$ on $L^\infty (G)$ is given by $\m(-)_2$. 
	See~\cite{Vaes-Vainerman:bicrossed} for detail. 
	We write $J^{\m}:=J^{\G}$ and $\hat{J}^{\m}:=\hat{J}^{\G}$. 
	Then the unitary 
	$Z^{\m}:=J^{\m}\hat{J}^{\m}(\hat{J}^{H}J^{H}\otimes \hat{J}^{G}J^{G})$ 
	in $\B(L^2(H)\otimes L^2(G))$ satisfies $\Ad Z=\m$. 
	The multiplicative unitary of $\DD(\m)$ can be described as 
	$W^{\DD(\m)}:=V^H_{31}{}^*Z^{\m}_{34}{}^*W^G_{24}Z^{\m}_{34}$ 
	in $\B((L^2(H)\otimes L^2(G))^{\otimes 2})$. 
	Also, the unitary $U^{\m}:=J^{\m}(\hat{J}^{H}\otimes J^G)$ satisfies 
	$U^{\m}\in L^\infty(H)\barotimes\B(L^2(G))$, 
	$(\Delta_H\otimes\id)(U^{\m})=U^{\m}_{23}U^{\m}_{13}$, and
	$\m(1\otimes x)=\Ad U^{\m}(1\otimes x)$ for any $x\in L^\infty(G)$. 
	
	\begin{prop}\label{prop:doublehom}
		Let $F,G,H$ be locally compact quantum groups, $\m$ be a matching of $G$ and $F$, and $\n$ be a matching of $H$ and $F$. 
		Suppose a homomorphism $\HG:H\to G$ satisfies 
		\begin{align}\label{eq:prop:doublehom1}
			(\Delta_{G}\rpb\HG\barotimes\id)\m
			=\n_{23}\m_{13}(\Delta_{G}\rpb\HG\barotimes\id). 
		\end{align}
		Let (P) be each one of the following properties about homomorphisms: 
		being proper, giving a closed quantum subgroup, and giving the dual of an open quantum subgroup. 
		Then the followings hold. 
		\begin{enumerate}
			\item
			Then there is a well-defined homomorphism $\HG^{\op}\bowtie\id_F:\DD(\n)\to\DD(\m)$
			satisfying 
			\begin{align}\label{eq:prop:doublehom2}
				(\HG^{\op}\bowtie\id_F)^*\Delta_{\DD(\m)}
				=(\id\barotimes(\sigma\m)\barotimes\id)(\HG^{\op}{}^*\Delta_{G^{\op}}\barotimes\Delta_F) , 
			\end{align}
			\begin{align}\label{eq:prop:doublehom3}
				\Delta_{\DD(\m)}\rpb(\HG^{\op}\bowtie\id_F)
				=(\id\barotimes(\sigma\n)\barotimes\id)(\Delta_{G^{\op}}\rpb\HG^{\op}\barotimes\Delta_F) . 
			\end{align}
			\item
			If $\hat{\HG}$ has (P), then $\hat{\HG^{\op}\bowtie\id}$ also has (P). 
			\item
			Suppose 
			$(\HG^*\Delta_G\otimes\id)U^{\m}=U^{\m}_{23}U^{\n}_{13}$. 
			If $\HG$ has (P), 
			then $\HG^{\op}\bowtie\id$ also has (P). 
		\end{enumerate}
	\end{prop}
	
	\begin{proof}
		First, the unitaries 
		$Y:=Z^{\m}_{23}{}^*W^F_{13}Z^{\m}_{23} 
		\in\B(L^2(F)\otimes L^2(G)\otimes L^2(F))$ and 
		$X:=V^{\HG}_{31}{}^*Y_{234} 
		\in\B(L^2(H)\otimes L^2(F)\otimes L^2(G)\otimes L^2(F))$ satisfy 
		\begin{align}\label{eq:prop:doublehom4}
			&
			\begin{aligned}
				&
				X_{1234}{}^*W^{\DD(\m)}_{3456}X_{1234}
				=
				Y_{234}{}^*
				V^{\HG}_{31}V^{G}_{53}{}^*V^{\HG}_{31}{}^*
				Y_{456}Y_{234}
				\\&
				=
				Y_{234}{}^*
				V^{\HG}_{51}{}^*V^{G}_{57}V^{G}_{57}{}^*
				V^{G}_{53}{}^*
				Y_{456}Y_{234}
				=
				V^{\HG}_{51}{}^*V^{G}_{57}
				Y_{234}{}^*
				V^{G}_{37}V^{G}_{53}{}^*V^{G}_{37}{}^*
				Y_{456}Y_{234}
				\\&
				=
				V^{\HG}_{51}{}^*V^{G}_{57}
				W^{\DD(\m)}_{7234}{}^*W^{\DD(\m)}_{3456}W^{\DD(\m)}_{7234}
				=
				V^{\HG}_{51}{}^*V^{G}_{57}
				W^{\DD(\m)}_{7256}W^{\DD(\m)}_{3456}
				=
				X_{1256}W^{\DD(\m)}_{3456}.
			\end{aligned}
		\end{align}
		We write $\Delta_L$ and $\Delta_R$ for the right hand sides of \cref{eq:prop:doublehom2} and \cref{eq:prop:doublehom3}, respectively. 
		Then we have 
		\begin{align*}
			&
			\Ad X^*(-)_{34}
			=\sigma_{23}\m_{24}^{-1}(\id\otimes\Delta_{F})\m_{23}(\HG^{\op}{}^*\Delta_{G}^{\cop}\otimes\id)
			=\Delta_L
			\quad\mathrm{on}\quad L^\infty(\DD(\m)) .
		\end{align*}
		It follows from \cref{eq:prop:doublehom4} 
		$(\Delta_L\barotimes\id)\Delta_{\DD(\m)}=(\id\barotimes\Delta_{\DD(\m)})\Delta_L$. 
		We can also calculate by the definition of a matching and \cref{eq:prop:doublehom1} that 
		$(\Delta_{\DD(\n)}\barotimes\id)\Delta_L
		=(\id\barotimes\Delta_L)\Delta_L$ on $L^\infty(\DD(\m))$. 
		This shows
		\begin{align}\label{eq:prop:doublehom5}
			&
			\begin{aligned}
				&
				W^{\DD(\n)}_{1234}{}^*X_{3478}W^{\DD(\n)}_{1234}
				=(\Delta_{\DD(\n)}{}_{12}\Delta_L{}_{12}(W^{\DD(\m)}))(W^{\DD(\m)}_{5678}{}^*)
				=(\Delta_L{}_{34}\Delta_L{}_{12}(W^{\DD(\m)}))(W^{\DD(\m)}_{5678}{}^*)
				\\&
				=(\Delta_L{}_{34}(X_{1256}W^{\DD(\m)}_{3456}))(W^{\DD(\m)}_{5678}{}^*)
				=X_{1278}X_{3478}. 
			\end{aligned}
		\end{align}
		By \cref{eq:prop:doublehom4} and \cref{eq:prop:doublehom5}, 
		it follows $X\in\M(C^r_0(\DD(\n))\otimes C^r_0(\hat{\DD(\m)}))$ and $X$ is a bicharacter (cf.~\cite[Lemma~3.4]{Meyer-Roy-Woronowicz:qghom}). 
		Therefore there is a unique homomorphism $\Phi:\DD(\n)\to\DD(\m)$
		such that $\Phi^*\Delta_{\DD(\m)}=\Delta_L$ and $X=W^\Phi$. 
		It follows 
		that $(\Delta_{\DD(\m)}\rpb\Phi\barotimes\id)\Delta_{\DD(\m)}
		=(\id\barotimes\Phi^*\Delta_{\DD(\m)})\Delta_{\DD(\m)}$. 
		We can also see by using \cref{eq:prop:doublehom1} to see 
		$(\Delta_R\barotimes\id)\Delta_{\DD(\m)}=(\id\barotimes\Delta_L)\Delta_{\DD(\m)}$, 
		and thus $(\Delta_R\barotimes\id)\Delta_{\DD(\m)}=(\Delta_{\DD(\m)}\rpb\Phi\barotimes\id)\Delta_{\DD(\m)}$. 
		Since $\left((\id\otimes\K(\DD(\m))^*)\Delta_{\DD(\m)}(L^\infty(\DD(\m)))\right)''=L^\infty(\DD(\m))$, 
		it follows $\Delta_R=\Delta_{\DD(\m)}\rpb\Phi$. 
		Thus (1) is proved. 
		
		We show (2). 
		When $\hat{\HG}$ is proper, the well-defined normal $*$-homomorphism 
		$\HG^r\barotimes\id:L^\infty(G)\barotimes L^\infty(F)\to L^\infty(H)\barotimes L^\infty(F)$ 
		clearly satisfies 
		$\Delta_L=(\HG^r\barotimes\id\barotimes\id\barotimes\id)\Delta_{\DD(\m)}$ and therefore $\HG^r\barotimes\id$ gives $\Phi^r$. 
		If moreover $\HG^r$ is surjective or injective, so is $\HG^r\barotimes\id$. 
		
		We consider (3). 
		We assume $\HG$ is proper and we show $\Phi$ is proper. 
		Since $V^\HG=(\hat{\HG}^r{}'\otimes\id)V^{H}$, 
		the comparison of $(\omega\otimes\id\otimes\id) W^{\DD(\n)}$ and $(\omega\otimes\id\otimes\id)W^{\Phi}$ for each $\omega\in\B(L^2(\DD(\n)))_*$ 
		shows it is enough to show 
		that $\Ad (Z^{\m}{}^*Z^{\n}):\Ad Z^{\n}{}^*L^\infty(\hat{F})_2\to \Ad Z^{\m}{}^*L^\infty(\hat{F})_2$ 
		and $\hat{\HG}^r{}':L^\infty(\hat{H})'\to L^\infty(\hat{G})'$ 
		extend to some well-defined normal $*$-homomorphism 
		\begin{align}\label{eq:prop:doublehom6}
			&
			L^\infty(\hat{\DD(\n)})
			=
			(L^\infty(\hat{H})')_1 \vee \Ad Z^{\n}{}^*L^\infty(\hat{F})_2
			\to 
			(L^\infty(\hat{G})')_1 \vee \Ad Z^{\m}{}^*L^\infty(\hat{F})_2
			=
			L^\infty(\hat{\DD(\m)}) . 
		\end{align}
		Indeed, this map will give $\hat{\Phi}^r$. 
		Then if $\hat{\HG}^r$ is surjective, $\hat{\Phi}^r$ will be also surjective. 
		
		When we write $\G$ for the bicrossed product associated with $\m$, we have for all $x\in L^\infty(\hat{G})'$, 
		\begin{align*}
			&
			J^{\m}Z^{\m}x_1^*Z^{\m}{}^*J^{\m}
			=
			R^{\G}(R^{\hat{G}}(\hat{J}^{G}x^*\hat{J}^{G})_1)
			\in L^{\infty}(\hat{G})_1\vee \Ad U^{\m}L^{\infty}(F)_2 =L^\infty(\G), 
		\end{align*}
		and for all $y\in L^\infty(\hat{F})$, 
		\begin{align*}
			&
			J^{\m}Z^{\m}(\Ad Z^{\m}{}^*y_2)^*Z^{\m}{}^*J^{\m}
			=U^{\m} R^{\hat{F}}(y)_2 U^{\m}{}^* \in \Ad U^{\m}L^\infty(\hat{F})_2. 
		\end{align*}
		Similar equalities hold for $\n$. 
		By~\cite[Proposition~5.45]{Kustermans-Vaes:lcqg}, 
		it is enough to show that 
		$\HG^r$ and $\id_{\B(L^2(F))}$ extend to some well-defined normal $*$-homomorphism
		$H\underset{\Ad U^{\n}}{\ltimes}\B(L^2(F)) \to G\underset{\Ad U^{\m}}{\ltimes}\B(L^2(F))$ via the commutative diagram 
		\begin{align*}
			\xymatrix@R=1.5em{
				L^\infty(\hat{\DD(\n)})
				\ar@{-->}[r]^-{\cref{eq:prop:doublehom6}}\ar[d]_-{J^{\n}Z^{\n}(-)^*Z^{\n}{}^*J^{\n}} &
				L^\infty(\hat{\DD(\m)})\ar[d]^-{J^{\m}Z^{\m}(-)^*Z^{\m}{}^*J^{\m}}
				\\
				H\underset{\Ad U^{\n}}{\ltimes}\B(L^2(F)) \ar[r] & G\underset{\Ad U^{\m}}{\ltimes}\B(L^2(F)) . 
			}
		\end{align*}
		If $W^{\HG}_{12}{}^* U^{\m}_{23}W^{\HG}_{12}=U^{\m}_{23}U^{\n}_{13}$, 
		it follows that 
		\begin{align*}
			&
			L^\infty(\hat{H})_1\vee \Ad U^{\n}\B(L^2(F))_2
			\xrarr{\Ad (W^{\HG}_{12}U^{\m}_{23})(-)_{13}}
			(\hat{\HG}^*\Delta_{\hat{H}}L^\infty(\hat{H}))_{21}\vee \Ad U^{\m}_{23}\B(L^2(F))_{3}
			\\&
			\xrarr{\hat{\HG}^r\barotimes\id}
			(\Delta_{\hat{G}}L^\infty(\hat{G}))_{21}\vee \Ad U^{\m}_{23}\B(L^2(F))_{3}
			\xrarr{\Ad \hat{V}^{G}_{21}{}^*}
			1\otimes \left(L^\infty(\hat{G})\vee \Ad U^{\m}\B(L^2(F))\right)
		\end{align*} 
		gives the desired map. 
		Finally, from this map and the commutative diagram above we can see that $\hat{\Phi}^r$ is injective if so is $\hat{\HG}^r$. 
	\end{proof}
	
	\begin{rem}\label{rem:doublehom}
		Let $F,G,H$ be locally compact quantum groups and $\m$ be a matching on $H$ and $G$. 
		\begin{enumerate}
			\item
			$\m^{\cop}:=\sigma\m\sigma \in \Aut(L^\infty(G)\barotimes L^\infty(H))$ 
			is a matching on $G$ and $H$. 
			Moreover, the $*$-isomorphism 
			$\sigma:L^\infty(H)\barotimes L^\infty(G)
			\xrarr{\sim}L^\infty(G)\barotimes L^\infty(H)$ gives 
			$\DD(\m^{\cop})\isom \DD(\m)^{\op}$ as locally compact quantum groups. 
			\item
			Let $\n$ be a matching on $H$ and $F$ and $\GF:G\to F$ be a homomorphism such that 
			$(\id\barotimes\Delta_F\rpb\GF)\n=\m_{13}\n_{12}$. 
			Then the matching $\n^{\cop}$ on $F$ and $H$, and the matching $\m^{\cop}$ on $G$ and $H$ satisfy 
			$(\Delta_F\rpb\GF\barotimes\id)\n^{\cop}=\m^{\cop}_{23}\n^{\cop}_{13}(\Delta_F\rpb\GF\barotimes\id)$. 
			Therefore by \cref{prop:doublehom}, there is a homomorphism $\GF^{\op}\bowtie\id:\DD(\m^{\cop})\to \DD(\n^{\cop})$. 
			Via isomorphism of locally compact quantum groups in (1) we get a homomorphism 
			$\DD(\m)\to \DD(\n)$, for which we write $\id\bowtie\GF$. 
		\end{enumerate}
	\end{rem}
	
	\begin{eg}\label{eg:doublehom}
		Let $F,G,H$ be locally compact quantum groups. 
		\begin{enumerate}
			\item
			When $\m$ is a matching on $H$ and $G$, 
			it is shown in~\cite{Baaj-Vaes:doublecrossed} that $H^{\op}$ and $G$ are closed quantum subgroups of $\DD(\m)$. 
			This can be also seen as the special case of \cref{prop:doublehom}. 
			Indeed, since if we consider $\id_{L^\infty(G)}$ as a matching on $1$ and $G$, we get $U^{\id_{L^\infty(G)}}=J^{G}J^{G}=1$ 
			and the assumptions in \cref{prop:doublehom} are satisfied in a trivial way. 
			This shows $G=\DD(\id_{L^\infty(G)})\to\DD(\m)$ gives a closed quantum subgroup, 
			and similarly for $H^{\op}$ by the previous remark. 
			\item
			Let $\HG:H\to G$, $\GF:G\to F$ be homomorphisms. 
			Consider the matching $\m:=\Ad W^{\GF}$ on $G^{\op}$ and $\hat{F}$, 
			and the matching $\n:=\Ad W^{\GF\HG}$ on $H^{\op}$ and $\hat{F}$. 
			We see 
			$((\sigma\HG^*\Delta_G)\otimes\id)W^{\GF}=W^{\GF\HG}_{23}W^{\GF}_{13}$. 
			Also~\cite[Proposition~8.1]{Baaj-Vaes:doublecrossed} says 
			$U^{\m}=W^{\GF}(\hat{J}^{G^{\op}}\otimes \hat{J}^{F})W^{\GF}{}^*(\hat{J}^{G^{\op}}\otimes \hat{J}^{F})
			=W^{\GF}V^{\GF}_{21}$, 
			and similarly $U^{\n}=W^{\GF\HG}V^{\GF\HG}_{21}$. 
			Here note that the modular conjugation of the left Haar weight on 
			$(L^\infty(\hat{G^{\op}}),\Delta_{\hat{G^{\op}}})
			=(L^\infty(\hat{G})', 
			\hat{J}^G(\Delta_{\hat{G}}(\hat{J}_G(-)\hat{J}^G))\hat{J}^G)$ equals to that of $G$. 
			Thus 
			\begin{align*}
				&
				W^{\HG}_{12}{}^*U^{\m}_{23}W^{\HG}_{12}
				=W^{\GF\HG}_{13}W^{\GF}_{23}V^{\GF\HG}_{31}V^{\GF}_{32}
				=W^{\GF\HG}_{13}V^{\GF\HG}_{31}W^{\GF}_{23}V^{\GF}_{32}
				=U^{\n}_{13}U^{\m}_{23} . 
			\end{align*}
			Therefore \cref{prop:doublehom} gives a homomorphism 
			$\HG\bowtie\id:\DD(\GF\HG)\to\DD(\GF)$, and
			$\HG\bowtie\id$ or its dual has the same properties (P) in \cref{prop:doublehom} if so does $\HG$ or $\hat{\HG}$. 
			Similarly, we also have a homomorphism 
			$\id\bowtie\hat{\GF}:\DD(\GF\HG)\to\DD(\HG)$. 
			\item 
			Consider a sequence of homomorphisms of locally compact quantum groups $1\to K\xrarr{\iota} H\xrarr{\HG} G\to 1$. 
			Let $\m$ be a matching on $G$ and $F$, and $\n$ be a matching on $H$ and $F$ satisfying \cref{eq:prop:doublehom1}. 
			Suppose $K$ is compact and this sequence is \emph{exact} in the sense of satisfying the two conditions next. 
			\begin{itemize}
				\item
				$\iota$ and $\hat{\HG}$ give closed quantum subgroups. 
				\item
				$C^r_0(H/K)=\HG^r(C^r_0(G))$. 
			\end{itemize}
			Here we note 
			$\HG^rC^r_0(G)=\HG^rR^GC^r_0(G)=R^H\HG^rC^r_0(G)=
			R^H((C^r_0(H)\rpb\iota)^K)=\Ind_{K\to 1}\iota^*C^r_0(H)$. 
			Then by the normality of \cref{eq:coten1}, we have
			$L^\infty(H/K)=\HG^r(L^\infty(G))$. 
			It follows that $\hat{\HG}$ gives an open quantum subgroup 
			by~\cite[Theorem~7.2]{Kalantar-Kasprzak-Skalski:open}
			and~\cite[Theorem~2.11]{Vaes-Vainerman:lowdim}. 
			By putting $\tilde{\iota}:=(\id_{H^{\op}}\bowtie1_{1\to F})\iota^{\op}$, 
			we show the sequence of induced homomorphisms 
			\begin{align*}
				\xymatrix@C=3em{
					1\ar[r] & 
					K^{\op}\ar[r]^-{\tilde{\iota}} &
					\DD(\n)\ar[r]^-{\HG^{\op}\bowtie\id_{F}} &
					\DD(\m)\ar[r] &
					1
				}
			\end{align*}
			is also exact in the sense above. 
			Indeed, $\tilde{\iota}$ and $\hat{\HG^{\op}\bowtie\id_{F}}$ 
			give closed quantum subgroups by (1) and \cref{prop:doublehom}. 
			Since $\HG^{\op}\bowtie\id_{F}$ has $(\HG^{\op}\bowtie\id_{F})^r$, it is easy to show 
			$\Ind_{K\to1}\tilde{\iota}^*C^r_0(\DD(\n))
			=(\HG^{\op}\bowtie\id_{F})^r C^r_0(\DD(\m))$. 
			Then by the similar reasoning as above shows $C^r_0(\DD(\n)/K^{\op})
			=(\HG^{\op}\bowtie\id_{F})^r C^r_0(\DD(\m))$. 
			Especially, it follows $\hat{\HG^{\op}\bowtie\id_{F}}$ gives an open quantum subgroup. 
		\end{enumerate}
	\end{eg}
	
	\subsection{Yetter--Drinfeld C*-algebras and twisted tensor products}\label{ssec:yd}
	In this section, we collect some definitions and properties about 
	Yetter--Drinfeld C*-algebras, 
	as another picture for left coactions of a generalized quantum double, 
	and their twisted tensor products by following~\cite{Nest-Voigt:eqpd},~\cite{Roy:double},~\cite{Meyer-Roy-Woronowicz:twiten}. 
	
	\begin{defn}\label{def:yd}
		Let $\HG:H\to G$ be a homomorphism of locally compact quantum groups. 
		\begin{enumerate}
			\item
			Let $A$ be a C*-algebra, $\alpha$ be a continuous left $H$-coaction on $A$ and $\hat{\alpha}$ be a continuous left $\hat{G}$-coaction on $A$. 
			Then $A=(A,\alpha,\hat{\alpha})$ is a \emph{(left) $\HG$-Yetter--Drinfeld C*-algebra} if the following diagram commutes, 
			
			\begin{align*}
				\xymatrix@R=0.4em{
					& \M({C^r_0(H)\otimes A}) \ar[r]^-{\id\otimes \hat{\alpha}} & 
					\M({C^r_0(H)\otimes C^r_0(\hat{G})\otimes A}) \ar@(r,u)[dr]^-{\Ad W^{\HG}_{12}} &
					\\
					A\ar@(ur,l)[ur]^-{\alpha}\ar@(dr,l)[dr]_-{\hat{\alpha}} &&&
					\M({C^r_0(H)\otimes C^r_0(\hat{G})\otimes A}). 
					\\
					& \M({C^r_0(\hat{G})\otimes A}) \ar[r]_-{\id\otimes \alpha} & 
					\M({C^r_0(\hat{G})\otimes C^r_0(H)\otimes A}) \ar@(r,d)[ur]_-{\sigma\otimes \id} &
				}
			\end{align*}
			\item
			Let $(A,\alpha,\hat{\alpha})$ and $(B,\beta,\hat{\beta})$ be a $\HG$-Yetter--Drinfeld C*-algebra, 
			$\E$ be a right Hilbert $A$-module, 
			and $\kappa$ and $\hat{\kappa}$ be continuous left $H$- and $\hat{G}$-coaction on $\K_B(\E\oplus B)$ 
			that are compatible with $\alpha$ and $\hat{\alpha}$, respectively. 
			Then $\E=(\E,\kappa,\hat{\kappa})$ is a \emph{$\HG$-Yetter--Drinfeld Hilbert $B$-module} 
			if $\K_B(\E\oplus B)$ is a $\HG$-Yetter--Drinfeld C*-algebra. 
			An $(A,B)$-correspondence $(\E,\pi)$ is a \emph{$\HG$-Yetter--Drinfeld $(A,B)$-correspondence}
			if $\E$ is a $\HG$-Yetter--Drinfeld Hilbert $B$-module and 
			$\pi:A\to\LC_B(\E)$ is a left $H$- and $\hat{G}$-$*$-homomorphism. 
			We abbreviate the word $\id_G$-Yetter--Drinfeld as $G$-Yetter--Drinfeld. 
		\end{enumerate} 
	\end{defn}
	
	For a left $\DD(\HG)$-C*-algebra $A$, 
	we can get a left $H$- and $\hat{G}$-coactions of $A$ 
	by the restrictions along the closed quantum subgroups $H$ and $\hat{G}$ of $\DD(\HG)$, respectively. 
	Then $A$ becomes a $\HG$-Yetter--Drinfeld C*-algebra.  
	Conversely, for a $\HG$-Yetter--Drinfeld C*-algebra $(A,\alpha,\hat{\alpha})$, it holds 
	$(A,(\id\otimes\hat{\alpha})\alpha)$ is a left $\DD(\HG)$-C*-algebra. 
	Thus a left $\DD(\HG)$-C*-algebra and a $\HG$-Yetter--Drinfeld C*-algebra are equivalent notions. 
	See~\cite[Proposition~7.6]{Roy:double} for detail. 
	A $*$-homomorphism of left $\DD(\HG)$-C*-algebras is a left $\DD(\HG)$-$*$-homomorphism if and only if 
	it is a left $H$- and $\hat{G}$-$*$-homomorphism via this correspondence. 
	
	We will make use of the notion of twisted tensor products in \cref{sec:MK} and \cref{sec:eg}. 
	We list its fundamental properties as follows. 
	All of them are contained in~\cite{Meyer-Roy-Woronowicz:twiten} and~\cite{Nest-Voigt:eqpd}, 
	or can be shown easily by using the techniques in them. 
	As for the notation of the categories in (3), see \cref{def:categories} and the remark after that. 
	
	\begin{prop}\label{prop:twiten}
		Let $\HG:H\to G$, $\GF:G\to F$, $\FE:F\to E$ be homomorphisms of regular locally compact quantum groups, and 
		$(A,\alpha,\hat{\alpha})$, $(B,\beta,\hat{\beta})$, $(C,\gamma,\hat{\gamma})$ 
		be left $\HG$-, $\GF$-, $\FE$-Yetter--Drinfeld C*-algebras, respectively. 
		Then the followings hold. 
		\begin{enumerate}
			\item
			$A\outensor{G}B:=\clin{\hat{\alpha}(A)_{12}\beta(B)_{13}}
			\subset \mult{A\otimes \K(G)\otimes B}$ is a non-degenerate C*-subalgebra. 
			Moreover, 
			\begin{align*}
				&
				(A\outensor{G}B, 
				\Ad W^{\HG}_{12}{}^* \sigma_{12}(\id\otimes\alpha\otimes\id), 
				\sigma_{12}\sigma_{23}\Ad W^{\GF}_{13}(\id\otimes\id\otimes \hat{\beta}))
			\end{align*}
			is a $\GF\HG$-Yetter--Drinfeld C*-algebra. 
			We call $A\outensor{G}B$ the \emph{twisted tensor product of $A$ and $B$}. 
			\item
			$\sigma_{12}\sigma_{23}\Ad W^{\GF}_{13}:
			A\outensor{G}(B\outensor{F}C) \xrarr{\sim}
			(A\outensor{G}B)\outensor{F}C$ 
			is a well-defined left $H$- and $\hat{E}$-$*$-isomorphism. 
			\item
			Let $\CC\in\{\Calg,\Cor,\KK\}$, where we additionally assume separability of $C^r_0(G)$ when $\CC=\KK$. 
			Then twisted tensor products with $A$ induces a well-defined functor 
			$A\outensor{G}-:\CC^{\DD(\GF)}\to\CC^{\DD(\GF\HG)}$. 
			Similarly we have the well-defined functor 
			$-\outensor{G}B:\CC^{\DD(\HG)}\to\CC^{\DD(\GF\HG)}$. 
			If $\CC=\KK$, these functors are triangulated. 
		$\sq$\end{enumerate} 
	\end{prop}
	
	\begin{eg}\label{eg:twiten}
		Let $\HG:H\to G$ be a homomorphism of locally compact quantum groups, and $(A,\alpha)$ be a left $G$-C*-algebra. 
		\begin{enumerate}
			\item
			We write $\C_\HG$ to be $\C$ endowed with the trivial structure of $\HG$-Yetter--Drinfeld C*-algebra. 
			Then $\alpha:\HG^*A\xrarr{\sim}\C_\HG\outensor{G}A$ 
			is a well-defined left $H$-$*$-isomorphism. 
			Similarly, for a left $\hat{H}$-C*-algebra $(B,\beta)$, we have the left $\hat{G}$-$*$-isomorphism 
			$\beta:\hat{\HG}^*B\xrarr{\sim} B\outensor{H}\C_{\HG}$. 
			\item
			Consider the $1_{\hat{G}^{\op}\to G}$-Yetter--Drinfeld C*-algebra 
			$(C^r_0(\hat{G}), \Delta_{\hat{G}}^{\cop}, \Delta_{\hat{G}})$. 
			Then 
			\begin{align*}
				&
				\Ad \hat{V}^{G}_{12}(-)_{13} : G\redltimes A\xrarr{\sim} C^r_0(\hat{G})\outensor{G}A
			\end{align*}
			is a well-defined left $\hat{G}^{\op}$-$*$-isomorphism, 
			if we regard the right $\hat{G}$-coaction of $G\redltimes A$ as a left $\hat{G}^{\op}$-coaction. 
			See Section~6.3 in~\cite{Meyer-Roy-Woronowicz:twiten} for detail. 
			When $G$ is regular, the Baaj--Skandalis duality~\cite{Baaj-Skandalis:mltu} says 
			$C^r_0(G^{\op})\outensor{\hat{G}^{\op}}C^r_0(\hat{G})\outensor{G}A
			\isom \hat{G}^{\op}\redltimes G\redltimes A\isom \K_A(L^2(G)\otimes A)$ naturally. 
			Here we have identified $C^r_0\bigl(\hat{\hat{G}^{\op}}\bigr)
			=J^GC^r_0(G)J^G\isom C^r_0(G^{\op})$ 
			via $\Ad (J^G\hat{J}^{G})$. 
		\end{enumerate}
	\end{eg}

	\providecommand{\bysame}{\leavevmode\hbox to3em{\hrulefill}\thinspace}

\end{document}